\documentclass{amsart}

\usepackage[T1]{fontenc}
\usepackage{amsfonts}
\usepackage{amsfonts,amsmath,amssymb}
\usepackage{mathrsfs,mathtools,stmaryrd,wasysym}
\usepackage{enumerate,enumitem}
\usepackage{esint}
\usepackage{graphicx,psfrag}
\usepackage{hyperref}
\usepackage{stackengine}
\usepackage{yfonts}
\usepackage{color}
\usepackage[linesnumbered,algoruled,boxed]{algorithm2e}

\usepackage{pgfplots}
\usepgfplotslibrary{groupplots}
\pgfplotsset{compat=newest}
\usepgfplotslibrary{external}
\usepgfplotslibrary{colorbrewer}

\newtheorem{theorem}{Theorem}[section]
\newtheorem{lemma}[theorem]{Lemma}

\theoremstyle{definition}

\theoremstyle{remark}
\newtheorem{remark}[theorem]{Remark}

\numberwithin{equation}{section}

\newcommand{\T}{\mathcal{T}}

\begin{document}

\title[Error estimates for a bang-bang optimal control problem]{A posteriori error estimates for a bang-bang optimal control problem}
\author{Francisco Fuica}
\address{Departamento de Matem\'atica y Ciencia de la Computaci\'on, Universidad de Santiago de Chile, Santiago, Chile.}
\email{francisco.fuica@usach.cl}
\thanks{The author is supported by ANID through FONDECYT postdoctoral project 3230126.}

\subjclass[2010]{Primary 
49M25,		   
65N15,         
65N30.         
}

\keywords{optimal control problems, bang-bang control, convergence, error estimates, a posteriori analysis.}

\date{}

\dedicatory{}

\begin{abstract}
We propose and analyze a posteriori error estimates for a control-constrained optimal control problem with bang-bang solutions. 
We consider a solution strategy based on the variational approach, where the control variable is not discretized; no Tikhonov regularization is made. 
We design, for the proposed scheme, a residual-type a posteriori error estimator that can be decomposed as the sum of two individual contributions related to the discretization of the state and adjoint equations.
We explore reliability and efficiency properties of the aforementioned error estimator.
We illustrate the theory with numerical examples.
\end{abstract}

\maketitle


\section{Introduction}\label{sec:intro}
The main purpose of this work is the design and analysis of a posteriori error estimates for a bang-bang optimal control problem governed by an elliptic partial differential equation (PDE) as state equation. 
This PDE-constrained optimization problem entails the minimization of a cost functional in which the cost of the control is negligible, so no Tikhonov regularization term is considered.
To make matters precise, let $\Omega\subset\mathbb{R}^d$ ($d\in \{2, 3\}$) be an open, bounded and convex polygonal/polyhedral domain with boundary $\partial\Omega$. 
Given a desired state $y_\Omega\in L^2(\Omega)$, we define the cost functional
\begin{align*}
J(u):=\frac{1}{2}\|y_{u} - y_{\Omega}\|_{\Omega}^{2}.
\end{align*} 
We will consider the following optimal control problem:
\begin{align}\label{def:opt_cont_prob}
\min J(u) ~ \text{ subject to } ~ -\Delta y_{u}  =  u  + f  \text{ in }  \Omega, \quad
y  =  0  \text{ on }  \partial\Omega,  ~ \text{ and } ~ u\in \mathbb{U}_{ad},
\end{align} 
where $f\in L^{2}(\Omega)$ denotes an external source and the set of admissible controls $\mathbb{U}_{ad}$ is given by $\mathbb{U}_{ad}:=\{v \in L^2(\Omega):  a \leq v(x) \leq b \text{ f.a.e.}~x \in \Omega \}$ with $a < b$.

One of the main challenges in deriving error estimates for problems within the setting of \eqref{def:opt_cont_prob} is that the cost function $J$ does not incorporate the standard Tikhonov regularization term $\alpha\|u\|_{\Omega}^{2}$ with $\alpha > 0$.
Note that without this term we cannot directly derive nor bound the standard error term $\alpha\|\bar{u} - \bar{u}_{h}\|_{\Omega}^{2}$, where $\bar{u}$ denotes the optimal control and $\bar{u}_{h}$ denotes a suitable approximation of $\bar{u}$. This has motivated the analysis of approximation techniques for such problems.
To the best of our knowledge, the first work that provides an a priori error analysis for problem \eqref{def:opt_cont_prob} is \cite{MR2891922}.
In such a work, the authors used the so-called variational approach, introduced in \cite{MR2122182}, in order to discretize problem \eqref{def:opt_cont_prob}.
In addition, they proved estimates for the approximation error associated with the optimal state and adjoint state without assuming that the control is of bang-bang type; see \cite[Lemma 2.1]{MR2891922}.
For the case when the optimal control is of bang-bang type, the authors proved an error estimate for all the optimal variables \cite[Theorem 2.2]{MR2891922}.
A suitable Tikhonov regularization of problem \eqref{def:opt_cont_prob} and its convergence were studied in \cite{MR3780469}, under an additional assumption on the structure of the optimal control.
The parabolic counterpart of \eqref{def:opt_cont_prob} was studied in \cite{MR4076464}, where the authors proved, using Petrov-Galerkin schemes in time and conforming finite elements in space, a priori estimates for the error between a discretized regularized problem and the limit problem.
In the particular case of bang-bang controls, the estimates were further improved; see \cite[Theorem 3.4]{MR4076464}. 
In \cite{MR3810878} the authors considered a Tikhonov regularization of problem \eqref{def:opt_cont_prob} with a semilinear elliptic PDE and derived a priori regularization error estimates for the control; a suitable extension to sparse optimal control problems was derived as well.
For a similar semilinear optimal control problem, a priori error estimates without regularization were studied \cite{MR4791221}.
In this work, the author derived error estimates for local minimizers that satisfy specific local growth conditions recently introduced in the context of solution stability, which involve the joint growth of the first and second variation of the objective functional.
Finally, for second-order analysis, stability, and approximation results for bang-bang optimal control problems, we refer the reader to \cite{MR2974742,MR3706910,MR3982675,MR4525177,CDJ2023,MR4744369}.

Among the different numerical methods that exist in the literature to approximate solutions to PDE-constrained optimization problems (and PDEs in general), a particular class stands out for its competitive performance, improving the quality of discrete approximations of the corresponding problem within a prescribed tolerance using a minimal amount of work. 
These are the adaptive finite element methods (AFEMs).
The main tools present in these iterative methods are a posteriori error estimates, which provide global and local information on the error of discrete solutions and that can be easily computed from the given numerical solution and problem data.
Regarding the use of these methods in the context of control--constrained linear--quadratic optimal control problems, several advances have been made in recent years.
For a discussion on this matter, we refer the interested reader to the following non-exhaustive list of works: \cite{MR1887737,MR1780911,MR2434065,MR3212590,MR3621827}.
As opposed to these advances, the analysis of AFEMs for bang-bang optimal control problems is rather scarce. 
To the best of our knowledge, the work \cite{MR3095657} appears to be the first one that provides a posteriori error estimates for problem \eqref{def:opt_cont_prob}. 
In this article, the author investigated Tikhonov regularization and discretization of bang-bang control problems, developing a parameter choice rule that adaptively selects the Tikhonov regularization parameter depending on a posteriori computable quantities. 
However, the error estimates were not robust with respect to $\alpha$.
We also mention the work \cite{MR3385653}, where robust global reliability estimates were provided.
We note that no efficiency estimates were provided in \cite{MR3095657,MR3385653}.

In the present manuscript, we consider the variational discretization \cite{MR2122182} to approximate the optimal control problem \eqref{def:opt_cont_prob}.
In particular, we use piecewise linear functions to approximate the solution of the state and adjoint equations whereas the admissible control set is not discretized.
To perform the analysis, we follow \cite{MR2891922} and do not consider a Tikhonov regularization.
This approach allows us to circumvent the necessity of choosing  a suitable regularization parameter  for each mesh, cf. \cite{MR3385653}.
 Within this framework, we devise a residual--based a posteriori error estimator that is formed by only two contributions that are related to the discretization of the state and adjoint equations.
In two- and three-dimensional convex polygonal/polyhedral domains, we prove efficiency estimates; reliability properties of the a posteriori error estimator are studied as well.
More precisely, we prove that the corresponding local error indicators associated to the discretization of the state and adjoint equations are \emph{locally efficient}; see Theorem \ref{thm:efficiency_est_ocp}.
We recall that reliability estimates are sufficient to obtain a numerical solution within a prescribed tolerance, whereas efficiency estimates are important as they ensure that the mesh is correctly refined, providing a numerical solution using a (nearly) minimal number of degrees of freedom \cite{MR3059294}.
Based on the proposed a posteriori error estimator, we design a simple adaptive loop that delivers optimal experimental rates of convergence for all the involved individual contributions of the corresponding error. 
In particular, the aforementioned loop delivers quadratic rates of convergence for the approximation error associated to all the optimal variables. 
In addition, and in contrast to \cite[Section 5.2]{MR3385653}, the error indicator that we consider for the adjoint variable in the $L^{\infty}(\Omega)$--norm allows for unbounded forcing terms. 
This is of importance since, as observed from \eqref{eq:discrete_adj_eq}, the discrete adjoint equation has $\bar{y}_{\ell} - y_\Omega$ as a forcing term and, in general, the latter does not necessarily belong to $L^{\infty}(\Omega)$. 

The rest of the manuscript is organized as follows. 
The remaining of this section is devoted to introduce the notation that we will use throughout the manuscript.
In section \ref{sec:ocp} we present a weak formulation for the optimal control problem under consideration and introduce a finite element discretization scheme.
The main part of the paper is section \ref{sec:a_post}, where we design an a posteriori error estimator for the proposed approximation scheme and analyze reliability and efficiency properties. 
Finally, in section \ref{sec:num_ex}, we present a series of two-dimensional numerical examples that illustrate the theory and reveal a competitive performance of the devised AFEMs.

%
%


\subsection{Notation}\label{sec:notation}

Throughout this work, we use standard notation for Lebesgue and Sobolev spaces and their corresponding norms. 
Given an open and bounded domain $G$, we denote by $(\cdot,\cdot)_{G}$ and $\| \cdot \|_{G}$ the inner product and norm of $L^{2}(G)$, respectively.
If $\mathcal{X}$ and $\mathcal{Y}$ are Banach function spaces, we write $\mathcal{X}\hookrightarrow \mathcal{Y}$ to denote that $\mathcal{X}$ is continuously embedded in $\mathcal{Y}$. 
The relation $\mathfrak{a} \lesssim \mathfrak{b}$ indicates that $\mathfrak{a} \leq C \mathfrak{b}$, with a positive constant that depends neither on $\mathfrak{a}$, $\mathfrak{b}$ nor on the discretization parameters. 
The value of $C$ might change at each occurrence.


\section{The optimal control problem}
\label{sec:ocp}

In this section, we briefly present a weak formulation for problem \eqref{def:opt_cont_prob} and recall first-order optimality conditions. 
We also introduce a finite element discretization scheme.

\subsection{Weak formulation}

We consider the following weak version of problem \eqref{def:opt_cont_prob}: Find
\begin{align}\label{def:weak_ocp}
\min \{ J(u): u \in \mathbb{U}_{ad}\}
\end{align}
subject to 
\begin{align}\label{eq:weak_st_eq}
y_{u}\in H_0^{1}(\Omega) ~:~ (\nabla y_{u},\nabla v)_{\Omega}=(u + f,v)_{\Omega} \quad \forall v \in H_0^1(\Omega).
\end{align}
Problem \eqref{def:weak_ocp}--\eqref{eq:weak_st_eq} admits a unique optimal solution $\bar{u}\in\mathbb{U}_{ad}$ \cite[Theorem 2.14]{MR2583281} and the optimal control $\bar{u}$ satisfies the (sufficient) first-order optimality condition \cite[Theorem 1.1]{MR2891922} 
\begin{align}\label{eq:var_ineq}
(\bar{p}, u - \bar{u})_{\Omega} \geq 0 \quad \forall u \in \mathbb{U}_{ad},
\end{align}
where the optimal \emph{adjoint state} $\bar{p}\in H_0^1(\Omega)$ solves the following adjoint equation:
\begin{align*}
(\nabla v,\nabla \bar{p})_{\Omega}=(\bar{y} - y_{\Omega},v)_{\Omega} \quad \forall v \in H_0^1(\Omega),
\end{align*}
with $\bar{y} := y_{\bar{u}}$.
Moreover, inequality \eqref{eq:var_ineq} implies that, for a.e. $x\in\Omega$, we have 
\begin{align*}
\bar{u}(x) = a ~ \text{ if } ~  \bar{p}(x) > 0, \quad \bar{u}(x) \in [a,b] ~ \text{ if } ~  \bar{p}(x) = 0, \quad \bar{u}(x) = b  ~ \text{ if } ~  \bar{p}(x) < 0;
\end{align*}
see \cite[Remark 1.2]{MR2891922}.

\subsection{Finite element approximation}

Let us introduce some ingredients of standard finite element approximations \cite{MR2050138,MR2373954}. 
Let $\mathcal{T} = \{T\}$ be a conforming partition of $\overline{\Omega}$ into simplices $T$ with size $h_T := \textrm{diam}(T)$. 
Let us denote by $\mathbb{T}$ the collection of conforming and shape regular meshes that are refinements of $\mathcal{T}_0$, where $\mathcal{T}_0$ represents an initial mesh. 
Given a mesh $\mathcal{T}_{\ell} \in \mathbb{T}$ with $\ell\in\mathbb{N}_{0}$, we denote by $\mathcal{E}_{\ell}$ the set of \emph{internal} $(d-1)$-dimensional interelement boundaries $e$ of $\mathcal{T}_{\ell}$. 
For $T \in \mathcal{T}_{\ell}$, we let $\mathcal{E}_T$ denote the subset of $\mathcal{E}_{\ell}$ which contains the sides of the element $T$. 

Given a mesh $\mathcal{T}_{\ell} \in \mathbb{T}$ with $\ell\in\mathbb{N}_{0}$, we define the finite element space of continuous piecewise polynomials of degree one that vanish on the boundary as
\begin{align*}
\mathbb{V}_{\ell} = \{v_{\ell}\in C(\overline{\Omega}): v_{\ell}|_{T}\in \mathbb{P}_1(T) ~ \forall T\in \mathcal{T}_{\ell} \} \cap H_0^{1}(\Omega).
\end{align*}
Given $v_{\ell} \in \mathbb{V}_{\ell}$, we define, for any internal side $e \in \mathcal{E}_{\ell}$, the jump or interelement residual $\llbracket \nabla v_{\ell}\cdot \mathbf{n} \rrbracket$ on $e$ by
\[
\llbracket \nabla v_{\ell}\cdot \mathbf{n} \rrbracket|_{e}:= \mathbf{n}^{+} \cdot \nabla v_{\ell}|^{}_{T^{+}} + \mathbf{n}^{-} \cdot \nabla v_{\ell}|^{}_{T^{-}},
\]
where $\mathbf{n}^{\pm}$ denote the unit exterior normal to the element $T^{\pm}$. Here, $T^{+}$, $T^{-} \in \T_{\ell}$ are such that $T^{+} \neq T^{-}$ and $\partial T^{+} \cap \partial T^{-} = e$.

We consider the following semi-discrete version of the optimal control problem \eqref{def:weak_ocp}--\eqref{eq:weak_st_eq}: Find 
\begin{align}\label{eq:discrete_opc_semi}
\min \{ J(\mathfrak{u}): \mathfrak{u} \in \mathbb{U}_{ad}\}
\end{align}
subject to the discrete state equation
\begin{align}\label{eq:discrete_pde_semi}
(\nabla y_{\ell}(\mathfrak{u}),\nabla v_\ell)_{\Omega}
=
(\mathfrak{u} + f, v_\ell)_{\Omega} 
\quad \forall v_\ell\in \mathbb{V}_{\ell}.
\end{align}
Existence of an optimal control $\bar{\mathfrak{u}}$ for \eqref{eq:discrete_opc_semi}--\eqref{eq:discrete_pde_semi} can be proved by standard arguments. 
However, uniqueness of $\bar{\mathfrak{u}}$ is not guaranteed \cite[Remark 3.1]{MR4076464}. 
Despite this fact, we can characterize optimal solutions as in the continuous case: every optimal control $\bar{\mathfrak{u}}$ for \eqref{eq:discrete_opc_semi}--\eqref{eq:discrete_pde_semi} satisfies 
\begin{align}\label{eq:semi_var_ineq}
(\bar{p}_{\ell}, u - \bar{\mathfrak{u}})_{\Omega} \geq 0 \quad \forall u \in \mathbb{U}_{ad},
\end{align}
where $\bar{p}_{\ell}\in \mathbb{V}_{\ell}$ solves the discrete adjoint equation
\begin{align}\label{eq:discrete_adj_eq}
(\nabla v_{\ell},\nabla \bar{p}_{\ell})_{\Omega}=(\bar{y}_{\ell} - y_{\Omega},v_{\ell})_{\Omega} \quad \forall v_{\ell} \in \mathbb{V}_{\ell}.
\end{align}
Here, $\bar{y}_{\ell}$ solves problem \eqref{eq:discrete_pde_semi} with $\mathfrak{u} = \bar{\mathfrak{u}}$. 
We immediately notice that $\bar{y}_{\ell}$ and $\bar{p}_{\ell}$ are unique, even if $\bar{\mathfrak{u}}$ is not unique; see \cite[Remark 3.1]{MR4076464}.

As in the continuous case, from \eqref{eq:semi_var_ineq} it stems the following characterization for optimal controls $\bar{\mathfrak{u}}$, for a.e. $x\in \Omega$: 
\begin{equation}\label{eq:pointwise_charac_discrete}
\bar{\mathfrak{u}}(x) = a ~ \text{ if } ~  \bar{p}_{\ell}(x) > 0, \quad \bar{\mathfrak{u}}(x) \in [a,b] ~ \text{ if } ~  \bar{p}_{\ell}(x) = 0, \quad \bar{\mathfrak{u}}(x) = b  ~ \text{ if } ~  \bar{p}_{\ell}(x) < 0. \hspace{-0.2cm}
\end{equation}
From \eqref{eq:pointwise_charac_discrete} it follows that, if $\bar{p}_{\ell}$ admits a zero level set of measure $0$, then $\bar{\mathfrak{u}}(x)=a$ or $\bar{\mathfrak{u}}(x)=b$ for a.e. $x\in \Omega$ and thus $\bar{\mathfrak{u}}$ is both unique and of bang-bang type.

Finally, since $\bar{\mathfrak{u}}$ implicitly depends on $\mathcal{T}_\ell$, in what follows we shall adopt the notation $\bar{\mathfrak{u}}_{\ell}$.


\section{A posteriori error estimates}\label{sec:a_post}

In the present section, we design an error estimator for the optimal control problem \eqref{eq:discrete_opc_semi}--\eqref{eq:discrete_pde_semi}.
We explore its reliability properties in section \ref{sec:reliability}, whereas its local efficiency is proved in section \ref{sec:efficiency}.


\subsection{Reliability}\label{sec:reliability}
The upcoming analysis mainly hinges on approximations of the error between a solution to the semi-discrete optimal control problem and suitable auxiliary variables that we shall define in what follows.


\subsubsection{Auxiliary upper bounds}\label{sec:aux}

Let $\bar{\mathfrak{u}}_{\ell}$ be a solution of the semi-discrete optimal control problem associated to a mesh $\mathcal{T}_{\ell}\in \mathbb{T}$. We introduce the auxiliary variable $y_{\bar{\mathfrak{u}}_{\ell}}\in H_0^{1}(\Omega)$, defined as the unique solution to 
\begin{align}\label{eq:aux_y_ul}
(\nabla y_{\bar{\mathfrak{u}}_{\ell}},\nabla v)_{\Omega}
=
(\bar{\mathfrak{u}}_{\ell} + f, v)_{\Omega} \quad \forall v \in H_0^{1}(\Omega).
\end{align}
We note that the discrete optimal state $\bar{y}_{\ell}$ corresponds to the finite element approximation of $y_{\bar{\mathfrak{u}}_{\ell}}$ in $\mathbb{V}_{\ell}$. 
Hence, since we have assumed that $\Omega$ is convex, we can use the results from \cite[Section 2.4]{MR1885308} to obtain that
\begin{equation}\label{eq:estimate_haty-yh}
\|y_{\bar{\mathfrak{u}}_{\ell}} - \bar{y}_{\ell}\|_{\Omega} \lesssim {\eta}_{st,2},
\end{equation}
where the error estimator $\eta_{st,2}$ and its local error indicators are defined by
\begin{align}\label{def:state_indicator}
{\eta}_{st,2}^{2}:=\sum_{T\in\T_{\ell}} {E}_{st,T}^{2}, 
\qquad
E_{st,T}^{2}:=h_{T}^{4}\|\bar{\mathfrak{u}}_{\ell} + f\|_{T}^{2} + \sum_{e\in\mathcal{E}_{T}}h_{T}^{3}\|\llbracket \nabla \bar{y}_{\ell} \cdot \mathbf{n}\rrbracket|_{e}\|_{e}^{2}.
\end{align}

We define $p_{\bar{y}_{\ell}} \in H_0^{1}(\Omega)$ as the unique solution to 
\begin{equation}\label{eq:aux_p}
(\nabla v, \nabla p_{\bar{y}_{\ell}})_{\Omega}
=
(\bar{y}_{\ell} - y_{\Omega}, v)_{\Omega} \quad \forall v \in H_0^{1}(\Omega),
\end{equation}
and immediately note that $\bar{p}_{\ell}$ corresponds to the finite element approximation of $p_{\bar{y}_{\ell}}$ in $\mathbb{V}_{\ell}$. 

We introduce, for each $T\in\T_{\ell}$, the following a posteriori local error indicators 
\begin{align}\label{def:pointwise_estimator}
E_{adj,\infty,T}:= h_T^{2-\frac{d}{2}} \| \bar{y}_{\ell} - y_{\Omega}\|_{T} + h_T\max_{e\in\mathcal{E}_{T}}\|\llbracket\nabla \bar{p}_{\ell}\cdot \mathbf{n}\rrbracket|_{e}\|_{L^\infty(e)},
\end{align}
and the error estimator ${\eta}_{adj,\infty}:=\max_{T\in \mathcal{T}_{\ell}}E_{adj,\infty,T}$. The proof of the following reliability estimate can be found in \cite[Lemma 4.2]{MR3878607}:
\begin{align}\label{eq:estimate_hatp_inf}
\|p_{\bar{y}_{\ell}} - \bar{p}_{\ell}\|_{L^{\infty}(\Omega)} \lesssim \iota_{\ell}{\eta}_{adj,\infty},
\end{align}
where the term $\iota_{\ell}$ is defined by
\begin{equation*}
\iota_{\ell}:= \left|\log\bigg(\max_{T\in\mathcal{T}_{\ell}}\frac{1}{h_T}\bigg)\right|.
\end{equation*}
An important feature of the local error indicator in \eqref{def:pointwise_estimator} is that it incorporates the $L^2(T)$-norm of the element residual instead of the $L^{\infty}(T)$-norm, which is the common consideration in the literature; see, e.g., \cite{MR1270622,MR1740762,MR3520007,MR4648515}.
 In particular, the error estimator ${\eta}_{adj,\infty}$ allows for a pointwise a posteriori error analysis with unbounded right-hand sides, cf. \cite{MR3407259}. 
 This is of importance since, as it can be observed in \eqref{def:pointwise_estimator}, its indicators contain the term  $\bar{y}_{\ell} - y_{\Omega}$ which is not necessarily bounded in $L^{\infty}(T)$.


\subsubsection{Reliability estimates}\label{sec:rel}

In what follows we shall assume a standard structural assumption on the adjoint state related to $\bar{u}$ (cf. \cite[eq. (2.10)]{MR2891922} and \cite[eq. (6.1)]{MR4298694})
\begin{equation}\label{eq:assumption_S}
\exists \beta \in (0,1], ~ \exists \mathfrak{C} > 0, ~ \forall \varepsilon > 0 \quad \text{such that} \quad |\{x\in \Omega : |\bar{p}(x)|\leq \varepsilon\}|\leq \mathfrak{C}\varepsilon^{\beta},
\end{equation}
where, given a set $A\subset \Omega$, $|A|$ denotes the Lebesgue measure of it.
We immediately mention that a control $\bar{u}$ satisfying first-order optimality conditions and assumption \eqref{eq:assumption_S} is of bang-bang type.

\begin{theorem}[reliability estimates]\label{thm:global_rel_}
Let $\bar{u}\in \mathbb{U}_{ad}$ be the unique solution to problem \eqref{def:weak_ocp}--\eqref{eq:weak_st_eq} with $\bar{y}$ and $\bar{p}$ being the corresponding state and adjoint state variables, respectively. Let $\bar{\mathfrak{u}}_{\ell}\in \mathbb{U}_{ad}$ be a solution to the semi-discrete problem with $\bar{y}_{\ell}$ and $\bar{p}_{\ell}$ being the corresponding discrete state and discrete adjoint state variables, respectively. If assumption \eqref{eq:assumption_S} holds, then
\begin{align*}
\|\bar{u} - \bar{\mathfrak{u}}_{\ell}\|_{L^{1}(\Omega)}
\lesssim 
(\iota_{\ell}\eta_{adj,\infty} + 
\eta_{st,2})^{\beta},
\end{align*}
and
\begin{align*}
\|\bar{y} - \bar{y}_{\ell}\|_{\Omega} + \|\bar{p} - \bar{p}_{\ell}\|_{L^{\infty}(\Omega)}
\lesssim  
\iota_{\ell}\eta_{adj,\infty} + 
\eta_{st,2} + (\iota_{\ell}\eta_{adj,\infty} + 
\eta_{st,2})^{\beta}.
\end{align*}
The hidden constants are independent of the continuous and discrete optimal variables, the size of the elements in the mesh $\mathcal{T}_{\ell}$, and $\#\mathcal{T}_{\ell}$.
\end{theorem}
\begin{proof}
We proceed in three steps.

\underline{Step 1}.  (estimation of $\|\bar{u} - \bar{\mathfrak{u}}_{\ell}\|_{L^1(\Omega)}$)
Assumption \eqref{eq:assumption_S} and inequality \eqref{eq:var_ineq} with $u=\bar{\mathfrak{u}}_{\ell}$ imply that \cite[Lemma 6.3]{MR3810878} (see also \cite[Lemma 6]{MR4298694})
\begin{align}\label{eq:estimate_u-ul_beta}
\|\bar{u} - \bar{\mathfrak{u}}_{\ell}\|_{L^{1}(\Omega)}^{1 + \frac{1}{\beta}}
\lesssim
(\bar{p}, \bar{\mathfrak{u}}_{\ell} - \bar{u})_{\Omega}.
\end{align}
Choosing $u=\bar{u}$ in \eqref{eq:semi_var_ineq} and using the obtained inequality in \eqref{eq:estimate_u-ul_beta} we arrive at
\begin{align*}
\|\bar{u} - \bar{\mathfrak{u}}_{\ell}\|_{L^{1}(\Omega)}^{1 + \frac{1}{\beta}}
\lesssim
(\bar{p} - \bar{p}_{\ell}, \bar{\mathfrak{u}}_{\ell} - \bar{u})_{\Omega}.
\end{align*}
We invoke $p_{\bar{y}_{\ell}}\in H_0^{1}(\Omega)$, solution to \eqref{eq:aux_p}, and write
\begin{equation}\label{eq:estimate_I_II}
\|\bar{u} - \bar{\mathfrak{u}}_{\ell}\|_{L^{1}(\Omega)}^{1 + \frac{1}{\beta}}
\lesssim 
(p_{\bar{y}_{\ell}} - \bar{p}_{\ell}, \bar{\mathfrak{u}}_{\ell} - \bar{u})_{\Omega} +
(\bar{p} - p_{\bar{y}_{\ell}}, \bar{\mathfrak{u}}_{\ell} - \bar{u})_{\Omega}
= \mathbf{I} + \mathbf{II}.
\end{equation}
We now bound the terms $\mathbf{I}$ and $\mathbf{II}$ in \eqref{eq:estimate_I_II}. 
To estimate $\mathbf{I}$, we use the a posteriori error estimate \eqref{eq:estimate_hatp_inf}:
\begin{align}\label{eq:estimation_I_i}
\mathbf{I} 
\leq 
\|p_{\bar{y}_{\ell}} - \bar{p}_{\ell}\|_{L^{\infty}(\Omega)}\|\bar{\mathfrak{u}}_{\ell} - \bar{u}\|_{L^{1}(\Omega)}
\lesssim
\iota_{\ell}\eta_{adj,\infty}\|\bar{\mathfrak{u}}_{\ell} - \bar{u}\|_{L^{1}(\Omega)}.
\end{align}
To estimate $\mathbf{II}$, we note that $y_{\bar{\mathfrak{u}}_{\ell}} - \bar{y}\in H_0^1(\Omega)$ solves
\begin{equation}\label{eq:problem_haty-y}
(\nabla (y_{\bar{\mathfrak{u}}_{\ell}} - \bar{y}), \nabla v)_{\Omega}  = (\bar{\mathfrak{u}}_{\ell} - \bar{u}, v)_{\Omega} \quad \forall v\in H_0^{1}(\Omega),
\end{equation}
and that $\bar{p} - p_{\bar{y}_{\ell}}\in H_0^1(\Omega)$ solves
\begin{align}\label{eq:problem_hatp-p}
(\nabla v, \nabla (\bar{p} - p_{\bar{y}_{\ell}}))_{\Omega}  = (\bar{y} - \bar{y}_{\ell}, v)_{\Omega} \quad \forall v\in H_0^{1}(\Omega).
\end{align}
Hence, replacing $v = \bar{p} - p_{\bar{y}_{\ell}}$ in \eqref{eq:problem_haty-y} and $v = y_{\bar{\mathfrak{u}}_{\ell}} - \bar{y}$ in \eqref{eq:problem_hatp-p} we obtain the identity $(\bar{\mathfrak{u}}_{\ell} - \bar{u}, \bar{p} - p_{\bar{y}_{\ell}})_{\Omega} =  (\bar{y} - \bar{y}_{\ell},  y_{\bar{\mathfrak{u}}_{\ell}} - \bar{y})_{\Omega}$, which, in turns, yields 
\begin{align*}
\mathbf{II} 
=
-\|y_{\bar{\mathfrak{u}}_{\ell}} - \bar{y}\|_{\Omega}^{2} + (y_{\bar{\mathfrak{u}}_{\ell}} - \bar{y}_{\ell},  y_{\bar{\mathfrak{u}}_{\ell}} - \bar{y})_{\Omega} 
\leq 
 (y_{\bar{\mathfrak{u}}_{\ell}} - \bar{y}_{\ell},  y_{\bar{\mathfrak{u}}_{\ell}} - \bar{y})_{\Omega}.
\end{align*}
The latter, in light of estimate \eqref{eq:estimate_haty-yh}, allows us to obtain $\mathbf{II} \leq \eta_{st,2} \|y_{\bar{\mathfrak{u}}_{\ell}} - \bar{y}\|_{\Omega}$.
The term $ \|y_{\bar{\mathfrak{u}}_{\ell}} - \bar{y}\|_{\Omega}$ is bounded in view of the stability estimate $\|y_{\bar{\mathfrak{u}}_{\ell}} - \bar{y}\|_{\Omega} \lesssim \|\bar{\mathfrak{u}}_{\ell} - \bar{u}\|_{L^{1}(\Omega)}$ (see \cite[Lemma 2.3]{CDJ2023}).
Therefore, we have proved that 
\begin{equation}\label{eq:estimation_II}
\mathbf{II}
\leq 
\eta_{st,2} \|\bar{\mathfrak{u}}_{\ell} - \bar{u}\|_{L^{1}(\Omega)}.
\end{equation}
Finally, combining estimates \eqref{eq:estimate_I_II}, \eqref{eq:estimation_I_i}, and \eqref{eq:estimation_II} we conclude that 
\begin{align}\label{eq:estimate_error_control}
\|\bar{u} - \bar{\mathfrak{u}}_{\ell}\|_{L^{1}(\Omega)}
\lesssim 
(\iota_{\ell}\eta_{adj,\infty} + 
\eta_{st,2})^{\beta}.
\end{align}

\underline{Step 2}. (estimation of $\|\bar{y} - \bar{y}_{\ell}\|_{\Omega}$) The use of triangle inequality, a posteriori error estimate \eqref{eq:estimate_haty-yh}, and the stability estimate $\|\bar{y} - y_{\bar{\mathfrak{u}}_{\ell}}\|_{\Omega}  \lesssim \|\bar{u} - \bar{\mathfrak{u}}_{\ell}\|_{L^{1}(\Omega)}$ results in 
\begin{equation*}
\|\bar{y} - \bar{y}_{\ell}\|_{\Omega}
\leq
\|\bar{y} -  y_{\bar{\mathfrak{u}}_{\ell}}\|_{\Omega} + \|y_{\bar{\mathfrak{u}}_{\ell}} - \bar{y}_{\ell}\|_{\Omega}
\lesssim
\|\bar{u} - \bar{\mathfrak{u}}_{\ell}\|_{L^{1}(\Omega)} + \eta_{st,2}.
\end{equation*}
We conclude by using the bound \eqref{eq:estimate_error_control}.

\underline{Step 3}. (estimation of $\|\bar{p} - \bar{p}_{\ell}\|_{L^{\infty}(\Omega)}$) 
Triangle inequality and the a posteriori error estimate \eqref{eq:estimate_hatp_inf} give
\begin{equation*}
\|\bar{p} - \bar{p}_{\ell}\|_{L^{\infty}(\Omega)}
\leq
\|\bar{p} - p_{\bar{y}_{\ell}}\|_{L^{\infty}(\Omega)} + \|p_{\bar{y}_{\ell}} - \bar{p}_{\ell}\|_{L^{\infty}(\Omega)}
\lesssim
\|\bar{p} - p_{\bar{y}_{\ell}}\|_{L^{\infty}(\Omega)} + \iota_{\ell} {\eta}_{adj,\infty}.
\end{equation*}
The convexity of $\Omega$, the continuous embedding $H^{2}(\Omega) \hookrightarrow C(\overline{\Omega})$, and the stability of problem \eqref{eq:problem_hatp-p} yield the bound $\|\bar{p} - p_{\bar{y}_{\ell}}\|_{L^{\infty}(\Omega)}\lesssim  \|\bar{p} - p_{\bar{y}_{\ell}}\|_{H^{2}(\Omega)} \lesssim \|\bar{y} - \bar{y}_{\ell}\|_{\Omega}$.
This, in view of the estimate derived in Step 2, gives as a results
\begin{align*}
\|\bar{p} - \bar{p}_{\ell}\|_{L^{\infty}(\Omega)}
\lesssim
\iota_{\ell}\eta_{adj,\infty} + 
\eta_{st,2} + (\iota_{\ell}\eta_{adj,\infty} + 
\eta_{st,2})^{\beta}.
\end{align*}
This concludes the proof.
\end{proof}

\begin{remark}[case $\beta=1$]\label{rmk:beta=1}
When $\beta=1$ in \eqref{eq:assumption_S}, we obtain the following simplified upper bound for the total approximation error:
\begin{equation*}
\|\bar{u} - \bar{\mathfrak{u}}_{\ell}\|_{L^1(\Omega)} + \|\bar{y} - \bar{y}_{\ell}\|_{\Omega} + \|\bar{p} - \bar{p}_{\ell}\|_{L^{\infty}(\Omega)}
\lesssim  
{\eta}_{st,2} +  \iota_{\ell} {\eta}_{adj,\infty}.
\end{equation*}
A sufficient condition to ensure that assumption \eqref{eq:assumption_S} is fulfilled with $\beta=1$ was given in \cite[Lemma 3.2]{MR2891922}.
\end{remark}

   
\subsection{Efficiency}\label{sec:efficiency}
In this section, we study efficiency properties for the local a posteriori error estimators ${\eta}_{st,2}$ and ${\eta}_{adj,\infty}$, defined in section \ref{sec:aux}. 
Before proceeding with the analysis, we introduce some notation: for an edge, triangle or tetrahedron $G$, let $\mathcal{V}(G)$ be the set of vertices of $G$. 
Given $T \in \T_{\ell}$ and $e \in \mathcal{E}_T$, we denote by $\mathcal{N}_e \subset \mathcal{T}_{\ell}$ the subset that contains the two elements that have $e$ as a side, namely, $\mathcal{N}_e=\{T^+,T^-\}$, where $T^+, T^- \in \mathcal{T}_{\ell}$ are such that $e = T^+ \cap T^-$. 
For $T \in \mathcal{T}_{\ell}$, we define the \emph{star} associated with the element $T$ as
\begin{equation}\label{def:patch}
\mathcal{N}_T:= \left \{ T^{\prime}\in\mathcal{T}_{\ell}: \mathcal{E}_{T}\cap \mathcal{E}_{T^\prime}\neq\emptyset \right \}.
\end{equation}
In an abuse of notation, below we denote by $\mathcal{N}_T$ either the set itself or the union of its elements.

Let $T \in \T_{\ell}$ and $e \in \mathcal{E}_T$. 
We introduce the standard \emph{interior} and \emph{edge} bubble functions $\varphi_{T}$ and $\varphi_{e}$, respectively; see, e.g., \cite[Section 2.3.1]{MR1885308}. 
We also introduce the following facet bubble function 
\begin{equation*}
\psi_{e}|_{\mathcal{N}_{e}}=d^{4d}
\left(\prod_{\texttt{v}\in\mathcal{V}(e)} \phi_{\texttt{v}}^{T^+} \phi_{\texttt{v}}^{T^-}\right)^{2},
\end{equation*}
where, for $\texttt{v} \in \mathcal{V}(e)$, $\phi_{\texttt{v}}^{T^{\pm}}$ denotes the barycentric coordinates of $T^\pm$, respectively, which are understood as functions over $\mathcal{N}_{e}$, i.e., $\phi_{\texttt{v}}^{T^{\pm}}$ are extended to an affine function on $\mathcal{N}_{e}$ that vanishes in $\Omega\setminus\mathcal{N}_{e}$.
The bubble function $\psi_e$ has the following properties: $\psi_{e} \in \mathbb{P}_{4d}(\mathcal{N}_{e})$, $\psi_{e} \in C^2(\mathcal{N}_{e})$, and $\psi_{e} = 0$ on $\partial \mathcal{N}_{e}$. 
In addition, it satisfies
\begin{equation*}
 \nabla \psi_{e} = 0 \textrm{ on } \partial \mathcal{N}_{e}, \quad 
\llbracket \nabla \psi_e\cdot\mathbf{n}\rrbracket|_{e} = 0 \textrm{ on } e.
\end{equation*}

Given $T\in\mathcal{T}_{\ell}$, let $\Pi_{T} : L^2(T) \to \mathbb{P}_{0}(T)$ be the orthogonal projection operator into constant functions over $T$. In particular, for $v\in L^{2}(\Omega)$, we have $\Pi_{T}v:=\tfrac{1}{|T|}\int_{T}v$ and $\|\Pi_T v\|_{T}\leq \|v\|_T$.

With all these ingredients at hand, we are ready to prove the local efficiency of the aforementioned error estimators. We start with ${\eta}_{st,2}$, defined in \eqref{def:state_indicator}.

\begin{lemma}[local efficiency of ${\eta}_{st,2}$]\label{lemma:efficiency_est_2}
Let $\bar{u}\in \mathbb{U}_{ad}$ be the unique solution to problem \eqref{def:weak_ocp}--\eqref{eq:weak_st_eq} with $\bar{y}$ being the corresponding optimal state. Let $\bar{\mathfrak{u}}_{\ell}\in \mathbb{U}_{ad}$ be a solution to the semi-discrete problem with $\bar{y}_{\ell}$ being the corresponding discrete state variable. 
Then, for $T\in\mathcal{T}_{\ell}$, the local error indicator $E_{st,T}$, defined as in \eqref{def:state_indicator}, satisfies
\begin{align*}
E_{st,T}^2
\lesssim 
\|\bar{y} - \bar{y}_{\ell}\|_{\mathcal{N}_T}^2
+
h_{T}^{4 - d}\|\bar{u}-\bar{\mathfrak{u}}_{\ell}\|_{L^1(\mathcal{N}_T)}^2
+
\sum_{T'\in\mathcal{N}_{T}}h_{T}^{4}\|(\bar{\mathfrak{u}}_{\ell} + f) - \Pi_{T}(\bar{\mathfrak{u}}_{\ell} + f )\|_{T'}^{2},
\end{align*}
where $\mathcal{N}_T$ is defined as in \eqref{def:patch} and the hidden constant is independent of continuous and discrete optimal variables, the size of the elements in the mesh $\mathcal{T}_{\ell}$, and $\#\mathcal{T}_{\ell}$.
\end{lemma}
\begin{proof} Let $v \in H_0^{1}(\Omega)$ be such that $v|_T\in C^2(T)$ for all $T\in \mathcal{T}_{\ell}$. Use $v$ as a test function in \eqref{eq:weak_st_eq} and apply elementwise integration by parts to obtain
\begin{equation*}
( \nabla (\bar{y} - \bar{y}_{\ell}), \nabla v)_{\Omega} -  (\bar{u} - \bar{\mathfrak{u}}_{\ell},v)_{\Omega} 
=
\sum_{T\in \T_{\ell}} ( \bar{\mathfrak{u}}_{\ell} + f, v )_{T}  - \sum_{e\in\mathcal{E}_{\ell}}(\llbracket \nabla \bar{y}_{\ell}\cdot \mathbf{n} \rrbracket|_{e},v)_{e}.
\end{equation*}
At the same time, we integrate by parts, again, to arrive at
\begin{equation*}
 (\nabla (\bar{y} - \bar{y}_{\ell}), \nabla v)_{\Omega} = \sum_{e\in\mathcal{E}_{\ell}} (\llbracket \nabla v\cdot \mathbf{n} \rrbracket|_{e}, \bar{y} - \bar{y}_{\ell})_{e} - \sum_{T\in \T_{\ell}} (\bar{y} - \bar{y}_\ell, \Delta v)_{T}.
\end{equation*}
Hence, combining the two previous identities we obtain, for all $v \in H_0^{1}(\Omega)$  such that $v|_T\in C^2(T)$ for all $T\in \mathcal{T}_{\ell}$, the equality
\begin{align}\label{eq:error_eq_st}
& \sum_{e\in\mathcal{E}_{\ell}} (\llbracket \nabla v\cdot \mathbf{n} \rrbracket|_{e}, \bar{y} - \bar{y}_\ell)_{e}  - \sum_{T\in \T_{\ell}} (\bar{y} - \bar{y}_\ell, \Delta v)_{T}  -  (\bar{u} - \bar{\mathfrak{u}}_{\ell},v)_{\Omega} \\
 =
& \sum_{T\in \T_{\ell}} \left[ (\Pi_{T}(\bar{\mathfrak{u}}_{\ell} + f), v )_{T} + ((\bar{\mathfrak{u}}_{\ell}  + f) - \Pi_{T}(\bar{\mathfrak{u}}_{\ell} + f), v )_{T} \right]  - \sum_{e\in\mathcal{E}_{\ell}} ( \llbracket \nabla \bar{y}_\ell\cdot \mathbf{n} \rrbracket|_{e}, v)_{e}. \nonumber
\end{align}

We now proceed in two steps.

\underline{Step 1.} (estimation of $h_{T}^{4}\|\bar{\mathfrak{u}}_{\ell} + f\|_{T}^{2}$) Let $T\in \T_{\ell}$. 
An application of the triangle inequality gives
\begin{align}\label{eq:triangle_res_st}
h_{T}^{4}\|\bar{\mathfrak{u}}_{\ell} + f\|_{T}^{2}
\lesssim
h_{T}^{4}\|\Pi_{T}(\bar{\mathfrak{u}}_{\ell} + f)\|_{T}^{2} + h_{T}^{4}\|(\bar{\mathfrak{u}}_{\ell} + f) - \Pi_{T}(\bar{\mathfrak{u}}_{\ell} + f)\|_{T}^2.
\end{align}
To estimate $h_{T}^{4}\|\Pi_{T}(\bar{\mathfrak{u}}_{\ell} + f)\|_{T}^{2}$, we first evaluate $v = \varphi_{T}^{2}\Pi_{T}(\bar{\mathfrak{u}}_{\ell} + f)$ in \eqref{eq:error_eq_st}. 
Then, we use that $\nabla (\varphi^{2}\Pi_{T}(\bar{\mathfrak{u}}_{\ell} + f))=0$ on $\partial T$, standard properties of the bubble function $\varphi_{T}$, and the inverse estimate $\|\varphi_{T}^{2}\Pi_{T}(\bar{\mathfrak{u}}_{\ell} + f)\|_{L^{\infty}(T)} \lesssim h_{T}^{-\frac{d}{2}}\|\varphi_{T}^{2}\Pi_{T}(\bar{\mathfrak{u}}_{\ell} + f)\|_T$. 
These arguments yield
\begin{align*}
\|\Pi_T(\bar{\mathfrak{u}}_{\ell} + f)\|_{T}^{2}
\lesssim & \,
\|\bar{y} - \bar{y}_\ell\|_{T}\|\Delta (\varphi_{T}^{2}\Pi_{T}(\bar{\mathfrak{u}}_{\ell} + f))\|_T  \\
 & + (h_{T}^{-\frac{d}{2}}\|\bar{u} - \bar{\mathfrak{u}}_{\ell}\|_{L^{1}(T)} + \|(\bar{\mathfrak{u}}_{\ell} + f) - \Pi_{T}(\bar{\mathfrak{u}}_{\ell} + f)\|_{T})\|\Pi_T(\bar{\mathfrak{u}}_{\ell} + f)\|_T.
\end{align*}
Using that $\Delta (\varphi_{T}^{2} \Pi_{T}(\bar{\mathfrak{u}}_{\ell} + f)) = \Pi_{T}(\bar{\mathfrak{u}}_{\ell} + f)\Delta \varphi_{T}^{2}$ in combination with properties of $\varphi_{T}$ it follows that $\|\Delta (\varphi_{T}^{2}\Pi_{T}(\bar{\mathfrak{u}}_{\ell} + f))\|_{T} \lesssim h_{T}^{-2}\|\Pi_{T}(\bar{\mathfrak{u}}_{\ell} + f)\|_{T}$. 
Thus, we obtain 
\begin{align}\label{eq:res_T_st}
h_{T}^{4}\|\Pi_{T}(\bar{\mathfrak{u}}_{\ell} + f)\|_{T}^{2}
\lesssim &\,
\|\bar{y} - \bar{y}_\ell\|_{T}^2 + h_{T}^{4-d}\|\bar{u} - \bar{\mathfrak{u}}_{\ell}\|_{L^{1}(T)}^2 \\
& + h_T^4\|(\bar{\mathfrak{u}}_{\ell} + f) - \Pi_{T}(\bar{\mathfrak{u}}_{\ell} + f)\|_{T}^2. \nonumber
\end{align}
A combination of \eqref{eq:res_T_st} and \eqref{eq:triangle_res_st} concludes the estimation.

\underline{Step 2.} (estimation of $h_{T}^{3}\|\llbracket \nabla \bar{y}_\ell\cdot \mathbf{n} \rrbracket|_{e}\|_{e}^{2}$) 
Let $T\in \T_{\ell}$ and $e\in \mathcal{E}_{T}$. 
We begin by noticing that $\llbracket \nabla \bar{y}_\ell\cdot \mathbf{n} \rrbracket|_{e}\in \mathbb{R}$.
We thus extend the jump term $\llbracket \nabla \bar{y}_\ell\cdot \mathbf{n} \rrbracket|_{e}$, defined on $e$, to the patch $\mathcal{N}_{e}$ with the same value. 
We mention that when the jump term is not constant, such extension can also be done, for example, by using a continuation operator as in \cite[Section 3]{MR1650051}.
Hereinafter we make no distinction between the jump term and its extension.

We invoke the bubble function $\psi_{e}$, evaluate $v = \llbracket\nabla \bar{y}_\ell\cdot\mathbf{n}\rrbracket|_{e}\psi_{e}$ in \eqref{eq:error_eq_st}  and use that $\llbracket\nabla \bar{y}_\ell\cdot\mathbf{n}\rrbracket|_{e}\in \mathbb{R}$, $\psi_e \in H^2_0(\mathcal{N}_e)$, and $\llbracket\nabla \psi_{e} \cdot \mathbf{n}\rrbracket|_{e} = 0$. 
These, the fact $\Delta v =  \llbracket \nabla \bar{y}_\ell\cdot\mathbf{n} \rrbracket|_{e} \Delta\psi_e$, and basic inequalities imply that
\begin{align*}
 \|\llbracket \nabla \bar{y}_\ell\cdot \mathbf{n}\rrbracket|_{e} \psi_e^{\frac{1}{2}}\|_{e}^{2}
\lesssim &
\sum_{T' \in \mathcal{N}_{e}} \left( h_{T'}^{-2}\|\bar{y} - \bar{y}_{\ell}\|_{T'}  + \|\bar{\mathfrak{u}}_{\ell} + f \|_{T'}  \right. \\
& \left. + \,  h_{T'}^{-\frac{d}{2}}\|\bar{u} - \bar{\mathfrak{u}}_{\ell}\|_{L^1(T')} \right) h_{T}^{\frac{1}{2}}\| \llbracket\nabla \bar{y}_\ell\cdot\mathbf{n}\rrbracket|_{e} \|_{e},
\end{align*}
upon using the bound $\|\llbracket\nabla \bar{y}_\ell\cdot\mathbf{n}\rrbracket|_{e}\psi_{e}\|_{T'} \lesssim h_{T}^{\frac{1}{2}}\| \llbracket\nabla \bar{y}_\ell\cdot\mathbf{n}\rrbracket|_{e} \|_{e}$.

With these estimates at hand, we thus use standard bubble functions arguments and the shape regularity property of the family $\{ \mathcal{T}_\ell \}$ to arrive at
\begin{align*}
h_T^{\frac{3}{2}}\| \llbracket \nabla \bar{y}_\ell\cdot\mathbf{n}\rrbracket|_{e}\|_{e}
\lesssim
\sum_{T'\in\mathcal{N}_e}  
(
\|\bar{y} - \bar{y}_{\ell}\|_{T'} + h_T^2\|\bar{\mathfrak{u}}_{\ell}  + f \|_{T'} + h_{T}^{2-\frac{d}{2}}\|\bar{u} - \bar{\mathfrak{u}}_{\ell}\|_{L^{1}(T')}
).
\end{align*}
The final estimate follows in view of  \eqref{eq:triangle_res_st}  and \eqref{eq:res_T_st}.
\end{proof}

We now study local efficiency properties of the estimator $\eta_{adj,\infty}$.

\begin{lemma}[local efficiency of ${\eta}_{adj,\infty}$]\label{lemma:efficiency_adj_infty}
Let $\bar{u}\in \mathbb{U}_{ad}$ be the unique solution to problem \eqref{def:weak_ocp}--\eqref{eq:weak_st_eq} with $\bar{y}$ and $\bar{p}$ being the corresponding optimal state and adjoint state, respectively. 
Let $\bar{\mathfrak{u}}_{\ell}\in \mathbb{U}_{ad}$ be a solution to the semi-discrete problem with $\bar{y}_{\ell}$ and $\bar{p}_{\ell}$ being the corresponding discrete state and adjoint state variables, respectively.  
Then, for $T\in\mathcal{T}_{\ell}$, the local error indicator $E_{adj,\infty,T}$, defined as in \eqref{def:pointwise_estimator}, satisfies 
\begin{equation*}
E_{adj,T,\infty}^2
\lesssim
\|\bar{p}-\bar{p}_{\ell}\|_{L^\infty(\mathcal{N}_T)}^2 + h_{T}^{4 - d}\|\bar{y} - \bar{y}_{\ell}\|_{\mathcal{N}_T}^2 + \sum_{T'\in\mathcal{N}_{T}}h_{T}^{4-d}\|y_{\Omega} - \Pi_{T}y_{\Omega}\|_{T'}^{2},
\end{equation*}
where $\mathcal{N}_T$ is defined as in \eqref{def:patch} and the hidden constant is independent of continuous and discrete optimal variables, the size of the elements in the mesh $\mathcal{T}_{\ell}$, and $\#\mathcal{T}_{\ell}$.
\end{lemma}
\begin{proof}
Similar arguments to the ones that lead to \eqref{eq:error_eq_st} yield, for every $v \in H_0^{1}(\Omega)$  such that $v|_T\in C^2(T)$ ($T\in \mathcal{T}_{\ell}$), the identity
\begin{align}\label{eq:error_eq_adj}
&\sum_{e\in\mathcal{E}_{\ell}} (\llbracket \nabla v\cdot \mathbf{n} \rrbracket|_{e}, \bar{p}-\bar{p}_{\ell})_{e} - \sum_{T\in \T_{\ell}} (\bar{p}-\bar{p}_{\ell}, \Delta v)_{T}  -  (\bar{y} - \bar{y}_{\ell},v)_{\Omega} \\
&=
\sum_{T\in \T_{\ell}}( \Pi_{T}y_{\Omega} - y_{\Omega},v)_{T} +
\sum_{T\in \T_{\ell}} ( \bar{y}_{\ell} - \Pi_{T}y_{\Omega}, v )_{T}  - \sum_{e\in\mathcal{E}_{\ell}} ( \llbracket \nabla \bar{p}_{\ell}\cdot \mathbf{n} \rrbracket|_{e}, v)_{e}. \nonumber \hspace{-0.3cm}
\end{align}
We now estimate the terms in \eqref{def:pointwise_estimator} in two steps.

\underline{Step 1.} (estimation of $h_{T}^{4-d}\|\bar{y}_{\ell} - y_{\Omega}\|_{T}^{2}$) Let $T\in\T_{\ell}$. 
To estimate this term in \eqref{def:pointwise_estimator}, we first use the triangle inequality to obtain
\begin{equation}\label{eq:triangle_st_2}
h_{T}^{4-d}\|\bar{y}_{\ell} - y_{\Omega}\|_{T}^{2}
\lesssim
h_{T}^{4-d}\|\bar{y}_{\ell} - \Pi_{T}y_{\Omega}\|_{T}^{2}
+
h_{T}^{4-d}\| \Pi_{T}y_{\Omega} - y_{\Omega}\|_{T}^{2}.
\end{equation}
We recall that $\Pi_{T}$ denotes the orthogonal projection operator into constant functions over $T$. 
To control $h_T^{4-d}\|\bar{y}_{\ell} - \Pi_{T}y_{\Omega}\|_{T}^{2}$ in the above inequality, we invoke the interior bubble function $\varphi_{T}$ and choose $v=\varphi_{T}^{2}(\bar{y}_{\ell} - \Pi_{T}y_{\Omega})$ in \eqref{eq:error_eq_adj}. 
Hence, using standard properties of $\varphi_{T}$ and the inequality  $
\|\Delta(\varphi_{T}^{2}(\bar{y}_{\ell} - \Pi_{T}y_{\Omega}))\|_{L^{1}(T)} 
\lesssim 
h_{T}^{\frac{d}{2}-2}\|\bar{y}_{\ell} - \Pi_{T}y_{\Omega}\|_{T}
$, which follows from the inverse estimate \cite[Lemma 4.5.3]{MR2373954}, we obtain
\begin{align*}
\|\bar{y}_{\ell} - \Pi_{T}y_{\Omega}\|_{T} 
\lesssim
\|\Pi_{T}y_{\Omega} - y_{\Omega}\|_{T} + \|\bar{y} - \bar{y}_{\ell}\|_{T} + h_{T}^{\frac{d}{2}-2}\|\bar{p} - \bar{p}_{\ell}\|_{L^{\infty}(T)},
\end{align*}
from which we conclude that
\begin{equation}\label{eq:estimate_residual_adj_2}
h_{T}^{4-d}\|\bar{y}_{\ell} - \Pi_{T}y_{\Omega}\|_{T}^{2}
\lesssim h_{T}^{4-d}\|\Pi_{T}y_{\Omega} - y_{\Omega}\|_{T}^2 + h_{T}^{4-d}\|\bar{y} - \bar{y}_{\ell}\|_{T}^2 + \|\bar{p} - \bar{p}_{\ell}\|_{L^{\infty}(T)}^{2}.
\end{equation}
The use of the latter in \eqref{eq:triangle_st_2} results in the desired bound.

\underline{Step 2.} (estimation of $h_T^{2}\|\llbracket\nabla \bar{p}_{\ell}\cdot \mathbf{n}\rrbracket|_{e}\|_{L^\infty(e)}^{2}$)  
Let $T\in \T_{\ell}$ and $e\in \mathcal{E}_{T}$. 
We choose $v = \llbracket \nabla \bar{p}_{\ell}\cdot \mathbf{n} \rrbracket|_{e}\psi_{e}$ in \eqref{eq:error_eq_adj}  and use that $\llbracket \nabla \bar{p}_{\ell}\cdot \mathbf{n} \rrbracket|_{e}\in \mathbb{R}$, $\psi_e \in H^2_0(\mathcal{N}_e)$, and $\llbracket\nabla \psi_{e} \cdot \mathbf{n}\rrbracket|_{e} = 0$. 
Consequently, basic inequalities imply that
\begin{align*}
\|\llbracket \nabla \bar{p}_{\ell}\cdot \mathbf{n} \rrbracket|_{e}\|_{e}^{2}
\lesssim & \,
\sum_{T'\in \mathcal{N}_{e}} \big(h_{T}^{\frac{d}{2}-2}\|\bar{p} - \bar{p}_{\ell}\|_{L^{\infty}(T')} \\
& + \|\bar{y} - \bar{y}_{\ell}\|_{T'} + \|\bar{y}_{\ell} - y_{\Omega}\|_{T'}\big)\|\psi_e\llbracket \nabla \bar{p}_{\ell}\cdot \mathbf{n} \rrbracket|_{e}\|_{T'}.
\end{align*}
Therefore, using that $\|\psi_e\llbracket \nabla \bar{p}_{\ell}\cdot \mathbf{n} \rrbracket|_{e}\|_{T'}\lesssim h_{T}^{\frac{1}{2}} \|\llbracket \nabla \bar{p}_{\ell}\cdot \mathbf{n} \rrbracket|_{e}\|_{e}$ in combination with estimate \eqref{eq:estimate_residual_adj_2}, we conclude that 
\begin{align}\label{eq:estimate_jump_pl}
h_{T}^{3}\|\llbracket \nabla \bar{p}_{\ell}\cdot \mathbf{n} \rrbracket|_{e}\|_{e}^{2}
\lesssim &
\sum_{T'\in \mathcal{N}_{e}}\big(h_{T}^{d}\|\bar{p} - \bar{p}_{\ell}\|_{L^{\infty}(T')}^{2} \\
& \, +  h_{T}^{4}\|\bar{y} - \bar{y}_{\ell}\|_{T'}^{2} + h_{T}^{4}\|\Pi_{T}y_{\Omega} - y_{\Omega}\|_{T'}^{2} \big). \nonumber
\end{align}
On the other hand, since $\llbracket \nabla \bar{p}_{\ell}\cdot \mathbf{n} \rrbracket|_{e}\in \mathbb{R}$, we deduce that 
\begin{equation*}
\|\llbracket \nabla \bar{p}_{\ell}\cdot \mathbf{n} \rrbracket|_{e}\|_{L^{\infty}(e)}^{2} 
=
|\llbracket \nabla \bar{p}_{\ell}\cdot \mathbf{n} \rrbracket|_{e}|^{2}
=
|e|^{-1}
\|\llbracket \nabla \bar{p}_{\ell}\cdot \mathbf{n} \rrbracket|_{e}\|_{e}^{2},
\end{equation*}
where $|e|$ denotes the measure of $e$. 
In view of the shape regularity of the mesh $\T_{\ell}$ we have that $|e|\approx h_T^{d-1}$ and consequently $h_T^{2+d}\|\llbracket \nabla \bar{p}_{\ell}\cdot \mathbf{n} \rrbracket|_{e}\|_{L^{\infty}(e)}^{2}  \approx
h_T^3 \|\llbracket \nabla \bar{p}_{\ell}\cdot \mathbf{n} \rrbracket|_{e}\|_{e}^{2}$. 
Hence, using the latter in estimate \eqref{eq:estimate_jump_pl} and multiplying by $h_T^{-d}$, we arrive at
\begin{equation*}
h_{T}^{2}\|\llbracket \nabla \bar{p}_{\ell}\cdot \mathbf{n} \rrbracket|_{e}\|_{L^{\infty}(e)}^{2}
\lesssim 
\sum_{T'\in \mathcal{N}_{e}}\big(\|\bar{p} - \bar{p}_{\ell}\|_{L^{\infty}(T')}^{2} +  h_{T}^{4-d}\|\bar{y} - \bar{y}_{\ell}\|_{T'}^{2} + h_{T}^{4-d}\|\Pi_{T}y_{\Omega} - y_{\Omega}\|_{T'}^{2} \big),
\end{equation*}
 which concludes the proof.
\end{proof}

The next result is a direct consequence of Lemmas \ref{lemma:efficiency_est_2} and \ref{lemma:efficiency_adj_infty}.

\begin{theorem}[local efficiency]\label{thm:efficiency_est_ocp}
In the framework of Lemma \ref{lemma:efficiency_adj_infty} we have, for $T\in\mathcal{T}_{\ell}$, that
\begin{align*}
E_{st,T}^2  \, +\,   & E_{adj,T,\infty}^2 
\lesssim  ~ (1 + h_T^{4-d})\|\bar{y} - \bar{y}_{\ell}\|_{\mathcal{N}_T}^2
+ \|\bar{p}-\bar{p}_{\ell}\|_{L^\infty(\mathcal{N}_T)}^2 + h_{T}^{4 - d}\|\bar{u}-\bar{\mathfrak{u}}_{\ell}\|_{L^1(\mathcal{N}_T)}^2 \\
 & \qquad + 
\sum_{T'\in\mathcal{N}_{T}}h_{T}^{4}\|(\bar{\mathfrak{u}}_{\ell} + f) - \Pi_{T}(\bar{\mathfrak{u}}_{\ell} + f)\|_{T'}^{2}  + \sum_{T'\in\mathcal{N}_{T}}h_{T}^{4-d}\|y_{\Omega} - \Pi_{T}(y_{\Omega})\|_{T'}^{2},
\end{align*}
where $\mathcal{N}_T$ is defined as in \eqref{def:patch} and the hidden constant is independent of continuous and discrete optimal variables, the size of the elements in the mesh $\mathcal{T}_{\ell}$, and $\#\mathcal{T}_{\ell}$.
\end{theorem}


\section{Numerical examples}\label{sec:num_ex}

In this section we conduct a series of numerical examples in 2D that support our theoretical findings and illustrate the performance of the error estimator 
\begin{align}\label{def:total_estimator}
E^{2}
:=
{\eta}_{st,2}^{2} +  {\eta}_{adj,\infty}^{2}
\end{align}
that we proposed and analyzed in section \ref{sec:a_post}.
In sections \ref{sec:ex_2} and \ref{sec:ex_3} below, we go beyond the presented theory and perform numerical experiments with a non-convex domain. 
The considered numerical examples have been carried out with the help of a code that we implemented using \texttt{MATLAB$^\copyright$ (R2024a)}. 
When assembling all system matrices and the term $(\mathfrak{u}_{\ell},v_{\ell})_{\Omega}$ we have used exact integration whereas approximation errors, error indicators, and the remaining right-hand sides are computed by a quadrature formula which is exact for polynomials of degree $19$.

For a given partition $\T_{\ell}$, we seek $\bar{y}_\ell \in \mathbb{V}_{\ell}$,  $\bar{p}_\ell \in \mathbb{V}_{\ell}$, and $\bar{\mathfrak{u}}_\ell \in \mathbb{U}_{ad}$ that solve \eqref{eq:discrete_pde_semi}, \eqref{eq:discrete_adj_eq}, and \eqref{eq:pointwise_charac_discrete}.
We solve such a nonlinear system of equations by using the fixed-point iteration described in Algorithm \ref{Algorithm1}; see also \cite[Section 4]{MR2891922}. 
Once a discrete solution is obtained, we compute the error indicator 
\begin{align}\label{def:total_indicator}
E_{T} := \left(E_{st,T}^{2\beta} + E_{adj,T,\infty}^{2\beta}\right)^{\frac{1}{2}} \qquad (\beta \in (0,1])
\end{align}
to drive the adaptive procedure described in Algorithm \ref{Algorithm2}.
For the numerical results, we define the total number of degrees of freedom $\rm{Ndofs} = 2\:\rm{dim}\mathbb{V}_{\ell}$.
Finally, we define the effectivity index 
\begin{equation*}
\mathcal{I}_{eff}:= \frac{E}{(\|\bar{u} - \bar{\mathfrak{u}}_{\ell}\|_{L^1(\Omega)}^{2} + \|\bar{y} - \bar{y}_{\ell}\|_{\Omega}^{2} + \|\bar{p} - \bar{p}_{\ell}\|_{L^{\infty}(\Omega)}^{2})^{\frac{1}{2}}}.
\end{equation*}

\begin{algorithm}[ht]
\caption{\textbf{Fixed-point Iteration.}}
\label{Algorithm1}
\SetKwInput{set}{Set}
\SetKwInput{ase}{while $tol > 10^{-10}$ do}
\SetKwInput{al}{endwhile}
\SetKwInput{ret}{Return}
\SetKwInput{Input}{Input}
\SetKwInput{Output}{Output}
\SetAlgoLined
\Input{Mesh $\T_{\ell}$, initial variables $(y_{\ell}^{0},p_{\ell}^{0},\mathfrak{u}_{\ell}^{0})$, constraints $a$ and $b$, desired state $y_\Omega$, right-hand side $f$, and $tol=1$.}
\set{$i=0$.}
\While{$tol > 10^{-10}$}{
Obtain $y_{\ell}^{i+1}$ by solving 
$
(\nabla y_{\ell}^{i+1},\nabla v_\ell)_{\Omega}
=
(\mathfrak{u}_{\ell}^{i} + f, v_\ell)_{\Omega} 
$ for all $v_\ell\in \mathbb{V}_{\ell}$;
\\
Obtain $p_{\ell}^{i+1}$ by solving
$
(\nabla p_{\ell}^{i+1},\nabla v_\ell)_{\Omega}
=
(y_{\ell}^{i+1} - y_{\Omega}, v_\ell)_{\Omega}$ for all $v_\ell\in \mathbb{V}_{\ell}$;
\\
Set $\mathfrak{u}_{\ell}^{i+1} = - \text{sign}(p_{\ell}^{i+1})$ (when $a=-1$ and $b=1$);
\\
Set 
$
tol= \left( \sum_{\texttt{v}\in \mathcal{V}_{\ell}}\left[\left(y_{\ell}^{i+1}(\texttt{v}) - y_{\ell}^{i}(\texttt{v})\right)^{2} + \left(p_{\ell}^{i+1}(\texttt{v}) - p_{\ell}^{i}(\texttt{v})\right)^{2}\right]  \right)^{\frac{1}{2}},
$
where $\mathcal{V}_{\ell}$ denotes the set of vertices of $\mathcal{T}_{\ell}$;
\\
Set $i \leftarrow i + 1$, and repeat.
}
\Output{Approximate optimal solutions $(y_{\ell}^{i},p_{\ell}^{i},\mathfrak{u}_{\ell}^{i})=(\bar{y}_{\ell},\bar{p}_{\ell},\bar{\mathfrak{u}}_{\ell})$.}
\end{algorithm}

\begin{algorithm}[ht]
\caption{\textbf{Adaptive Algorithm.}}
\label{Algorithm2}
\SetKwInput{set}{Set}
\SetKwInput{ase}{Fixed-point iteration}
\SetKwInput{al}{A posteriori error estimation}
\SetKwInput{Input}{Input}
\SetAlgoLined
\Input{Initial mesh $\T_0$, initial variables $(y_{0},p_{0},\mathfrak{u}_{0})$, constraints $a$ and $b$, desired state $y_\Omega$, right-hand side $f$, and $tol=1$.}
\set{$\ell=0$.}
\ase{}
Compute $[\bar{y}_{\ell},\bar{p}_{\ell},\bar{\mathfrak{u}}_{\ell}]=\textbf{Fixed-point iteration}[{\mathcal{T}_{\ell}},(y_{0},p_{0},\mathfrak{u}_{0}),a,b, y_\Omega, f,tol]$ by using Algorithm \ref{Algorithm1};
\\
\al 
\\
For each $T\in{\T_\ell}$ compute the local error indicator $E_{T}$ given in \eqref{def:total_indicator};
\\
Mark an element $T$ for refinement if $E_{T}> \displaystyle\frac{1}{2}\max_{T'\in{\T_\ell}} E_{T'}$;
\\
From step $\mathbf{3}$, construct a new mesh, using a longest edge bisection algorithm. Set $\ell \leftarrow \ell + 1$, and go to step $\mathbf{1}$.
\end{algorithm}
\normalsize


\subsection{Exact solution on convex domain}\label{sec:ex_1}
Following \cite[Section 3.3]{MR3095657} (see also \cite[Section 5.3]{MR3385653}) we set $\Omega:=(0,1)^{2}$, $a=-1$, $b=1$, $\beta=1$, and take $f$ and $y_{\Omega}$ such that 
\begin{align*}
  &\bar{y}(x_{1},x_{2}) = \sin(\pi x_{1})\sin(\pi x_{2}), \\ 
  & \bar{p}(x_{1},x_{2}) = -\sin(2\pi x_{1})\sin(2\pi x_{2})/(8\pi^{2}), \quad \bar{u}(x_{1},x_{2}) = -\mathrm{sign}(\bar{p}(x_{1},x_{2}))
\end{align*}
for  $(x_{1},x_{2})\in \Omega$.


\begin{figure}[!ht]
\centering
\begin{minipage}[c]{0.45\textwidth}\centering
\includegraphics[trim={0 0 0 0},clip,width=5.5cm,height=5.3cm,scale=0.30]{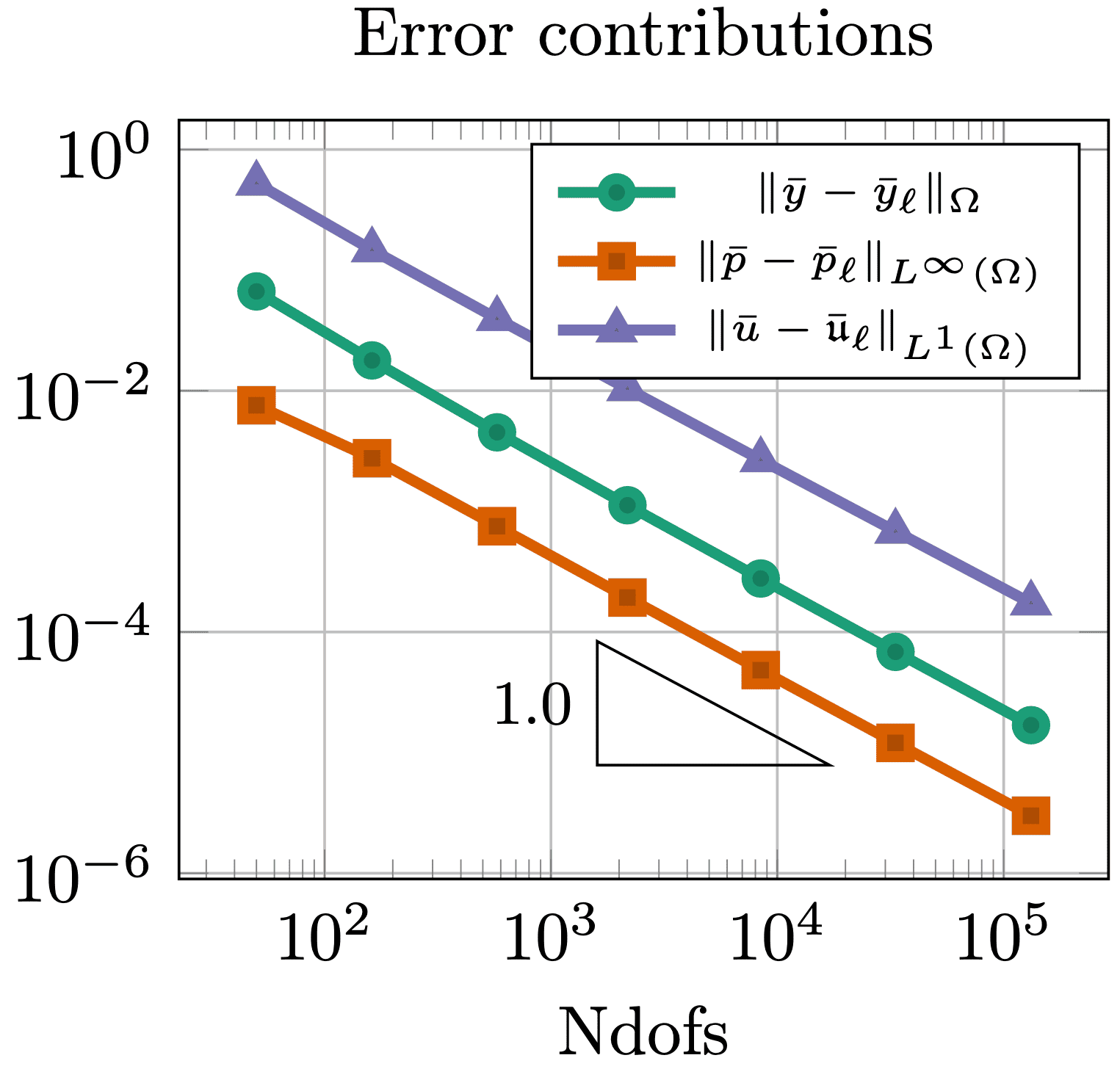}\\
\qquad
{\small{(1.A)}}
\end{minipage}
\begin{minipage}[c]{0.45\textwidth}\centering
\includegraphics[trim={0 0 0 0},clip,width=5.5cm,height=5.3cm,scale=0.30]{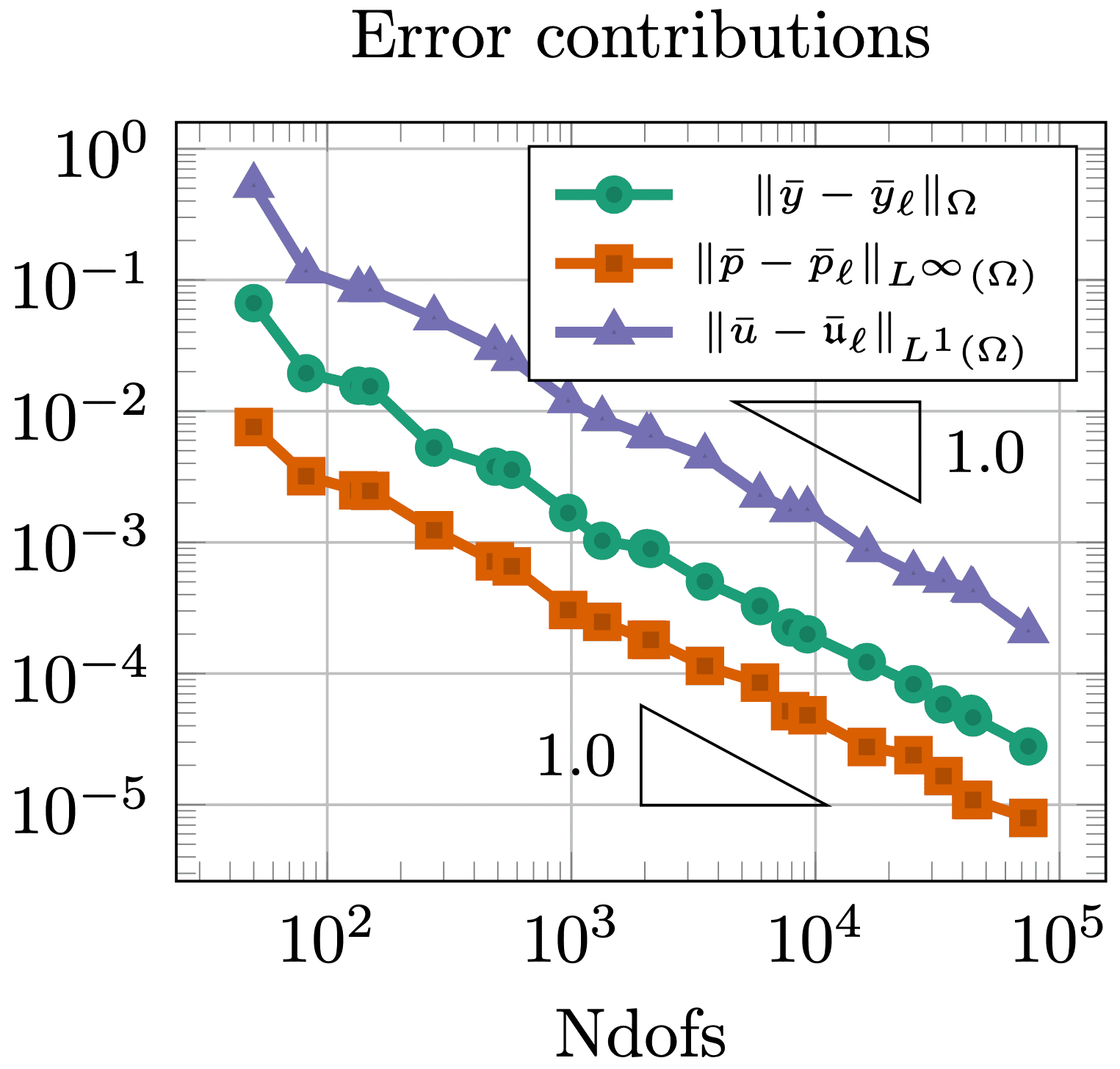}\\
\qquad
{\small{(1.B)}}
\end{minipage}
\\
\begin{minipage}[c]{0.45\textwidth}\centering
\includegraphics[trim={0 0 0 0},clip,width=5.5cm,height=5.3cm,scale=0.30]{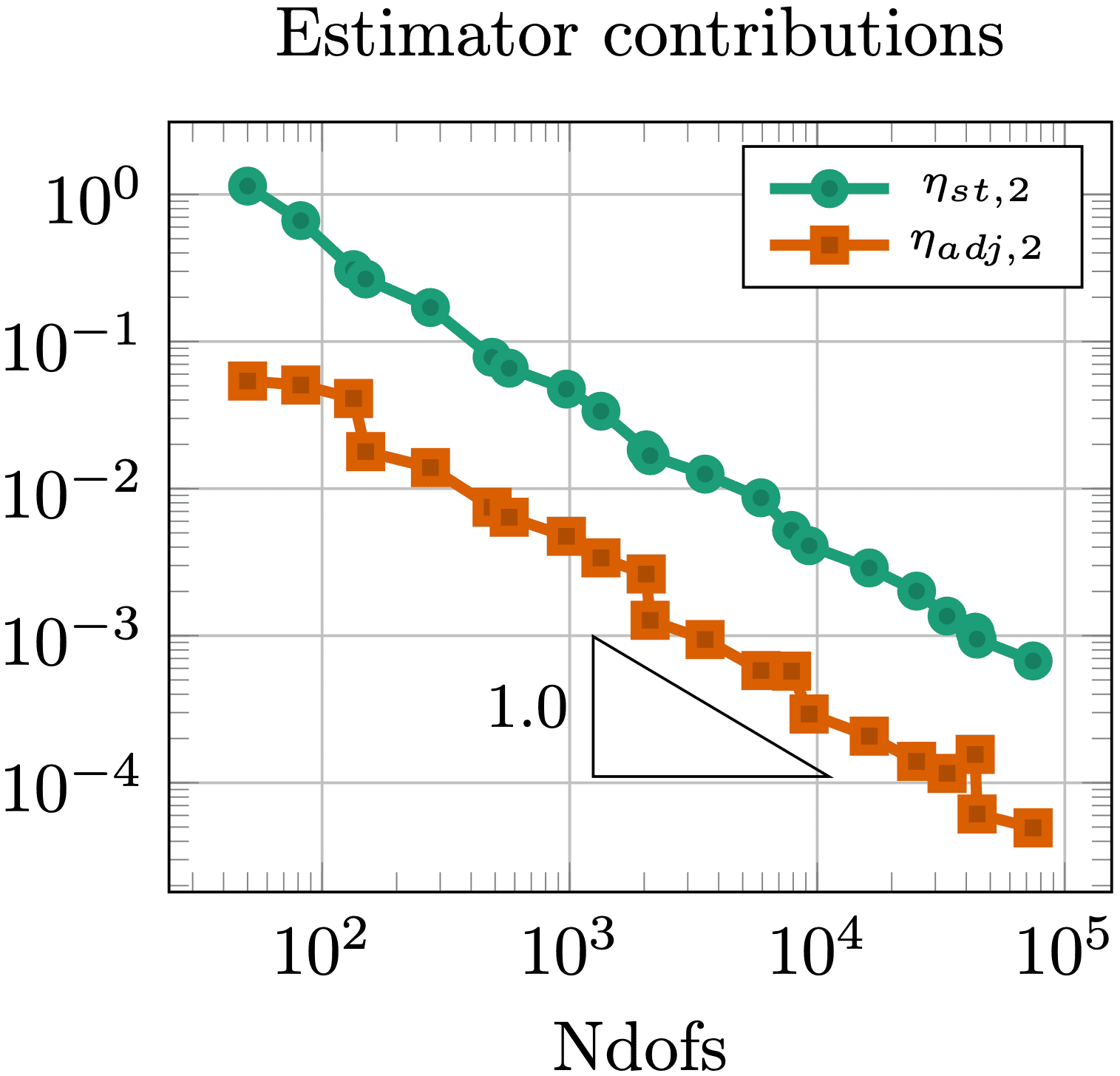}\\
\qquad
{\small{(1.C)}}
\end{minipage}
\begin{minipage}[c]{0.45\textwidth}\centering
\includegraphics[trim={0 0 0 0},clip,width=5.5cm,height=5.3cm,scale=0.30]{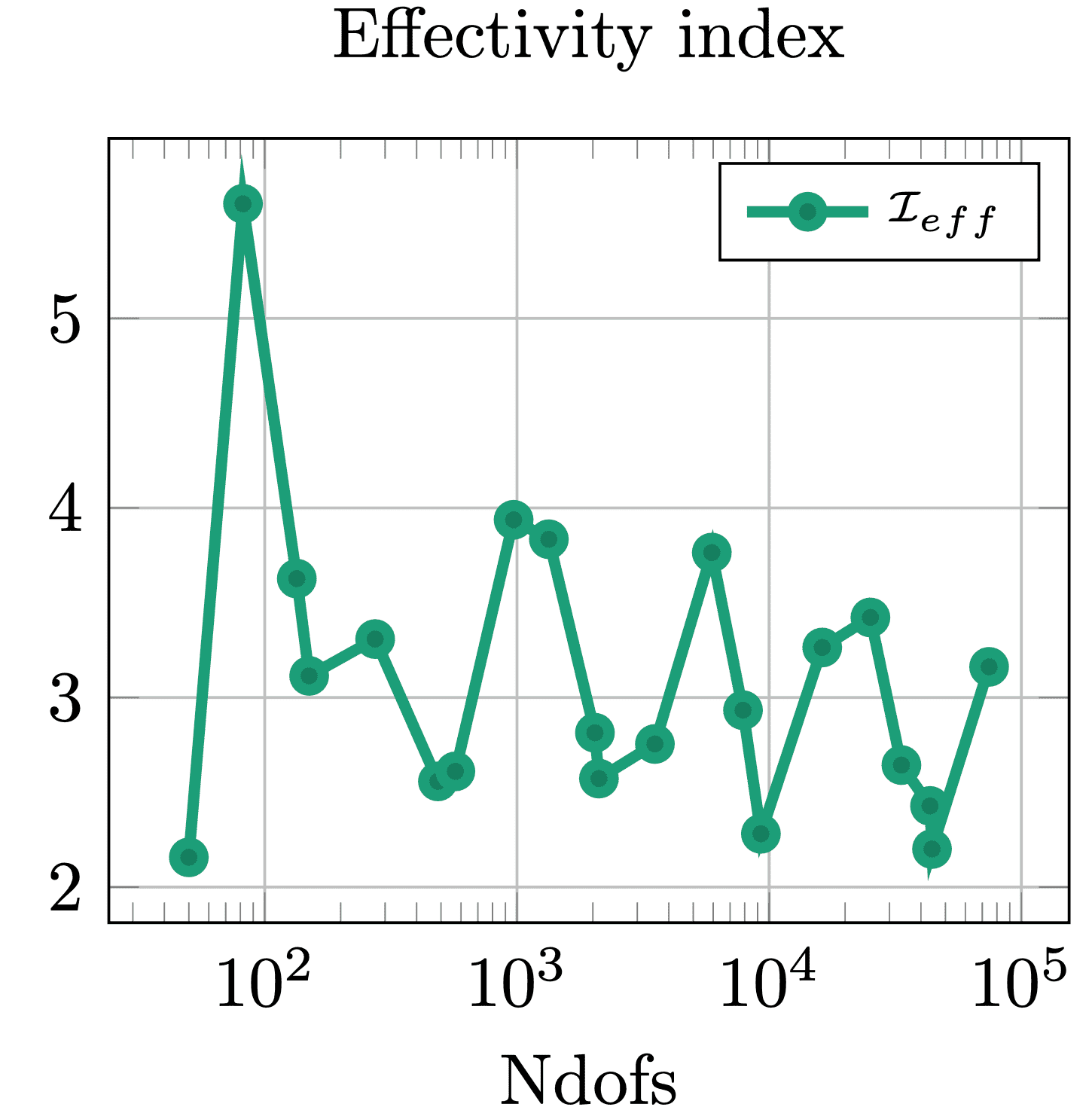}\\
\qquad
{\small{(1.D)}}
\end{minipage}
 \caption{Experimental rates of convergence for individual contributions of the total error with uniform (1.A) and adaptive (1.B) refinement, convergence rates for individual contributions of the estimator $E$ (1.C), and effectivity index (1.D) with adaptive refinement for the problem from section \ref{sec:ex_1}.}
\label{fig:ex_1}
\end{figure}


\begin{figure}[!ht]
\begin{minipage}[c]{0.32\textwidth}\centering
\includegraphics[trim={0 0 0 0},clip,width=4.0cm,height=4.0cm,scale=0.30]{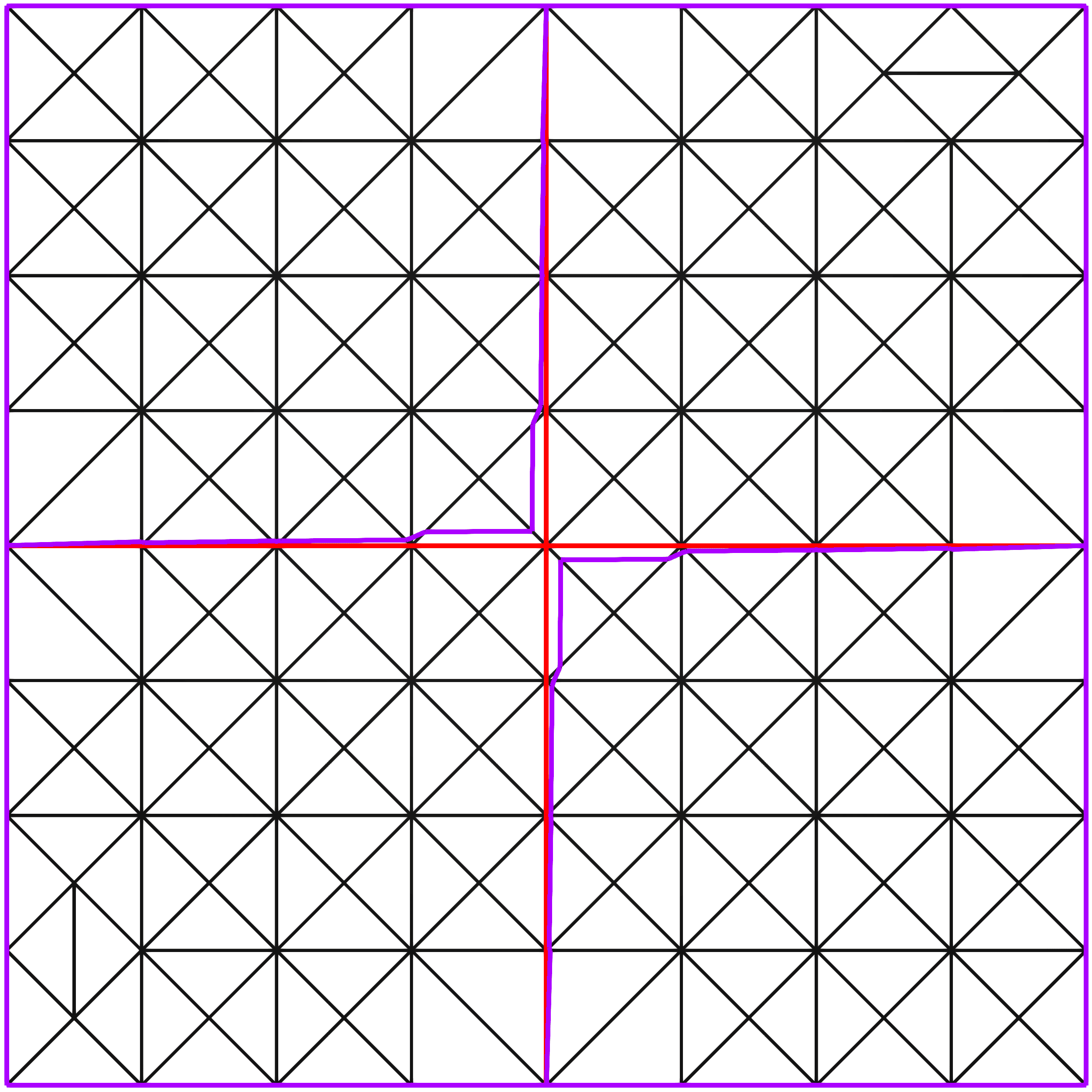}\\
{\small{(2.A)}}
\end{minipage}
\begin{minipage}[c]{0.32\textwidth}\centering
\includegraphics[trim={0 0 0 0},clip,width=4.0cm,height=4.0cm,scale=0.30]{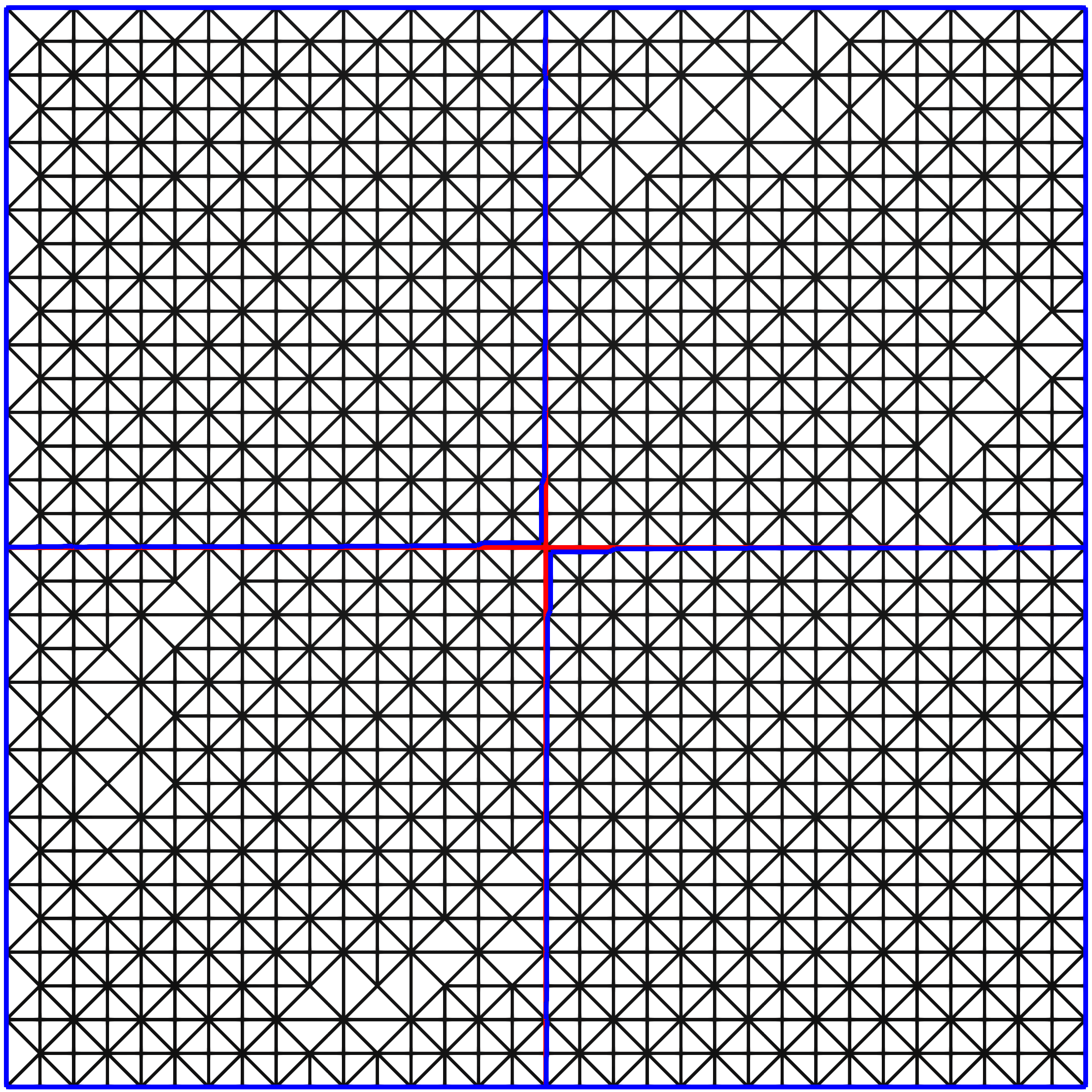}\\
{\small{(2.B)}}
\end{minipage}
\begin{minipage}[c]{0.32\textwidth}\centering
\includegraphics[trim={0 0 0 0},clip,width=4.0cm,height=4.0cm,scale=0.30]{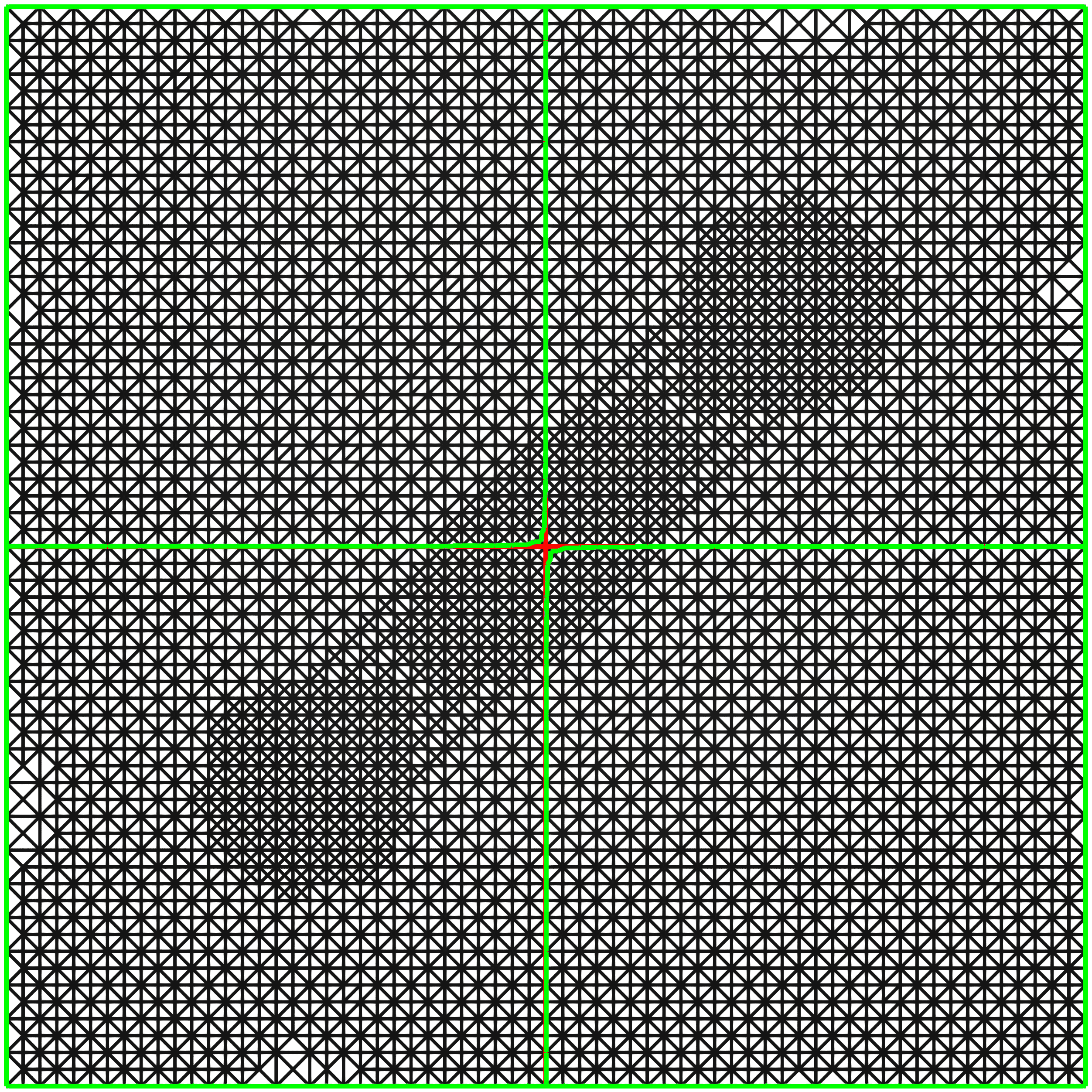}\\
{\small{(2.C)}}
\end{minipage}
\\
\begin{minipage}[c]{0.32\textwidth}\centering
\includegraphics[trim={0 0 0 0},clip,width=4.0cm,height=4.0cm,scale=0.30]{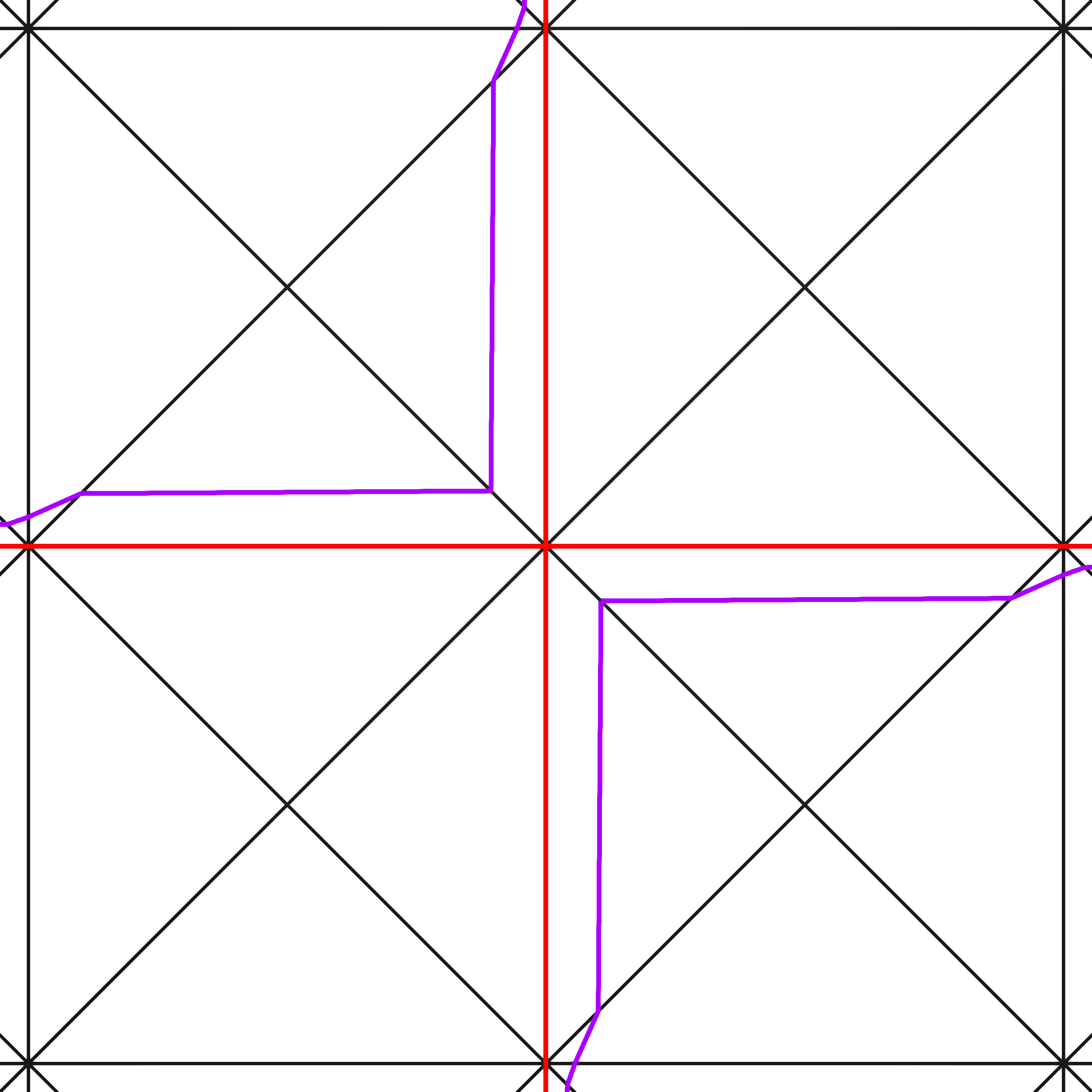}\\
{\small{(2.D)}}
\end{minipage}
\begin{minipage}[c]{0.32\textwidth}\centering
\includegraphics[trim={0 0 0 0},clip,width=4.0cm,height=4.0cm,scale=0.30]{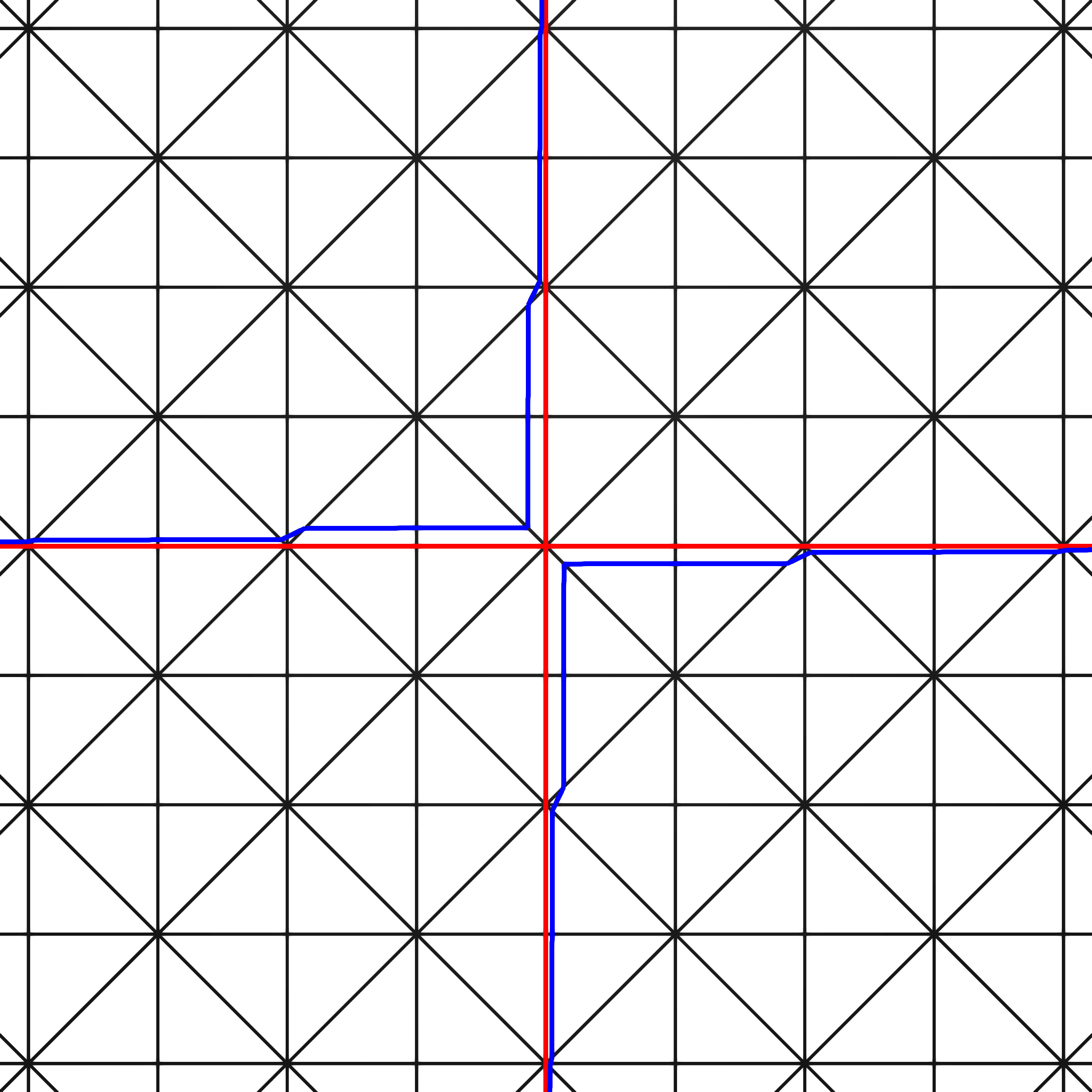}\\
{\small{(2.E)}}
\end{minipage}
\begin{minipage}[c]{0.32\textwidth}\centering
\includegraphics[trim={0 0 0 0},clip,width=4.0cm,height=4.0cm,scale=0.30]{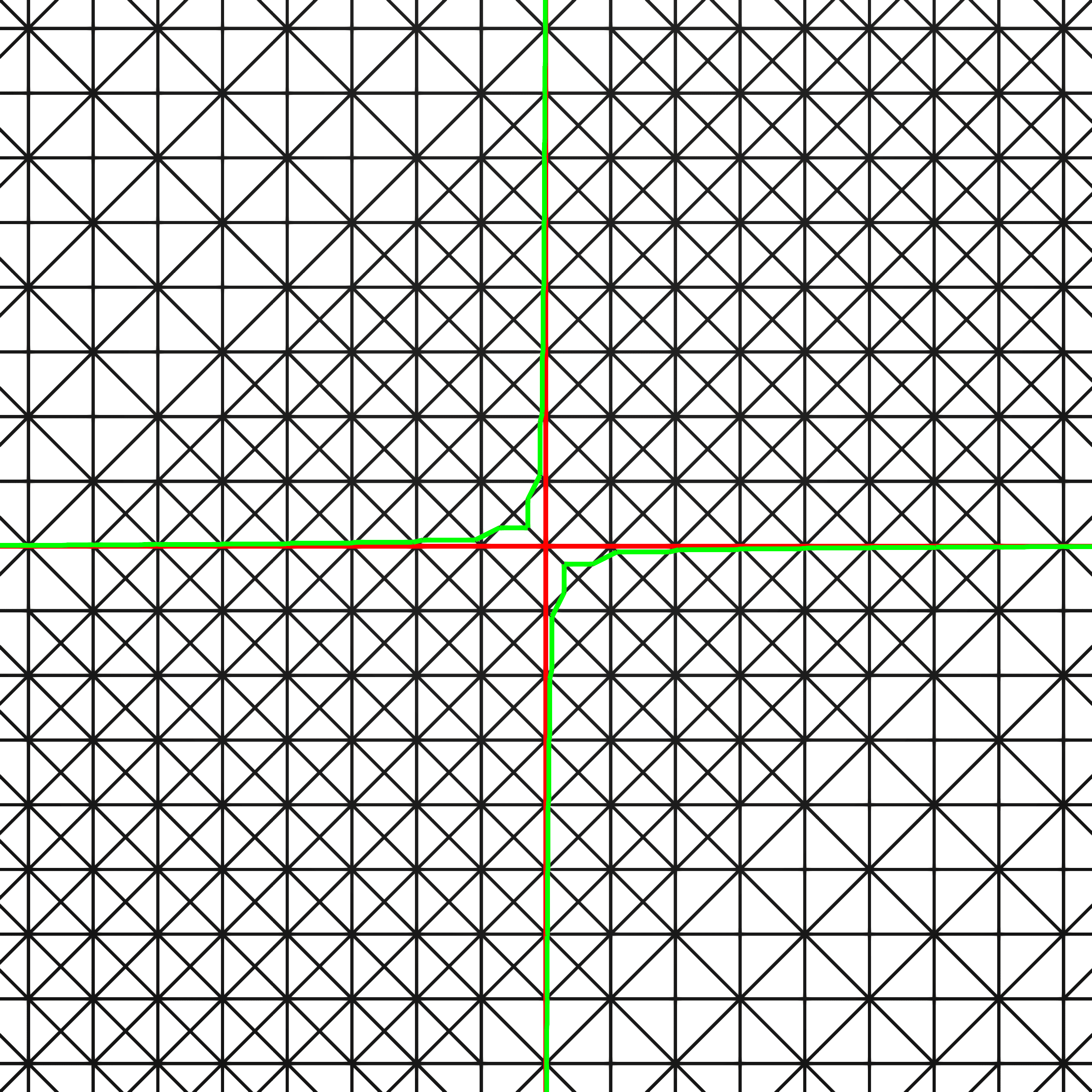}\\
{\small{(2.F)}}
\end{minipage}
\caption{Comparison of the continuous (red) and discrete switching sets on the adaptively refined meshes obtained after $5$ ((2.A) and (2.D)), $10$ ((2.B) and (2.E)), and $15$ ((2.C) and (2.F)) iterations for the problem from section \ref{sec:ex_1}.}
\label{fig:ex_1_2}
\end{figure}


\begin{figure}[!ht]
\centering
\begin{minipage}[c]{0.32\textwidth}\centering
\includegraphics[trim={0 0 0 0},clip,width=4.0cm,height=4.0cm,scale=0.30]{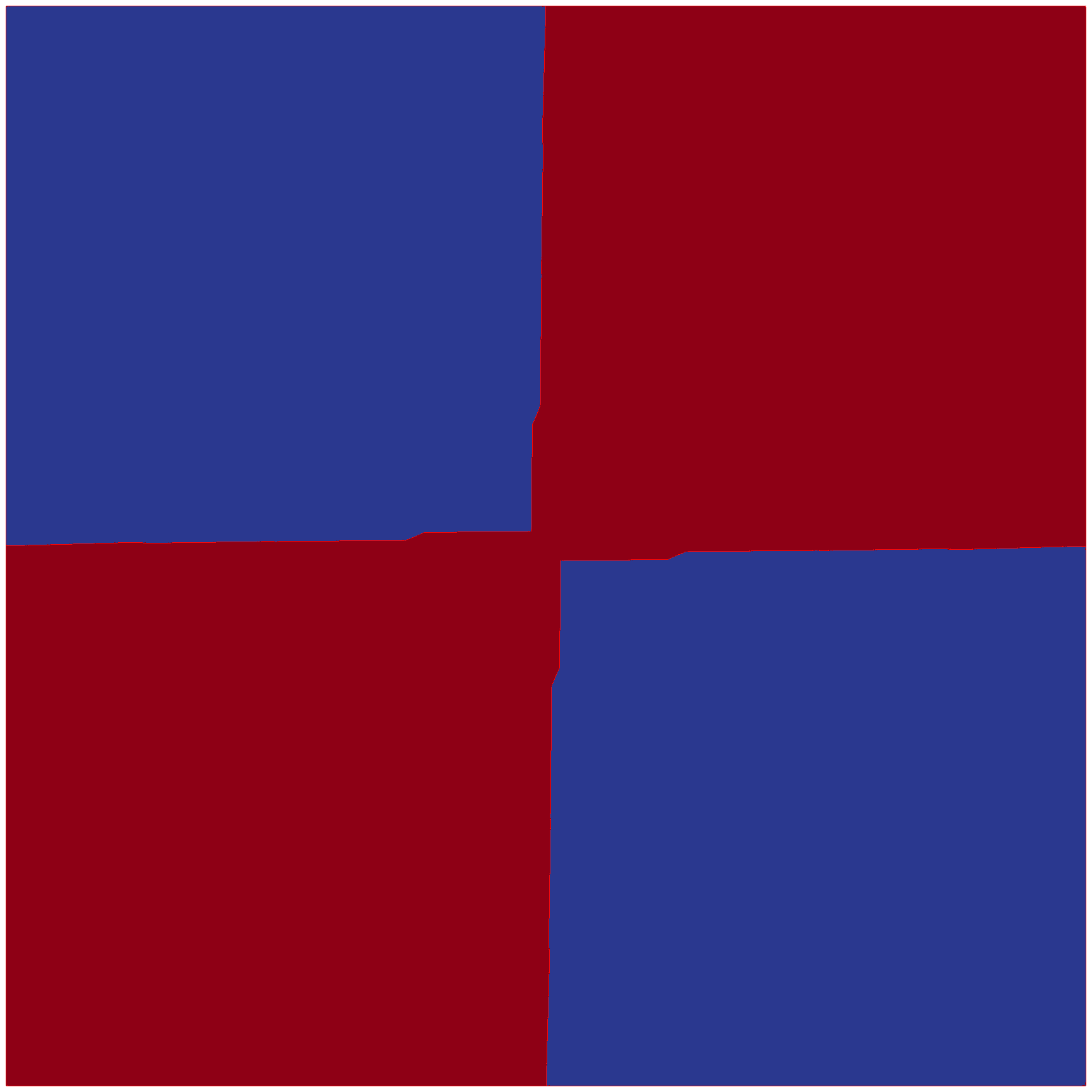}
{\small{(3.A)}}
\end{minipage}
\begin{minipage}[c]{0.32\textwidth}\centering
\includegraphics[trim={0 0 0 0},clip,width=4.0cm,height=4.0cm,scale=0.30]{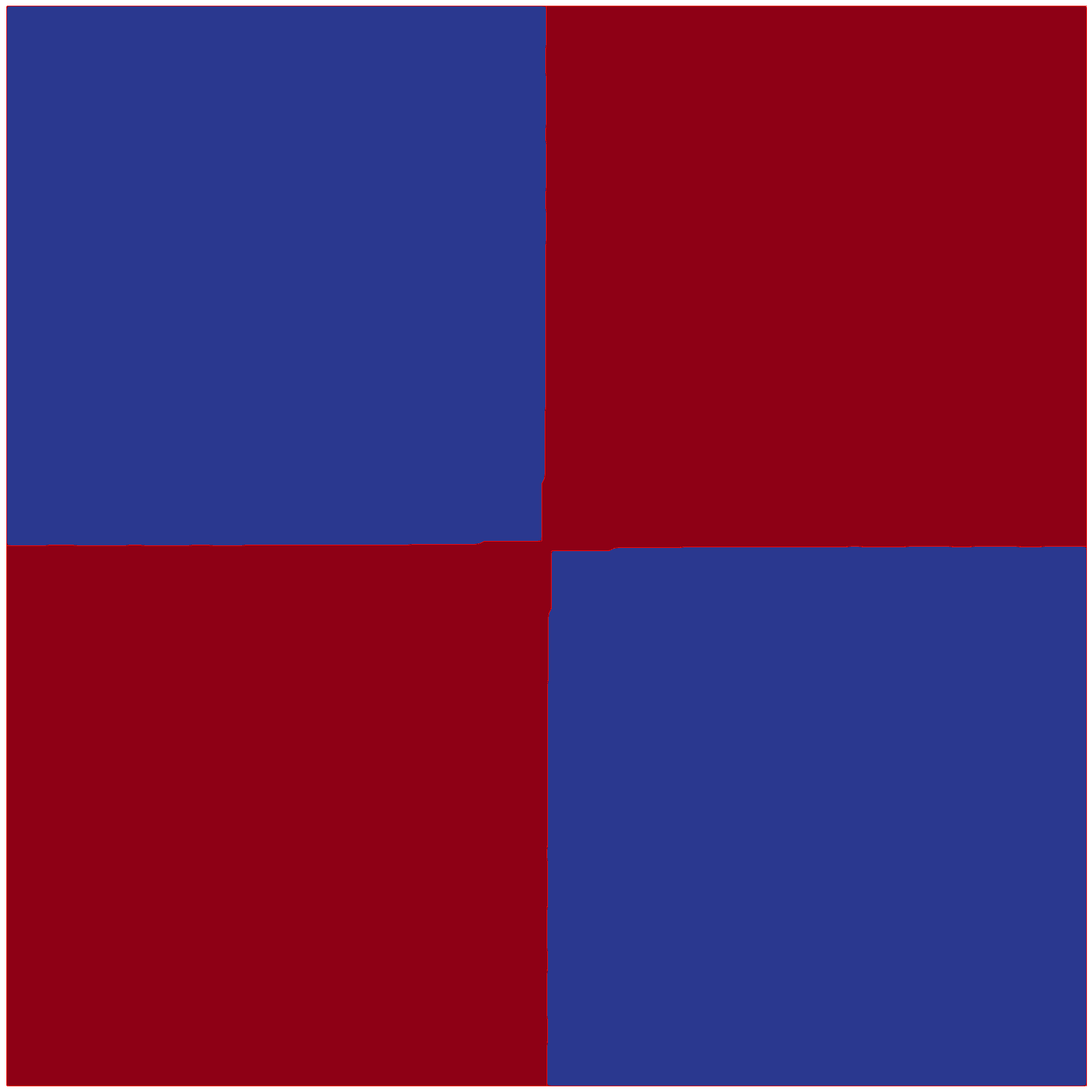}
{\small{(3.B)}}
\end{minipage}
\begin{minipage}[c]{0.32\textwidth}\centering
\includegraphics[trim={0 0 0 0},clip,width=4.0cm,height=4.0cm,scale=0.30]{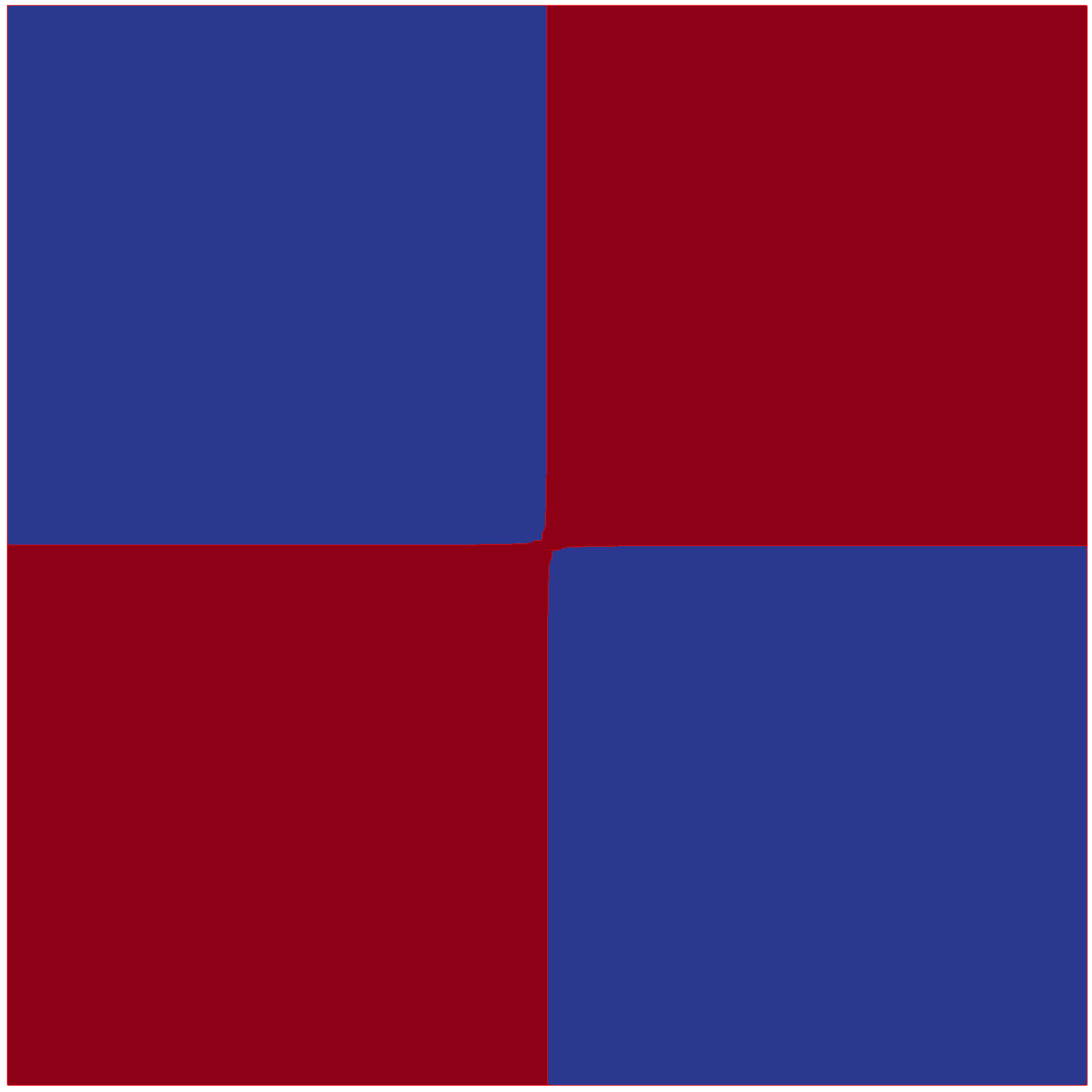}
{\small{(3.C)}}
\end{minipage}
\caption{Approximate control $\bar{\mathfrak{u}}_{\ell}$ obtained after $5$ (3.A), $10$ (3.B), and $15$ (3.C) iterations for the problem from section \ref{sec:ex_1};
in the red region the value is $1$ whereas in the blue region is $-1$.}
\label{fig:ex_1_3}
\end{figure}

In Figures \ref{fig:ex_1}, \ref{fig:ex_1_2}, and \ref{fig:ex_1_3} we display the results obtained for this example. 
We show, in Fig. \ref{fig:ex_1}, experimental rates of convergence for each contribution of the total error when uniform and adaptive refinement are considered.
We also present experimental rates of convergence for all the individual contributions of the error estimator $E$ (see \eqref{def:total_estimator}) and the effectivity index, when adaptive refinement is considered.
We observe that all the approximation errors obtained for both schemes exhibit optimal experimental rates of convergence (Figs. (1.A) and (1.B)); the same convergence rate is observed for the error estimator (Fig. (1.C)).
When the total number of degrees of freedom increases, we observe that the effectivity index stabilizes around the values $2$ and $4$ (Fig. (1.D)).
In Fig. \ref{fig:ex_1_2} we present adaptively refined meshes obtained after 5, 10, and 15 iterations.
We do not observe an explicit connection between the switching set and the adaptively refined meshes.
This is probably due to the fact that we do not have an error estimator accounting for such a set.
Even when the adaptive refinement is not necessarily concentrated near the discrete switching set, we observe that such a set seems to converge to the continuous switching set when the total number of degrees of freedom increases. 
Finally, in Fig. \ref{fig:ex_1_3}, we display the approximate optimal control $\bar{\mathfrak{u}}_{\ell}$ obtained after $5$, $10$, and $15$ iterations. 
We observe the classical bang-bang structure in the three approximations.


\subsection{Exact solution on non-convex domain}\label{sec:ex_2}
We set $\Omega=(-1,1)^2\setminus[0,1)\times(-1,0]$, $a=-1$, $b=1$, $\beta=1$, and take $f$ and $y_{\Omega}$ such that the exact optimal state and adjoint state are given, in polar coordinates $(\rho,\omega)$ with $\omega\in[0,3\pi/2]$, by
\begin{align*}\label{eq:sol_y_p_L}
\bar{y}(\rho,\omega)&=\sin(\pi(\rho\sin(\omega)+1)/2)\sin(\pi(\rho \cos(\omega)+1)/2)\rho^{2/3}\sin(2\omega/3), 
\\
\bar{p}(\rho,\omega)&= (0.5-\rho)\bar{y}(\rho,\omega).
\end{align*}
The purpose of this example is to investigate the performance of the devised a posteriori error estimator when we violate the convexity assumption considered on the domain.  


\begin{figure}[!ht]
\centering
\begin{minipage}[c]{0.45\textwidth}\centering
\includegraphics[trim={0 0 0 0},clip,width=5.5cm,height=5.3cm,scale=0.30]{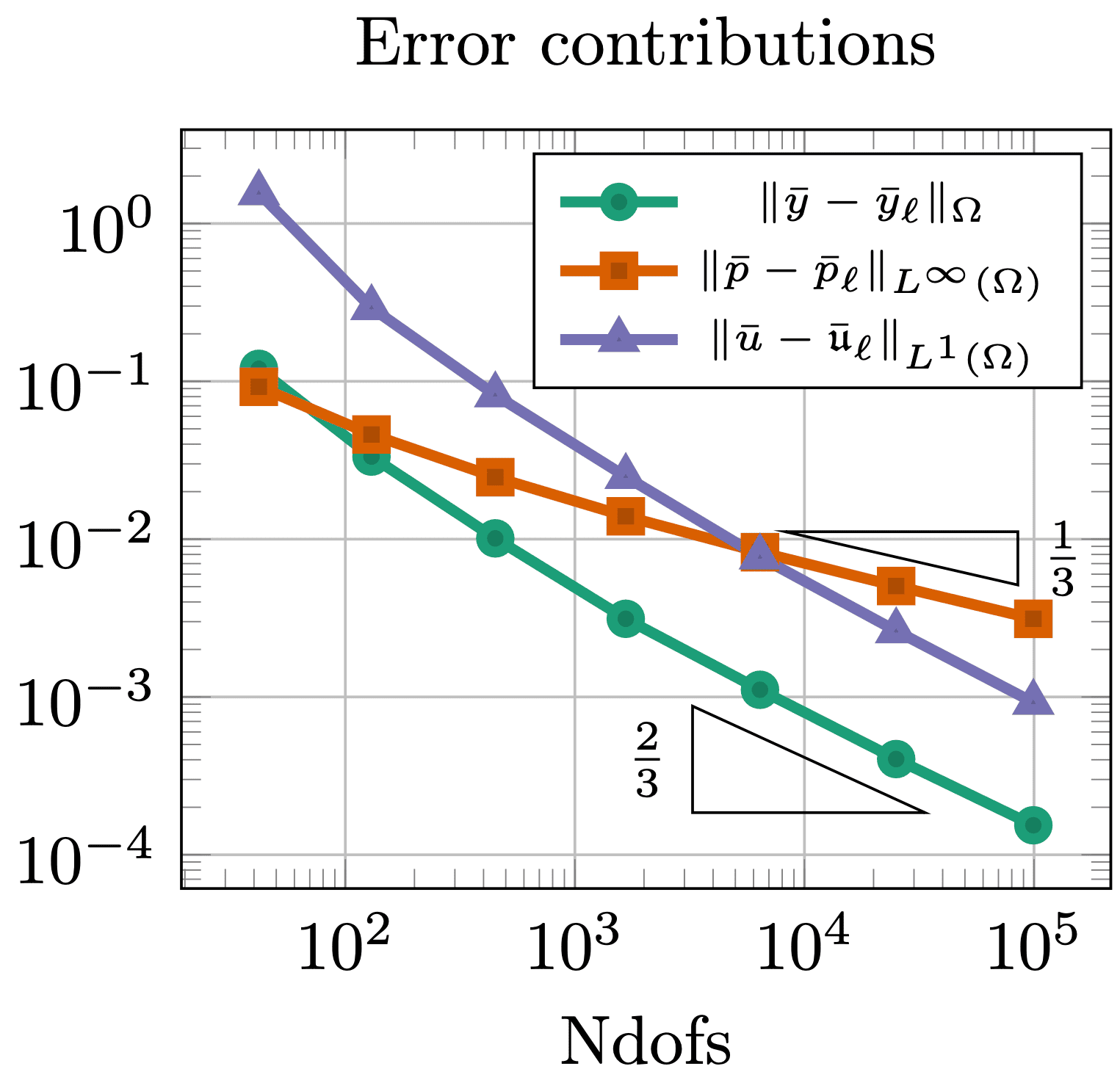}\\
\qquad
{\small{(4.A)}}
\end{minipage}
\begin{minipage}[c]{0.45\textwidth}\centering
\includegraphics[trim={0 0 0 0},clip,width=5.5cm,height=5.3cm,scale=0.30]{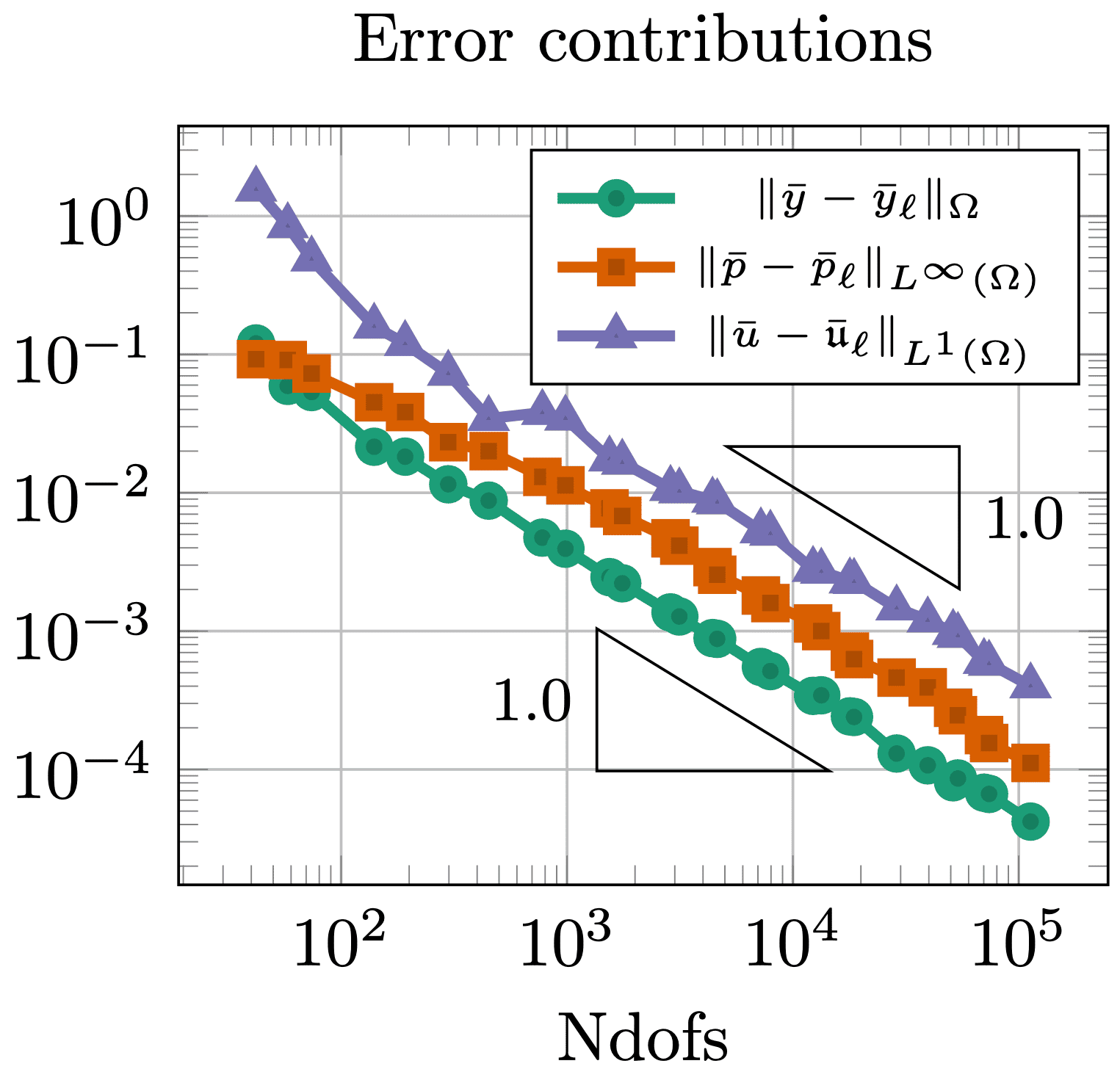}\\
\qquad
{\small{(4.B)}}
\end{minipage}
\\
\begin{minipage}[c]{0.45\textwidth}\centering
\includegraphics[trim={0 0 0 0},clip,width=5.5cm,height=5.3cm,scale=0.30]{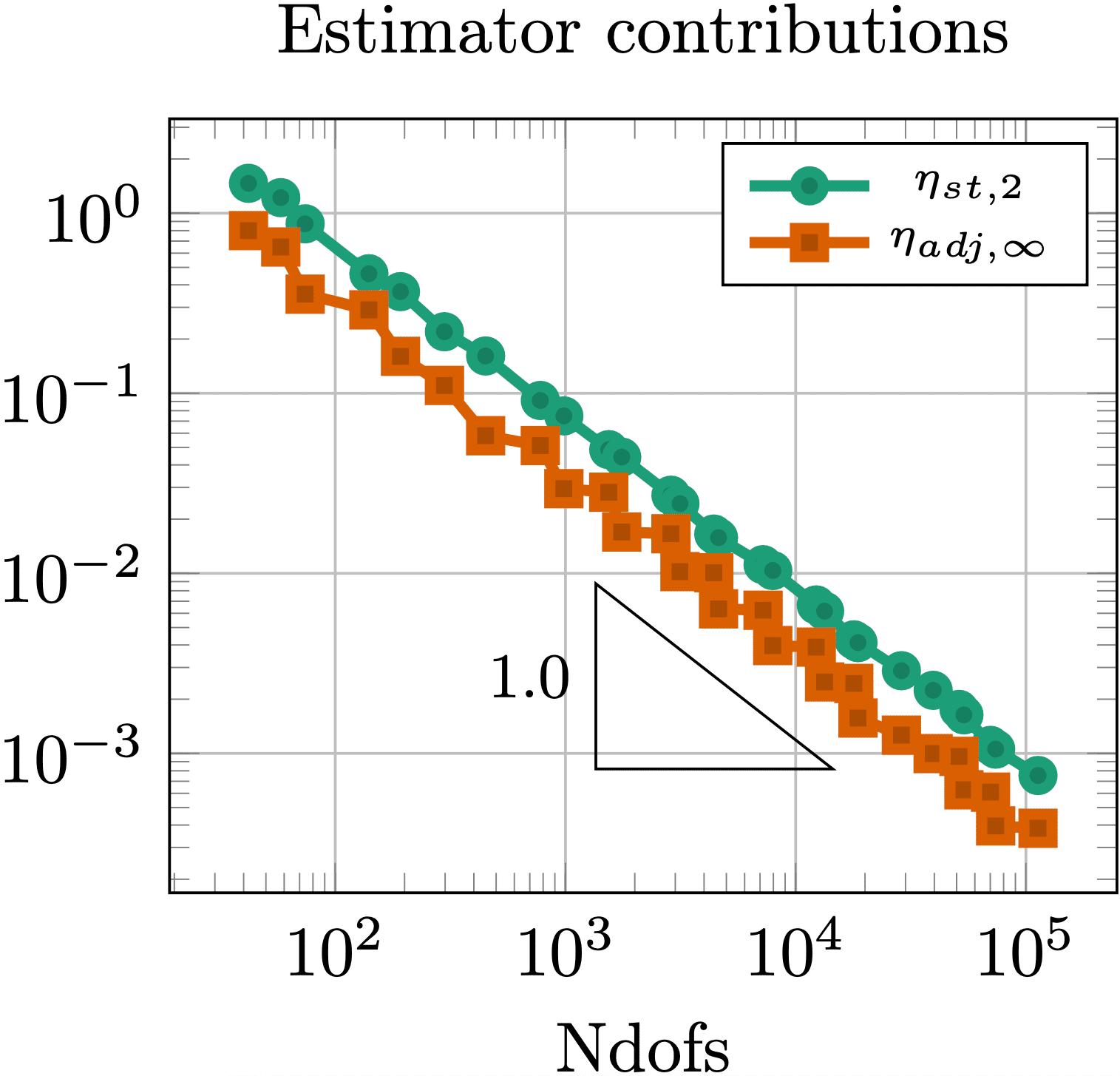}\\
\qquad
{\small{(4.C)}}
\end{minipage}
\begin{minipage}[c]{0.45\textwidth}\centering
\includegraphics[trim={0 0 0 0},clip,width=5.5cm,height=5.3cm,scale=0.30]{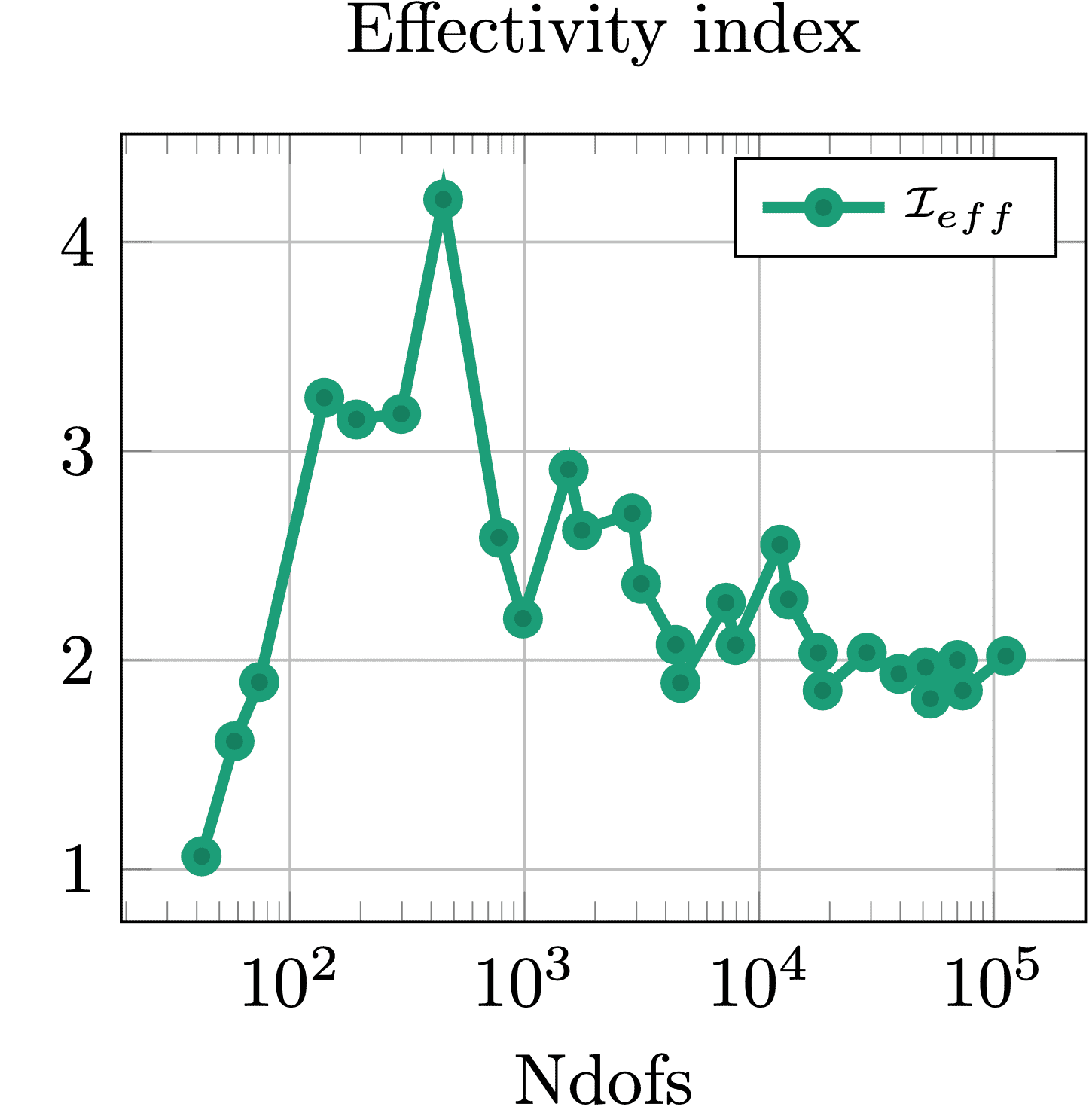}\\
\qquad
{\small{(4.D)}}
\end{minipage}
 \caption{Experimental rates of convergence for individual contributions of the total error with uniform (4.A) and adaptive (4.B) refinement, convergence rates for individual contributions of the estimator $E$ (4.C), and effectivity index (4.D) with adaptive refinement for the problem from section \ref{sec:ex_2}.}
\label{fig:ex_2}
\end{figure}


\begin{figure}[!ht]
\begin{minipage}[c]{0.32\textwidth}\centering
\includegraphics[trim={0 0 0 0},clip,width=4.0cm,height=4.0cm,scale=0.30]{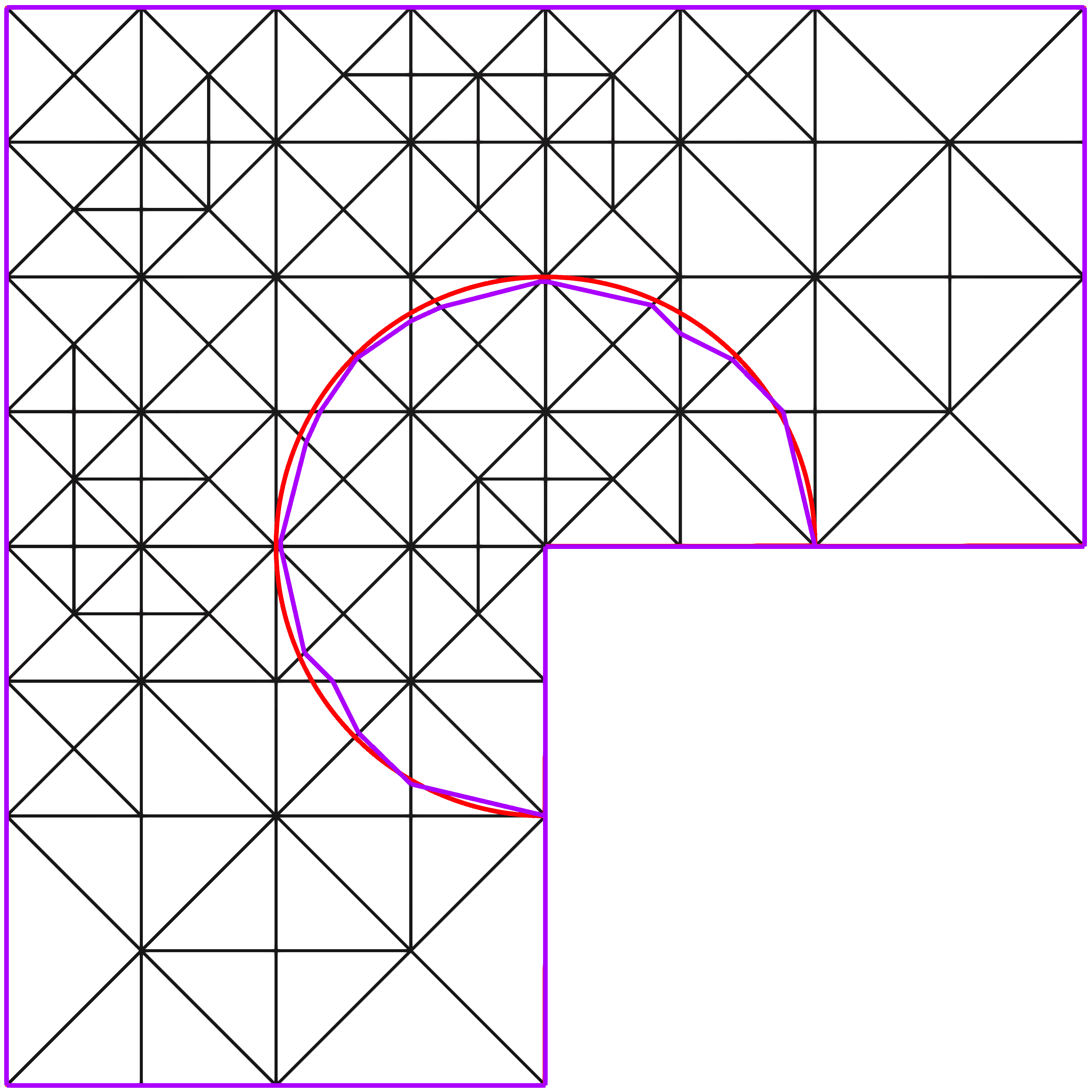}\\
{\small{(5.A)}}
\end{minipage}
\begin{minipage}[c]{0.32\textwidth}\centering
\includegraphics[trim={0 0 0 0},clip,width=4.0cm,height=4.0cm,scale=0.30]{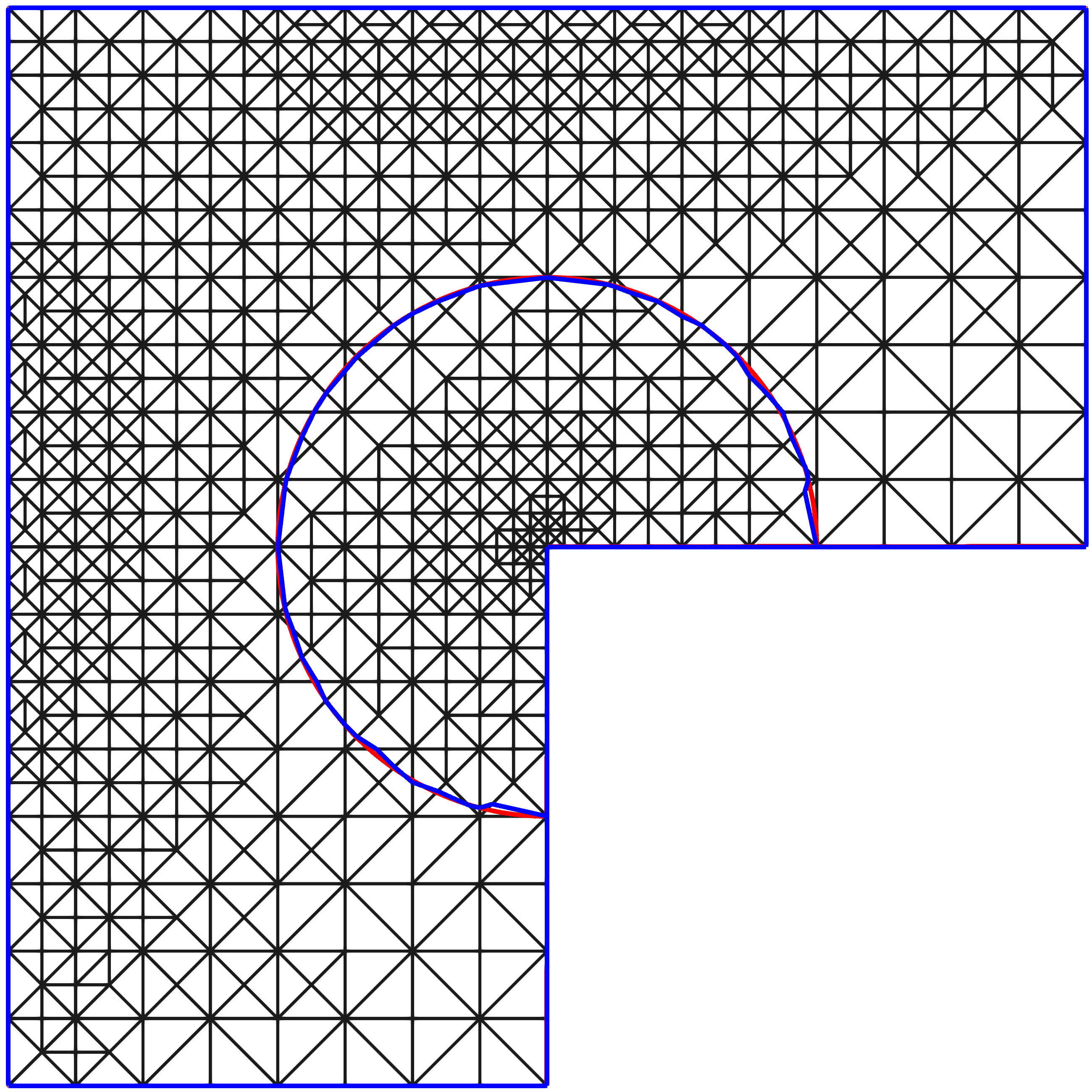}\\
{\small{(5.B)}}
\end{minipage}
\begin{minipage}[c]{0.32\textwidth}\centering
\includegraphics[trim={0 0 0 0},clip,width=4.0cm,height=4.0cm,scale=0.30]{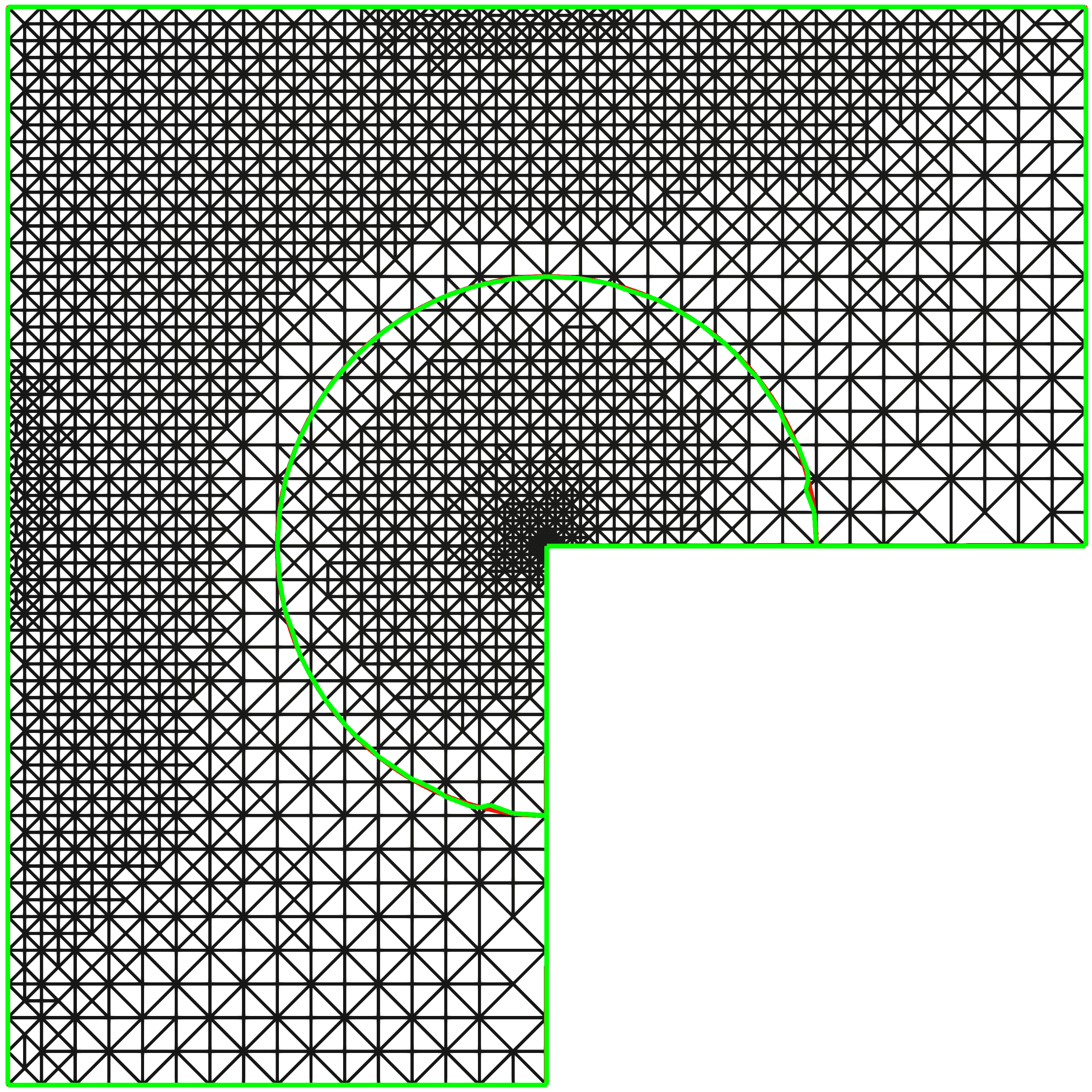}\\
{\small{(5.C)}}
\end{minipage}
\\
\begin{minipage}[c]{0.32\textwidth}\centering
\includegraphics[trim={0 0 0 0},clip,width=4.0cm,height=4.0cm,scale=0.30]{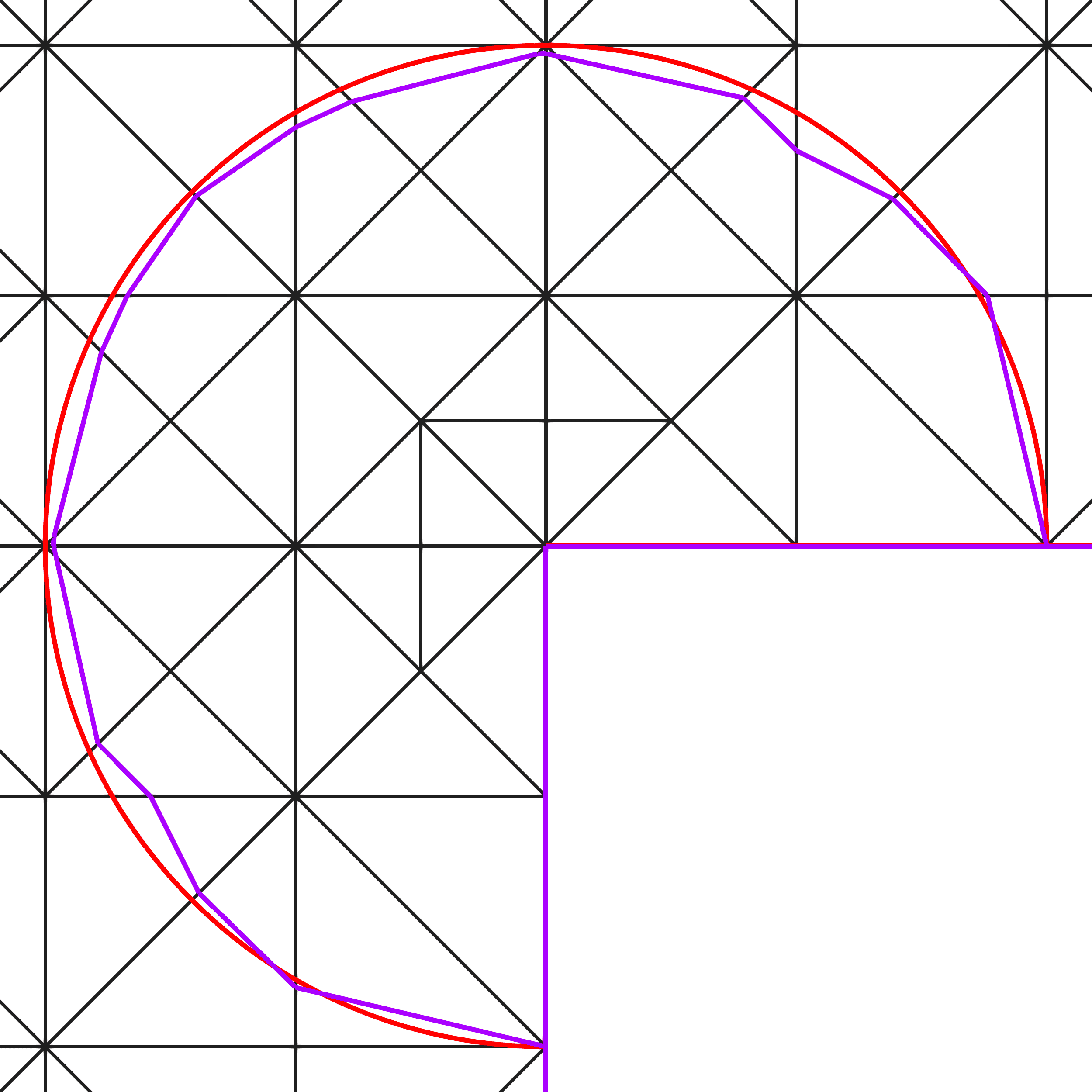}\\
{\small{(5.D)}}
\end{minipage}
\begin{minipage}[c]{0.32\textwidth}\centering
\includegraphics[trim={0 0 0 0},clip,width=4.0cm,height=4.0cm,scale=0.30]{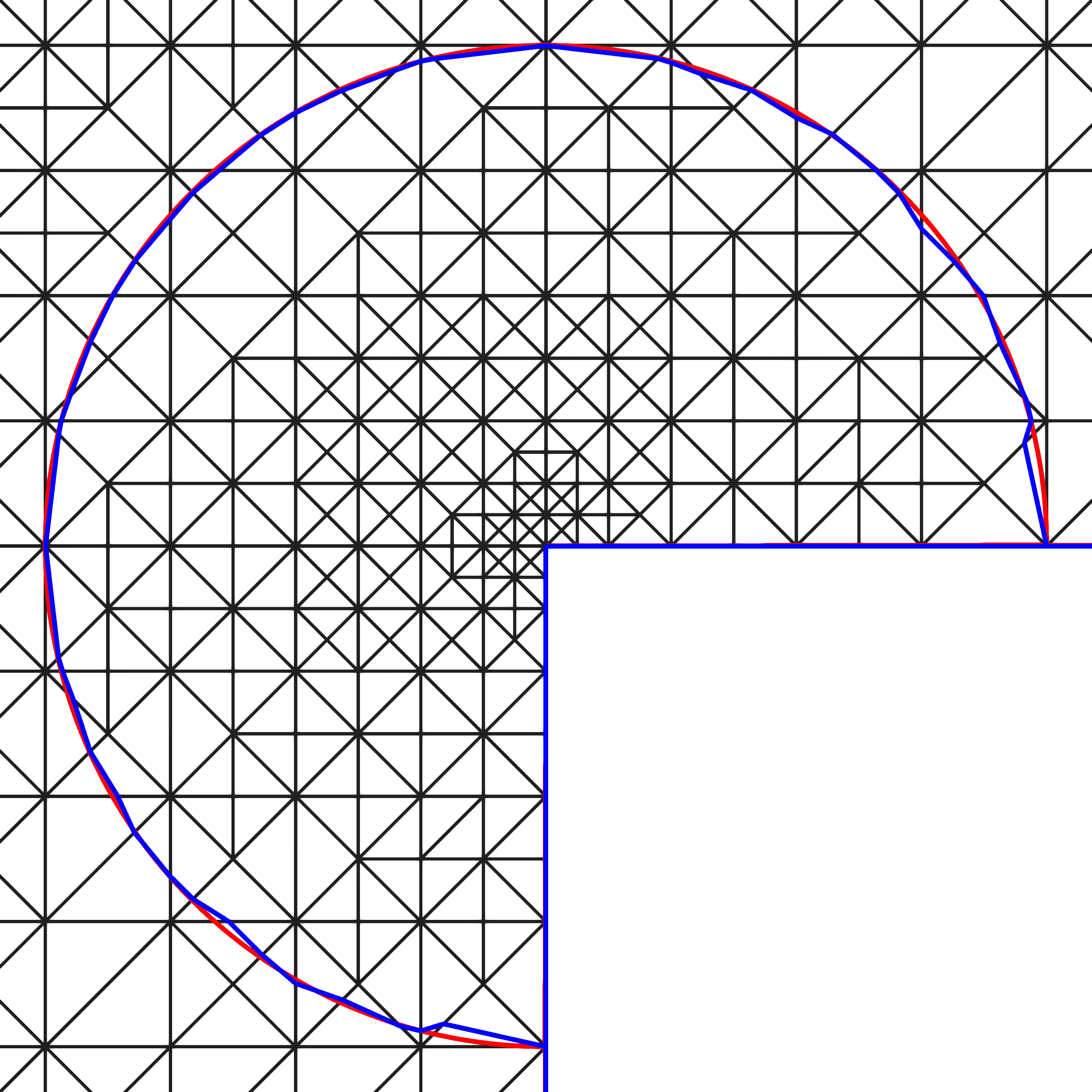}\\
{\small{(5.E)}}
\end{minipage}
\begin{minipage}[c]{0.32\textwidth}\centering
\includegraphics[trim={0 0 0 0},clip,width=4.0cm,height=4.0cm,scale=0.30]{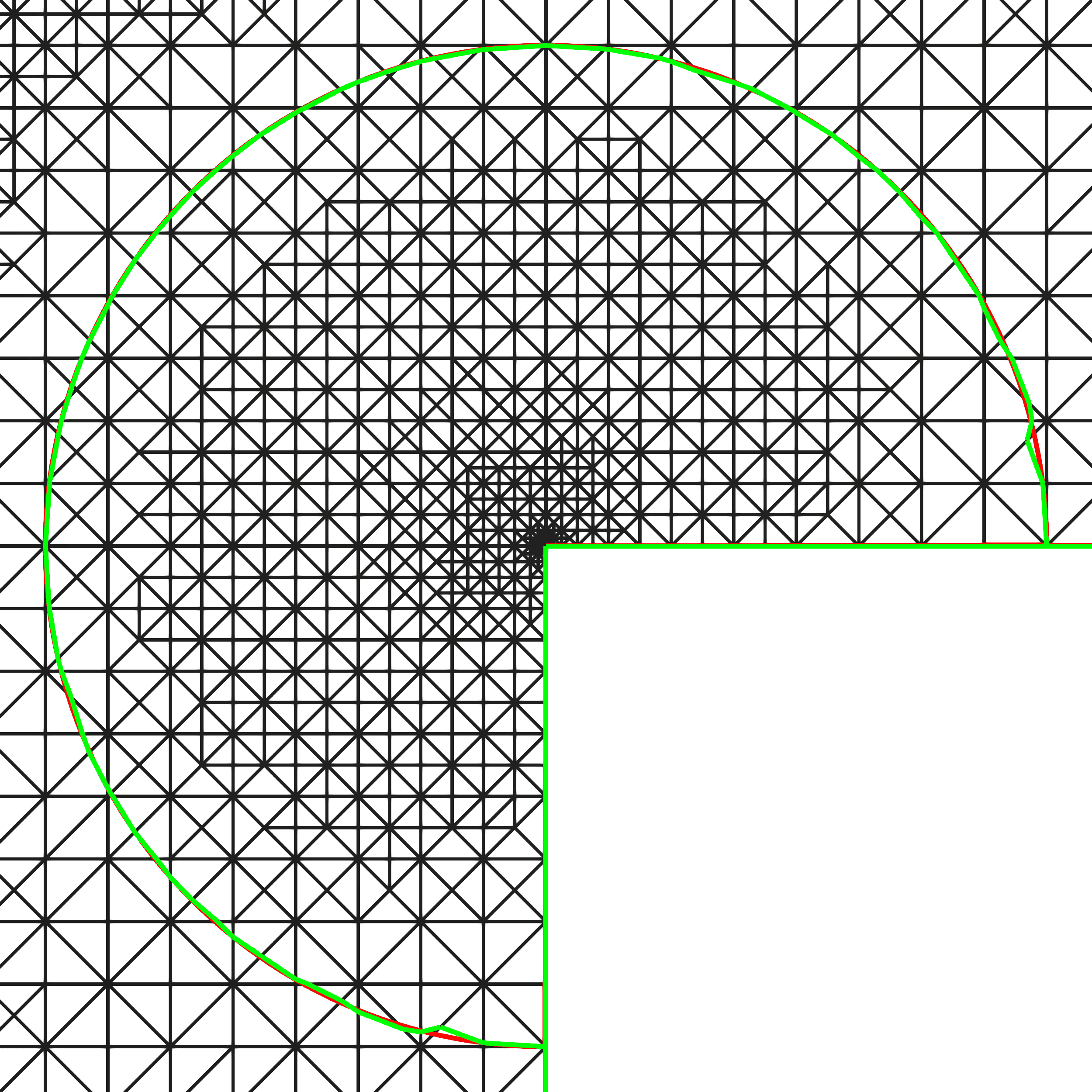}\\
{\small{(5.F)}}
\end{minipage}
\caption{Comparison of the continuous (red) and discrete switching sets on the adaptively refined meshes obtained after $5$ ((5.A) and (5.D)), $10$ ((5.B) and (5.E)), and $15$ ((5.C) and (5.F)) iterations for the problem from section \ref{sec:ex_2}.}
\label{fig:ex_2_1}
\end{figure}                


\begin{figure}[!ht]
\centering
\begin{minipage}[c]{0.42\textwidth}\centering
\includegraphics[trim={0 0 0 0},clip,width=5cm,height=2.5cm,scale=0.30]{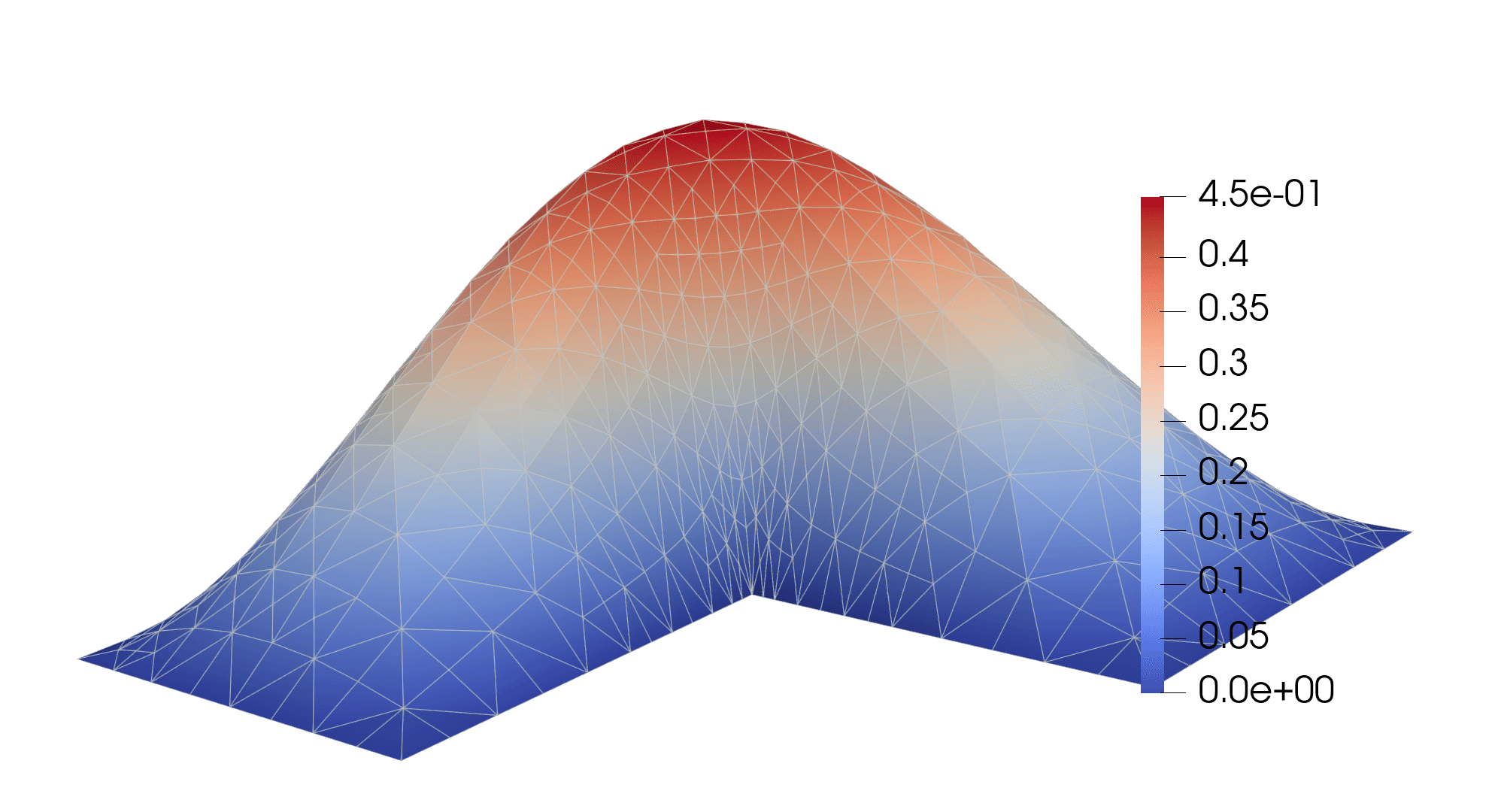}
{\small{(6.A)}}
\end{minipage}
\begin{minipage}[c]{0.42\textwidth}\centering
\includegraphics[trim={0 0 0 0},clip,width=5cm,height=2.3cm,scale=0.30]{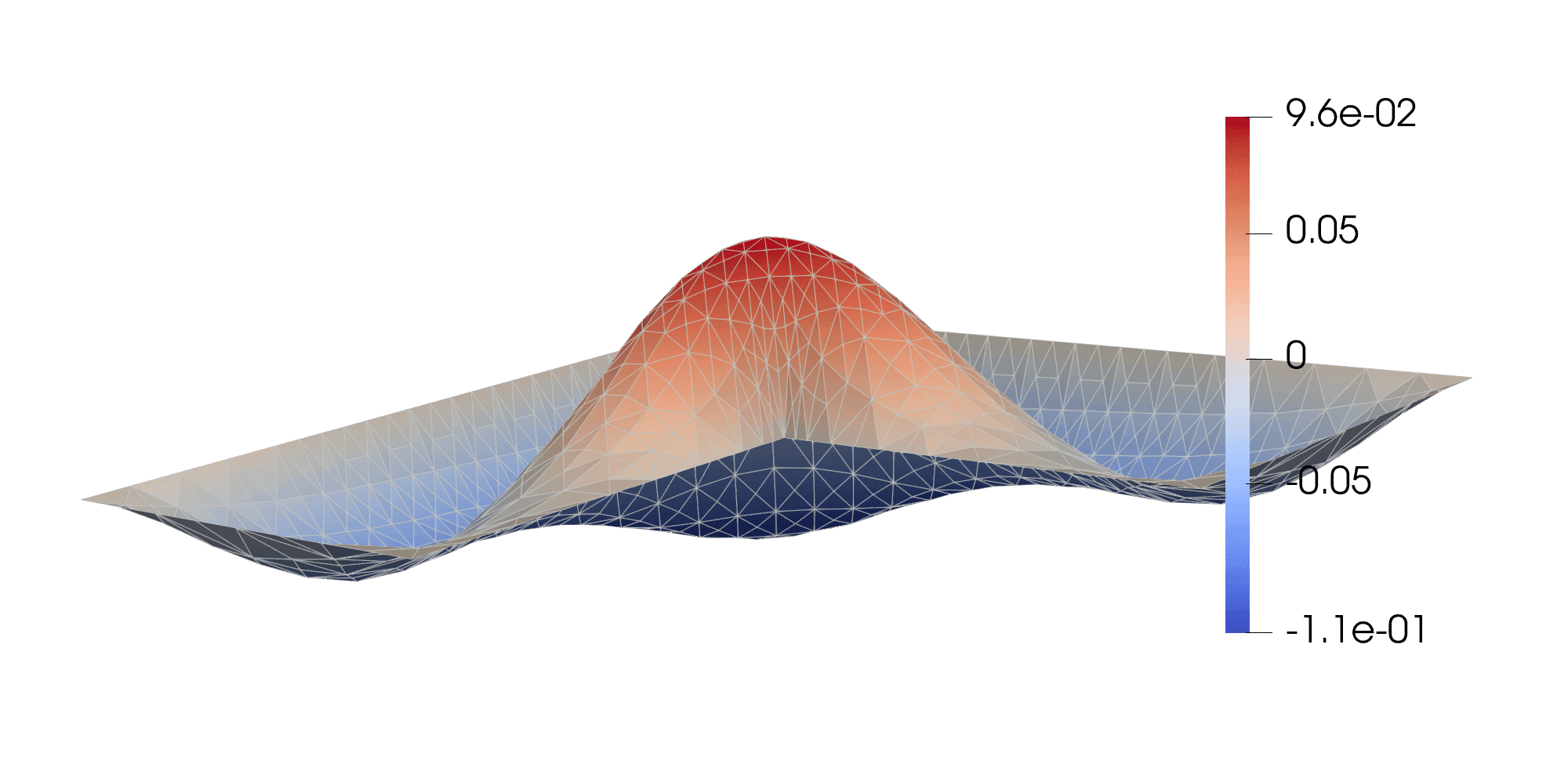}
{\small{(6.B)}}
\end{minipage}
\\
\begin{minipage}[c]{0.32\textwidth}\centering
\includegraphics[trim={0 0 0 0},clip,width=4.0cm,height=4.0cm,scale=0.30]{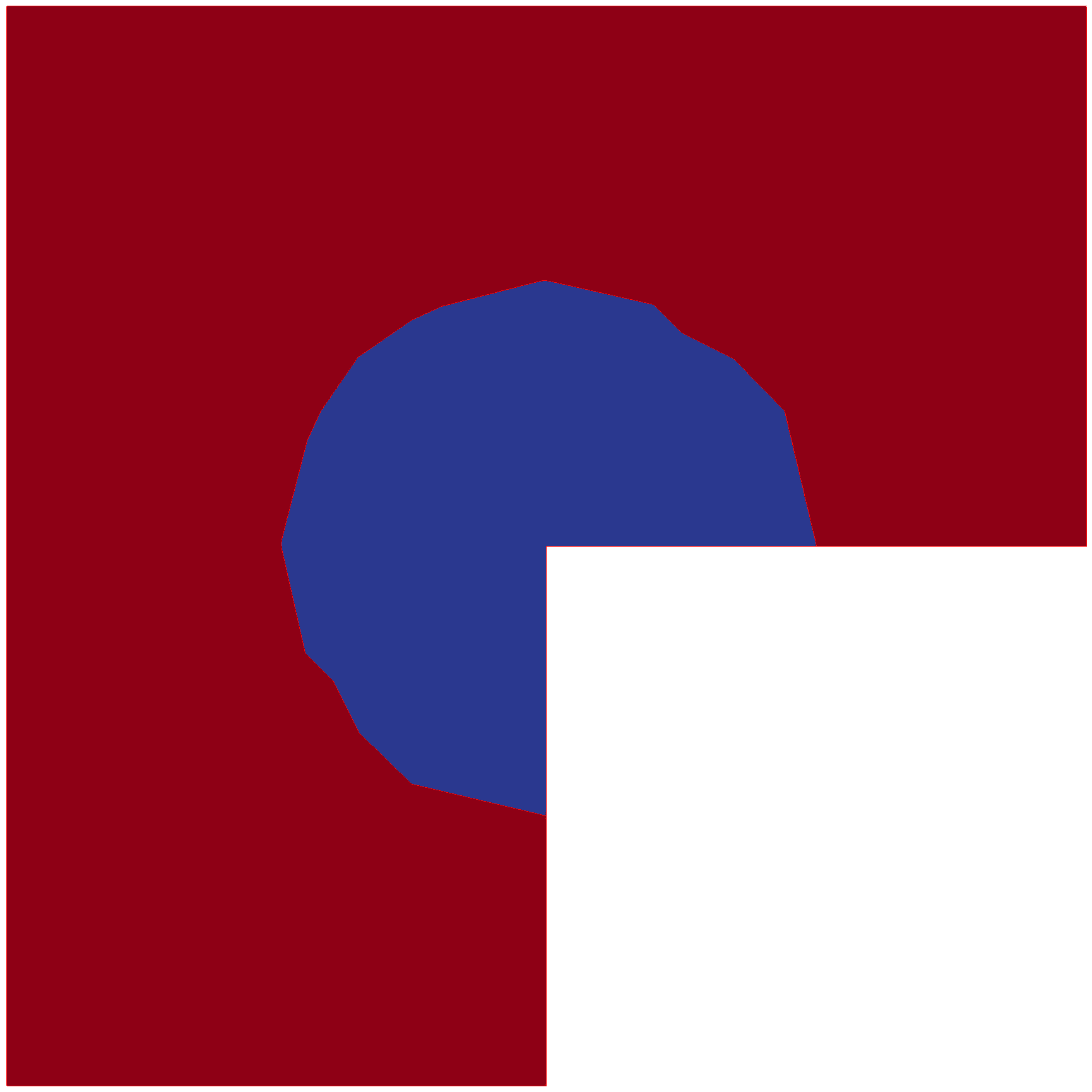}
{\small{(6.C)}}
\end{minipage}
\begin{minipage}[c]{0.32\textwidth}\centering
\includegraphics[trim={0 0 0 0},clip,width=4.0cm,height=4.0cm,scale=0.30]{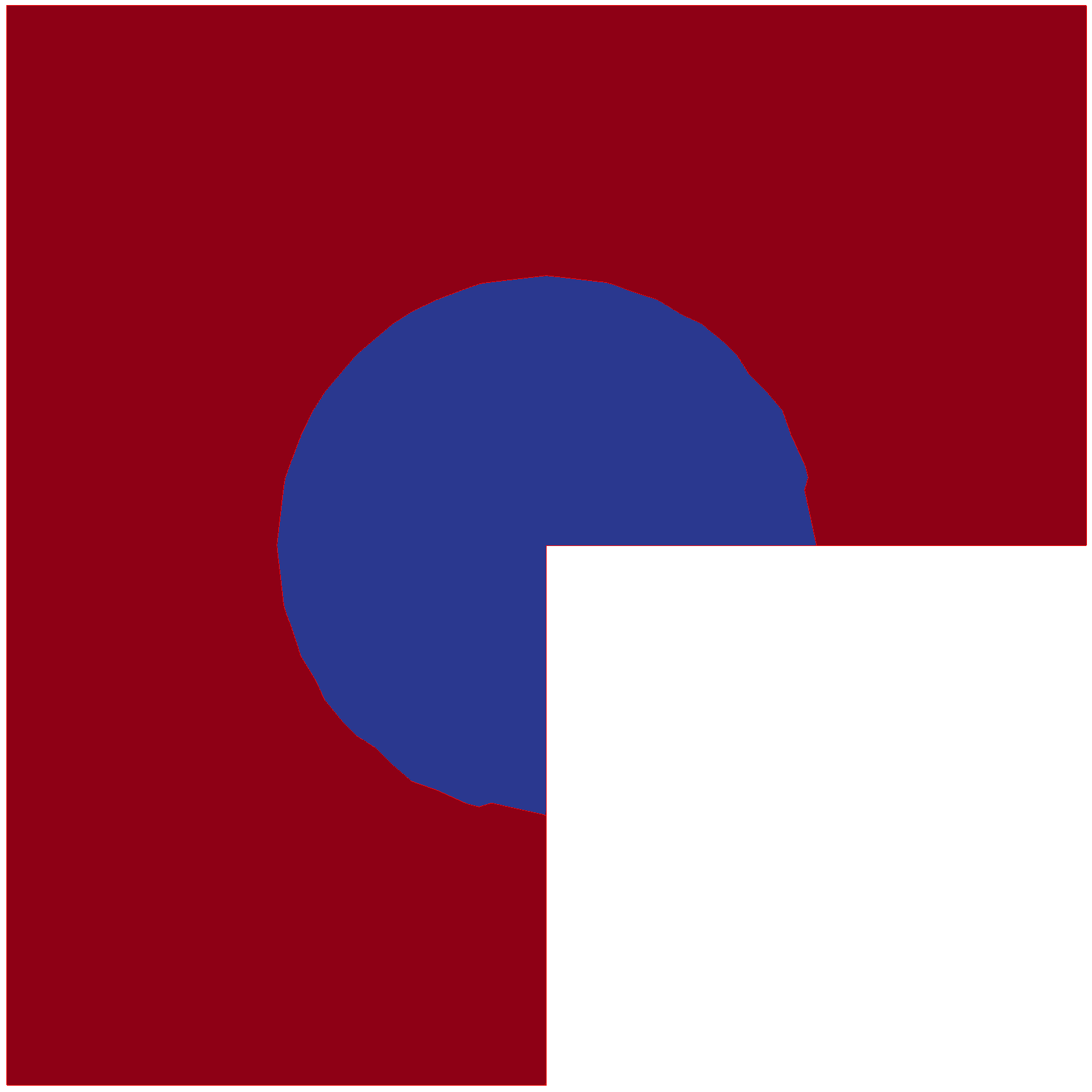}
{\small{(6.D)}}
\end{minipage}
\begin{minipage}[c]{0.32\textwidth}\centering
\includegraphics[trim={0 0 0 0},clip,width=4.0cm,height=4.0cm,scale=0.30]{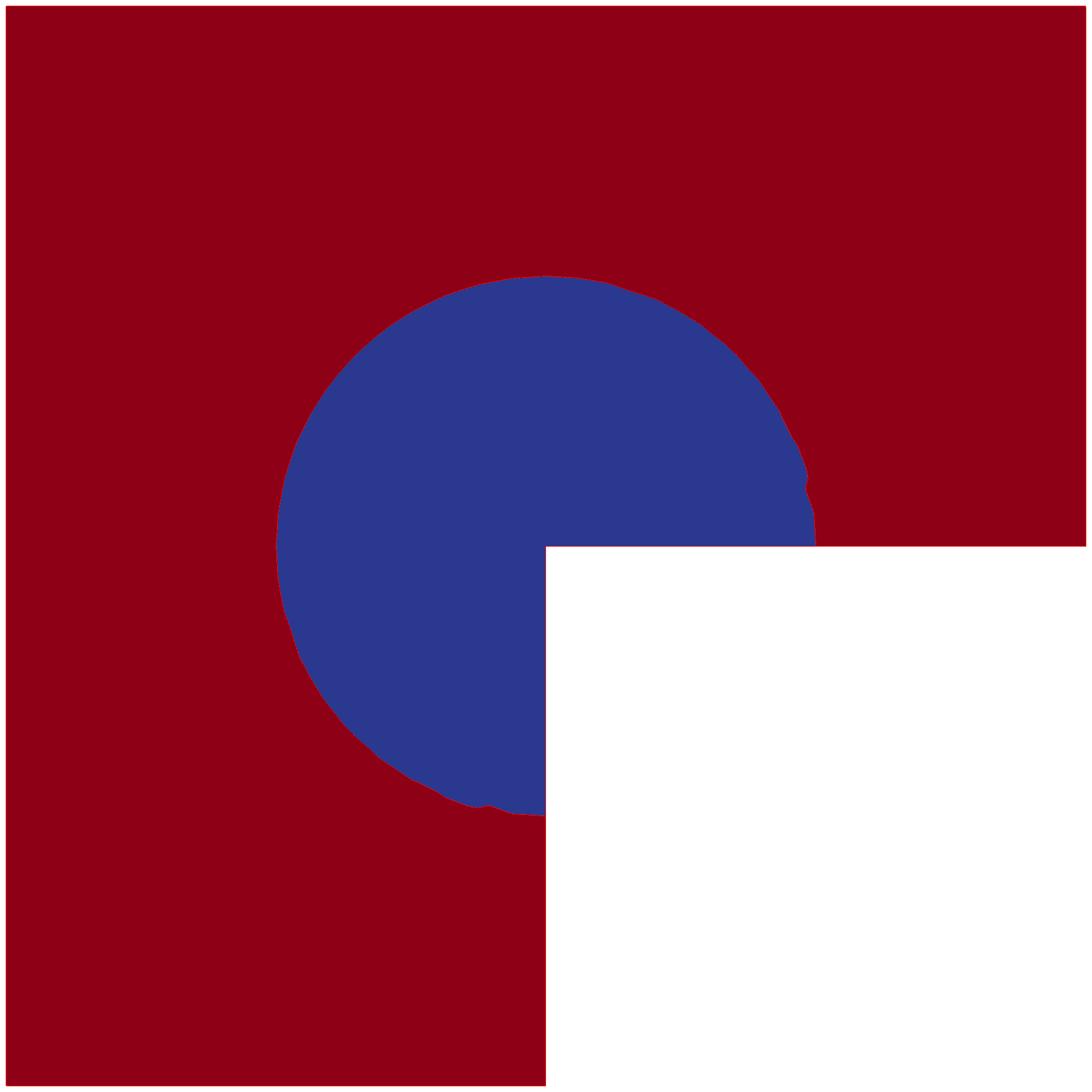}
{\small{(6.E)}}
\end{minipage}
\caption{Approximate solutions $\bar{y}_{\ell}$ (6.A) and $\bar{p}_{\ell}$ (6.B) obtained after $10$ iterations, and approximate control $\bar{\mathfrak{u}}_{\ell}$ obtained after $5$ (6.C), $10$ (6.D), and $15$ (6.E) iterations for the problem from section \ref{sec:ex_1};
in the red region the value is $1$ whereas in the blue region is $-1$.}
\label{fig:ex_2_2}
\end{figure}

We present the results obtained for this example in Figures \ref{fig:ex_2}, \ref{fig:ex_2_1}, and \ref{fig:ex_2_2}.
In Fig. \ref{fig:ex_2}, we display experimental rates of convergence for each contribution of the total error when uniform and adaptive refinement are considered, experimental rates of convergence for all the individual contributions of the error estimator $E$ (see \eqref{def:total_estimator}), and the effectivity index when adaptive refinement is considered.
We observe that the designed adaptive procedure outperforms uniform refinement.
In particular, it exhibits optimal experimental rates of convergence for each contribution of the total error and the error estimator (Figs. (4.B) and (4.C)).
We observe that the effectivity seems to stabilize around the values $2$ when the total number of degrees of freedom increases  (Fig. (4.D)).
In Fig. \ref{fig:ex_2_2} we present adaptively refined meshes obtained after 5, 10, and 15 iterations.
It can be observed that the refinement is being concentrated at the re-entrant corner $(0, 0)$, as well as in the upper and left regions of the domain;
an explicit connection between the switching set and the adaptively refined meshes is not observed.
The refinement produced in the upper and left regions of the domain may stem from the structure of the optimal state and adjoint state near these regions; see Fig. \ref{fig:ex_2_2}.
We also observe that the discrete switching set seems to converge to the switching set of the continuous solution when the total number of degrees of freedom increases.
Approximate optimal solutions are displayed in Fig. \ref{fig:ex_2_2}.
The approximate optimal control $\bar{\mathfrak{u}}_{\ell}$ exhibits the bang-bang structure in the three approximations (Figs. (6.C), (6.D), and (6.E)).


\subsection{Unknown solution on non-convex domain}\label{sec:ex_3}

We set $\Omega=(-1,1)^2\setminus[0,1)\times(-1,0]$, $a=-1$, $b=1$, and data
\begin{align*}
y_{\Omega}(x_{1},x_{2}) = \frac{1}{\sqrt[4]{x_{1}^2 + x_{2}^2}} - 10\sin(x_{1}x_{2}), \qquad (x_{1},x_{2})\in \Omega.
\end{align*}
We note that $y_{\Omega}\in L^{2}(\Omega) \setminus L^{\infty}(\Omega)$. 
The purpose of this example is to investigate the influence of $\beta$ (see \eqref{eq:assumption_S}) in the error indicator \eqref{def:total_indicator} within the adaptive procedure.
We consider $\beta\in\{0.2,0.4,0.6,0.8,1.0\}$.

In Figs. \ref{fig:ex_3} and \ref{fig:ex_3_1} we show the results obtained for this example. 
In Fig. \ref{fig:ex_3} we consider the case when $\beta=1.0$.
Similar conclusions to the ones presented for the example from section \ref{sec:ex_2} can be derived from Fig. \ref{fig:ex_3}.
Particularly, we observe optimal experimental rates of convergence for all the individual contributions of the error estimator $E$ within the adaptive loop.
In Fig.  \ref{fig:ex_3_1} we show experimental rates of convergence for all the individual contributions of the error estimator $E$ with different values of $\beta$.
We observe that, when the value of $\beta$ decreases more degrees of freedom are needed to attain the optimal rate of convergence.


\begin{figure}[!ht]
\centering
\begin{minipage}[c]{0.42\textwidth}\centering
\includegraphics[trim={0 0 0 0},clip,width=5cm,height=2.4cm,scale=0.30]{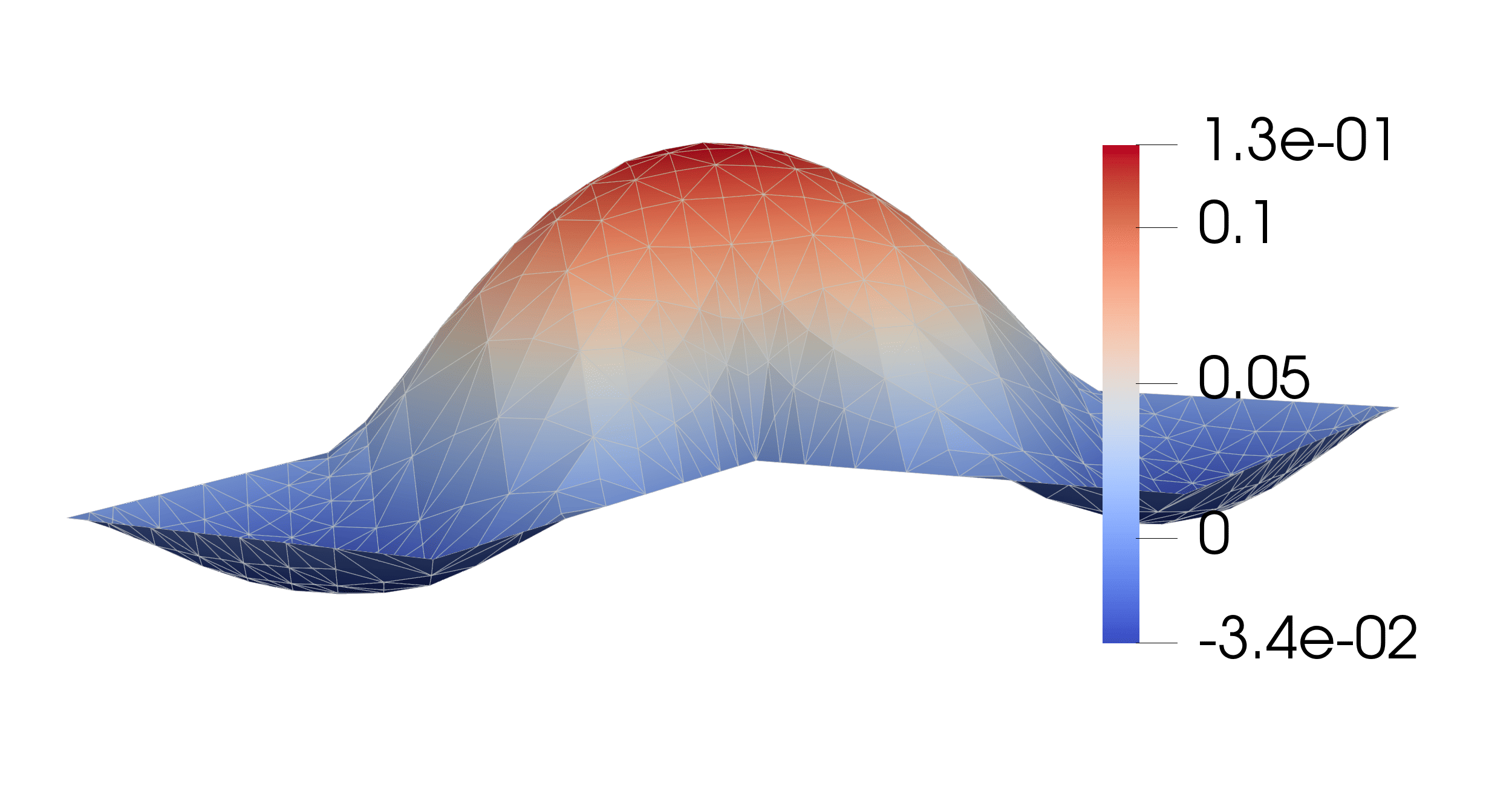}
{\small{(7.A)}}
\end{minipage}
\begin{minipage}[c]{0.42\textwidth}\centering
\includegraphics[trim={0 0 0 0},clip,width=5cm,height=2.6cm,scale=0.30]{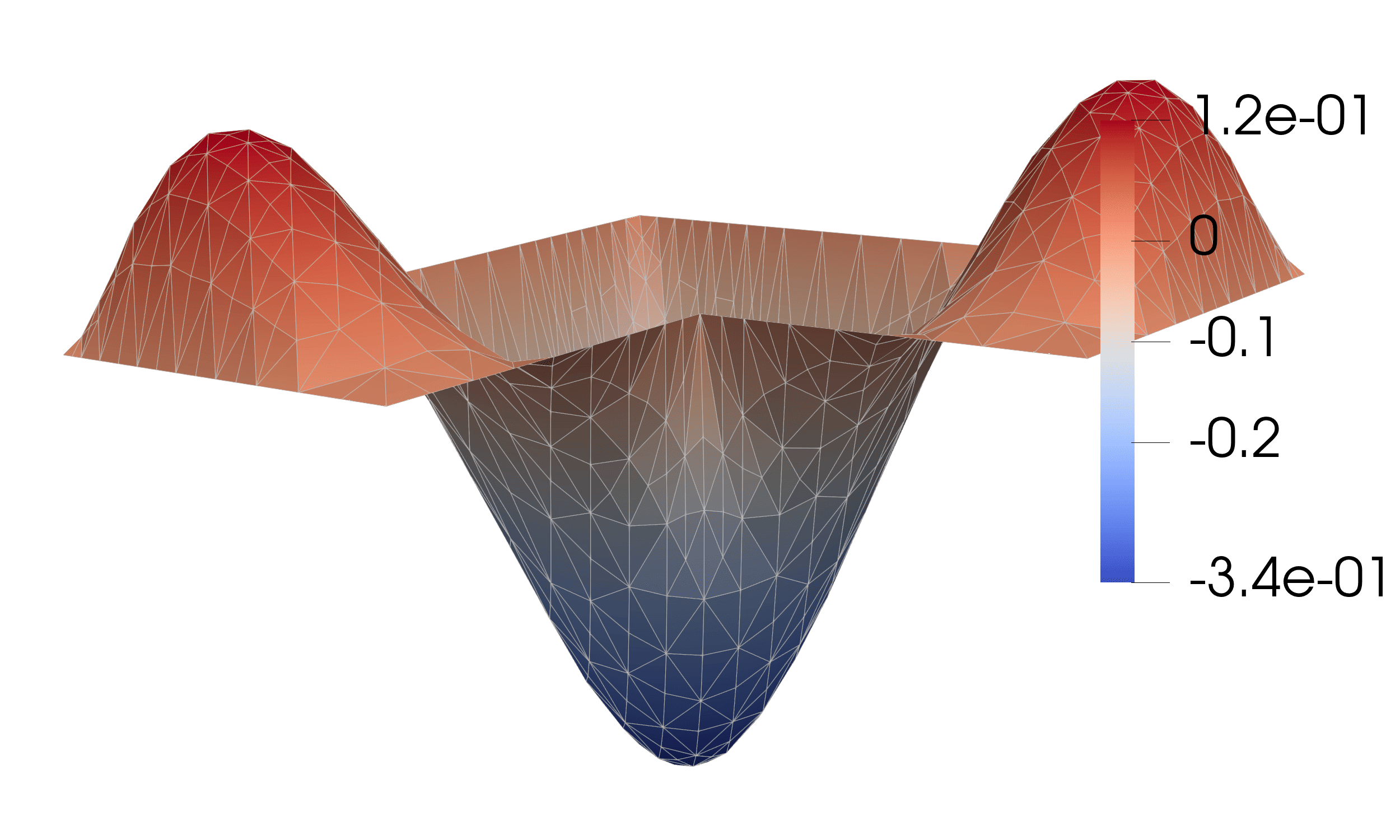}
{\small{(7.B)}}
\end{minipage}
\\
\begin{minipage}[c]{0.32\textwidth}\centering
\includegraphics[trim={0 0 0 0},clip,width=4.0cm,height=4.0cm,scale=0.30]{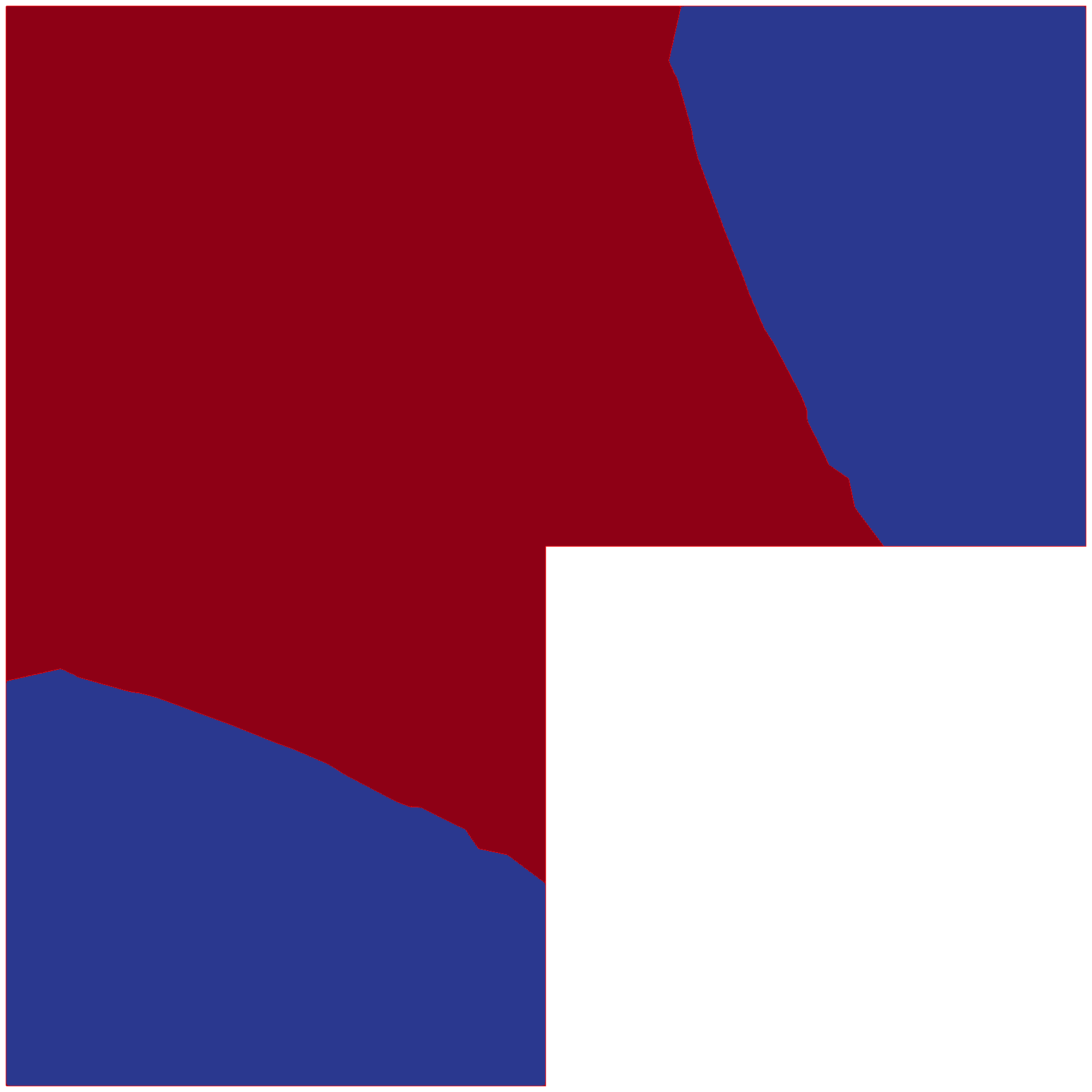}
{\small{(7.C)}}
\end{minipage}
\begin{minipage}[c]{0.32\textwidth}\centering
\includegraphics[trim={0 0 0 0},clip,width=4.0cm,height=4.0cm,scale=0.30]{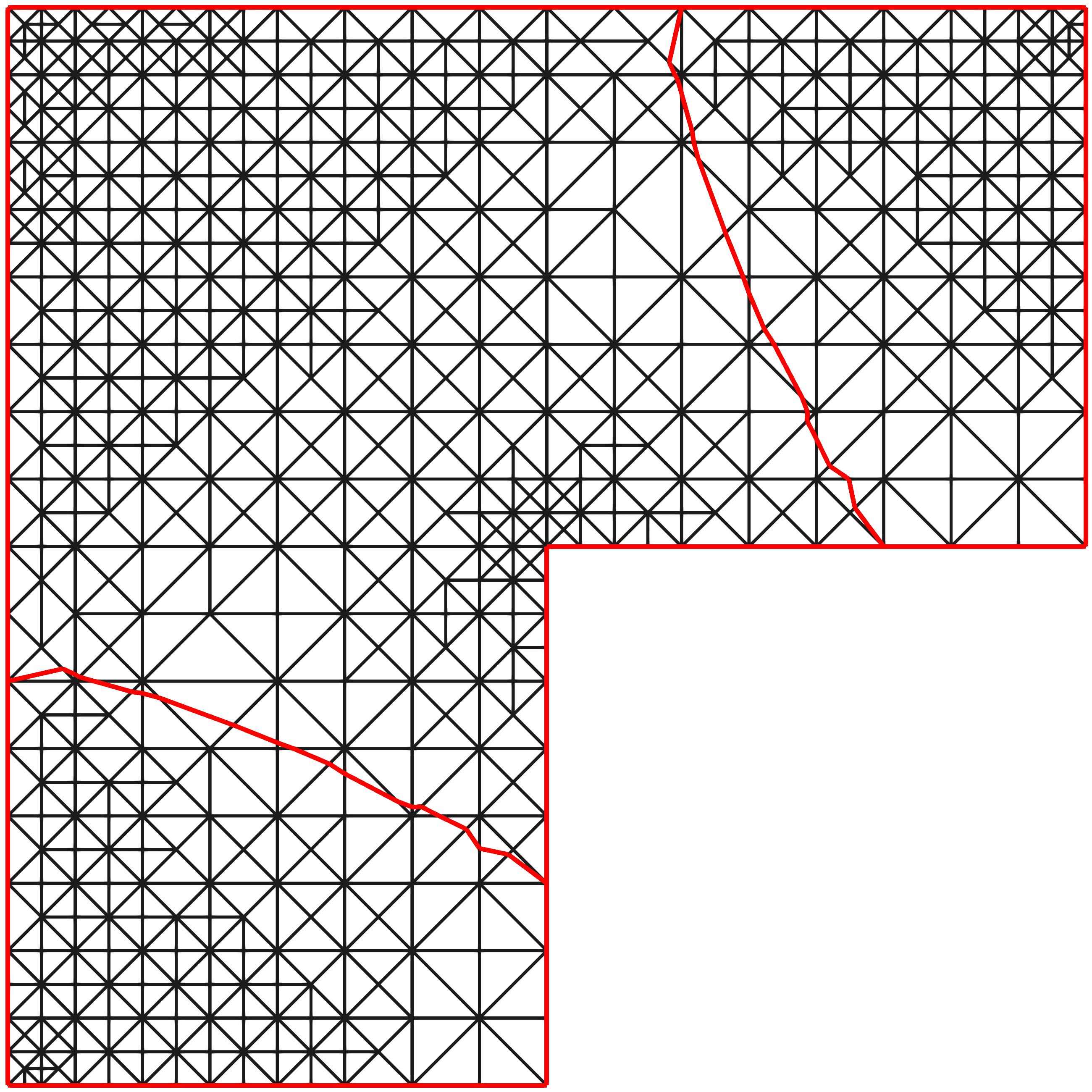}
{\small{(7.D)}}
\end{minipage}
\begin{minipage}[c]{0.32\textwidth}\centering
\includegraphics[trim={0 0 0 0},clip,width=4.3cm,height=4.2cm,scale=0.30]{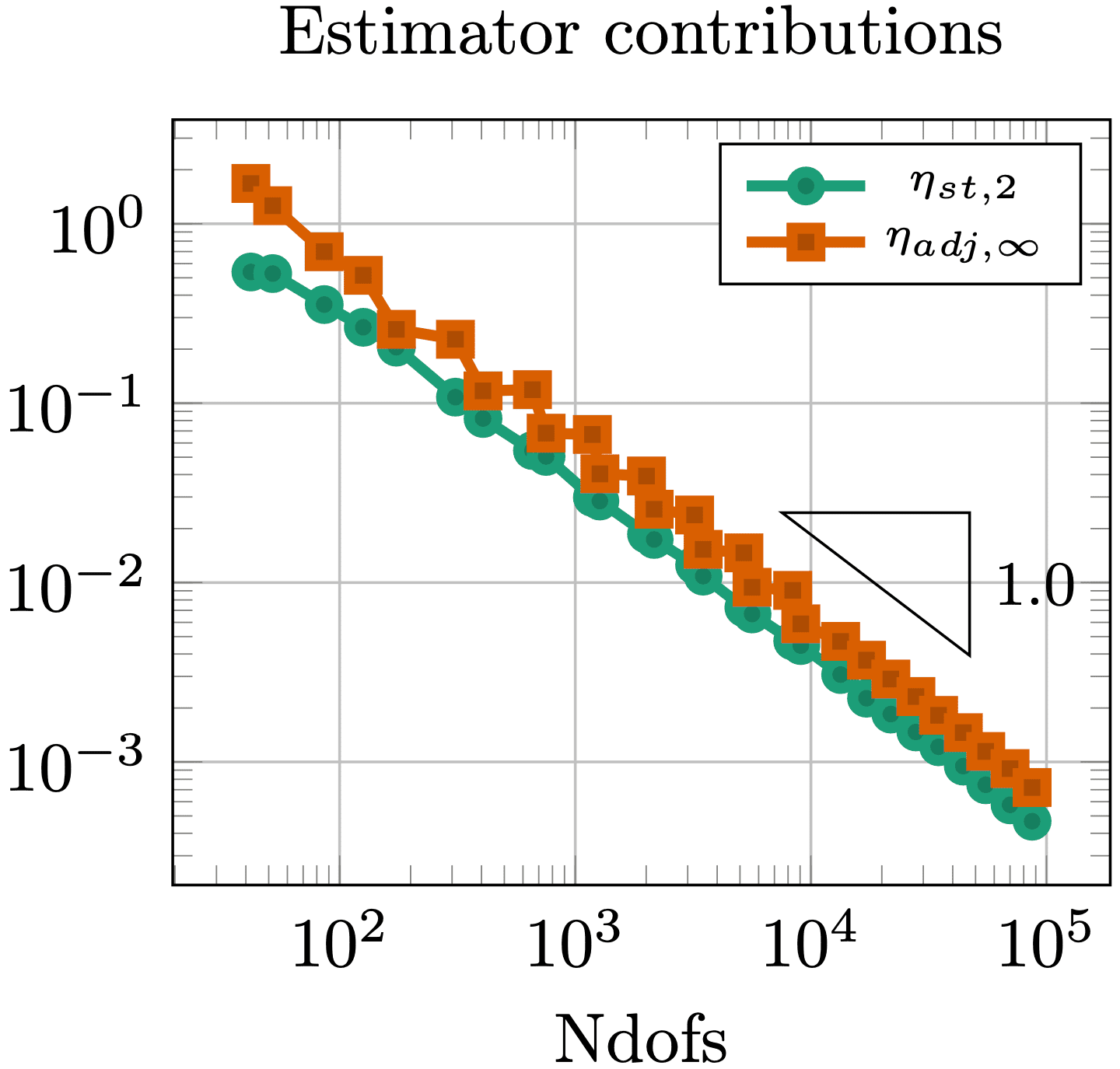}
{\small{(7.E)}}
\end{minipage}
\caption{Approximate solutions $\bar{y}_{\ell}$ (7.A), $\bar{p}_{\ell}$ (7.B), and $\bar{\mathfrak{u}}_{\ell}$ (7.C), and adaptively refined mesh (7.D) obtained after $10$ iterations.
Experimental rates of convergence for individual contributions of the estimator $E$ with $\beta=1.0$ (7.E) for the problem from section \ref{sec:ex_3}.}
\label{fig:ex_3}
\end{figure}


\begin{figure}[!ht]
\centering
\begin{minipage}[c]{0.45\textwidth}\centering
\includegraphics[trim={0 0 0 0},clip,width=5.5cm,height=5.3cm,scale=0.30]{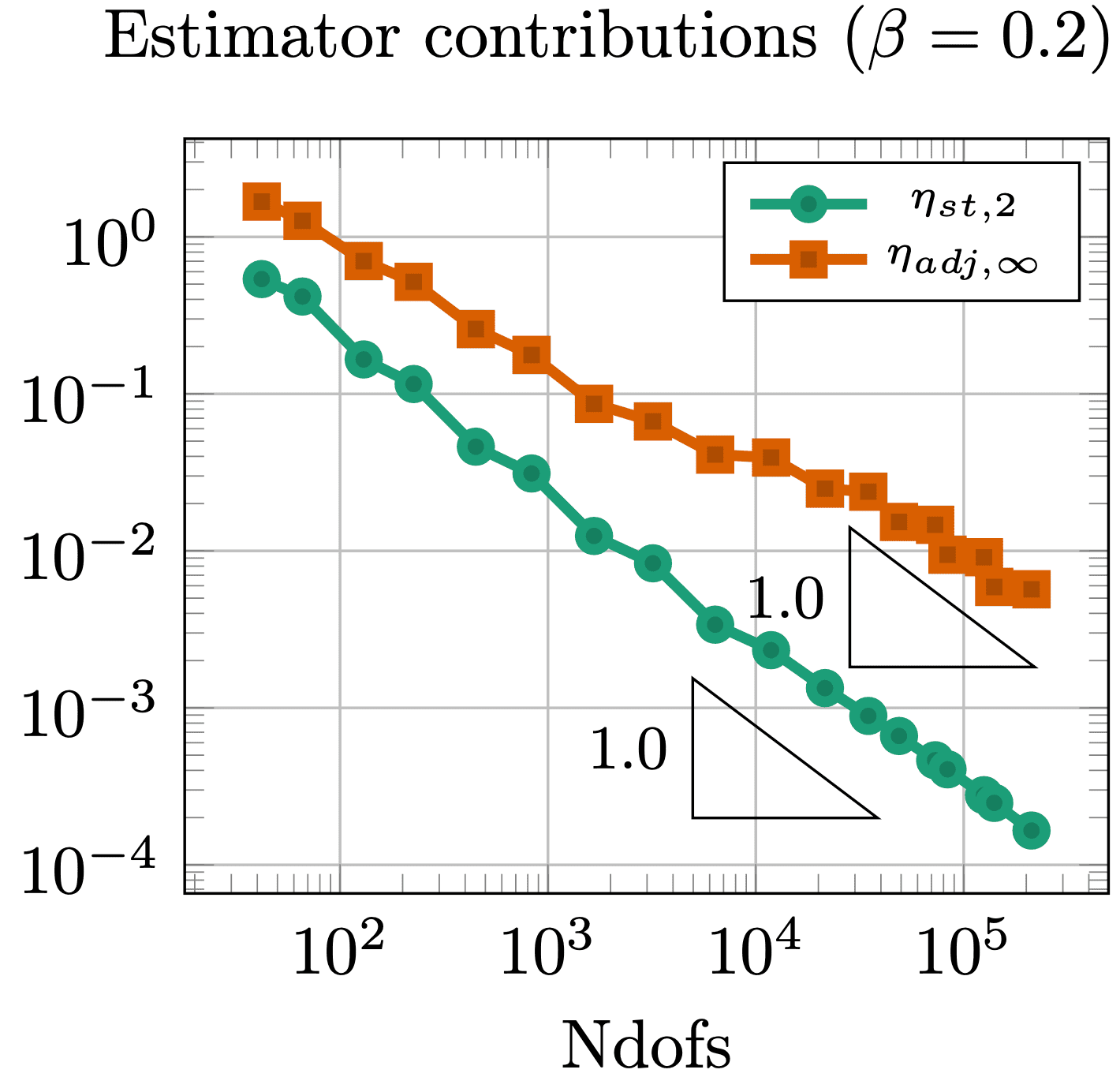}\\
\qquad
{\small{(8.A)}}
\end{minipage}
\begin{minipage}[c]{0.45\textwidth}\centering
\includegraphics[trim={0 0 0 0},clip,width=5.5cm,height=5.3cm,scale=0.30]{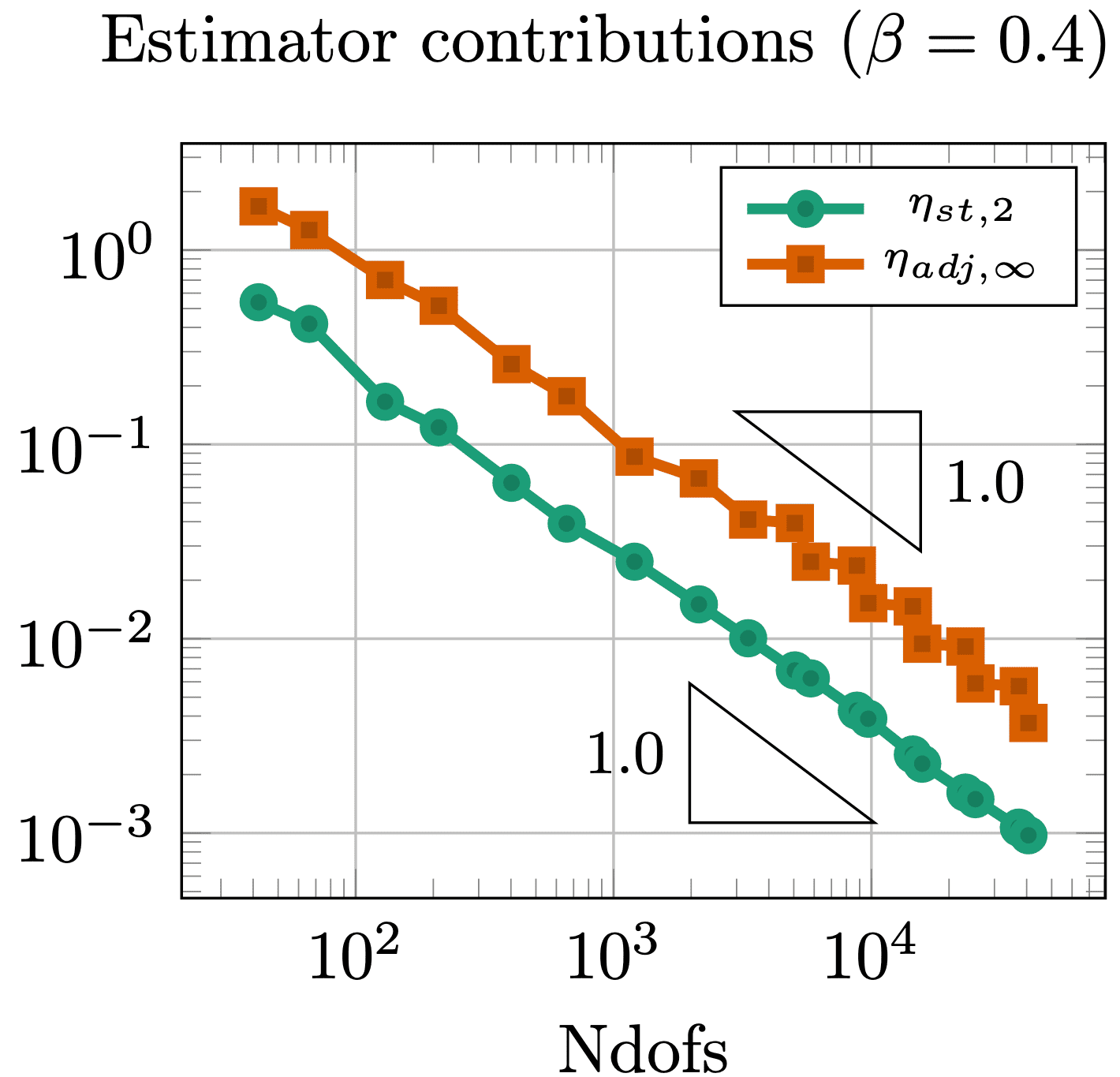}\\
\qquad
{\small{(8.B)}}
\end{minipage}
\\
\begin{minipage}[c]{0.45\textwidth}\centering
\includegraphics[trim={0 0 0 0},clip,width=5.5cm,height=5.3cm,scale=0.30]{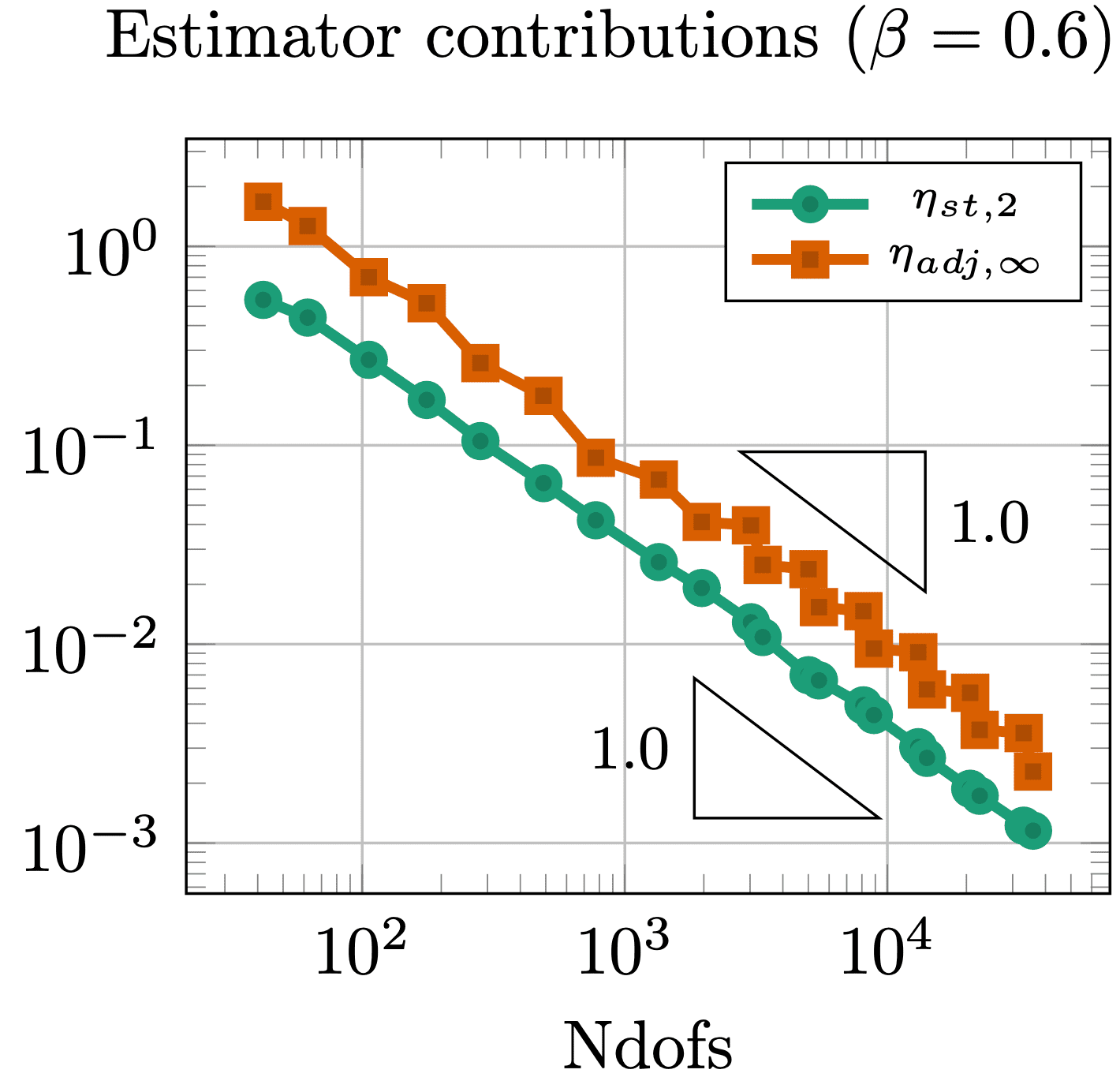}\\
\qquad
{\small{(8.C)}}
\end{minipage}
\begin{minipage}[c]{0.45\textwidth}\centering
\includegraphics[trim={0 0 0 0},clip,width=5.5cm,height=5.3cm,scale=0.30]{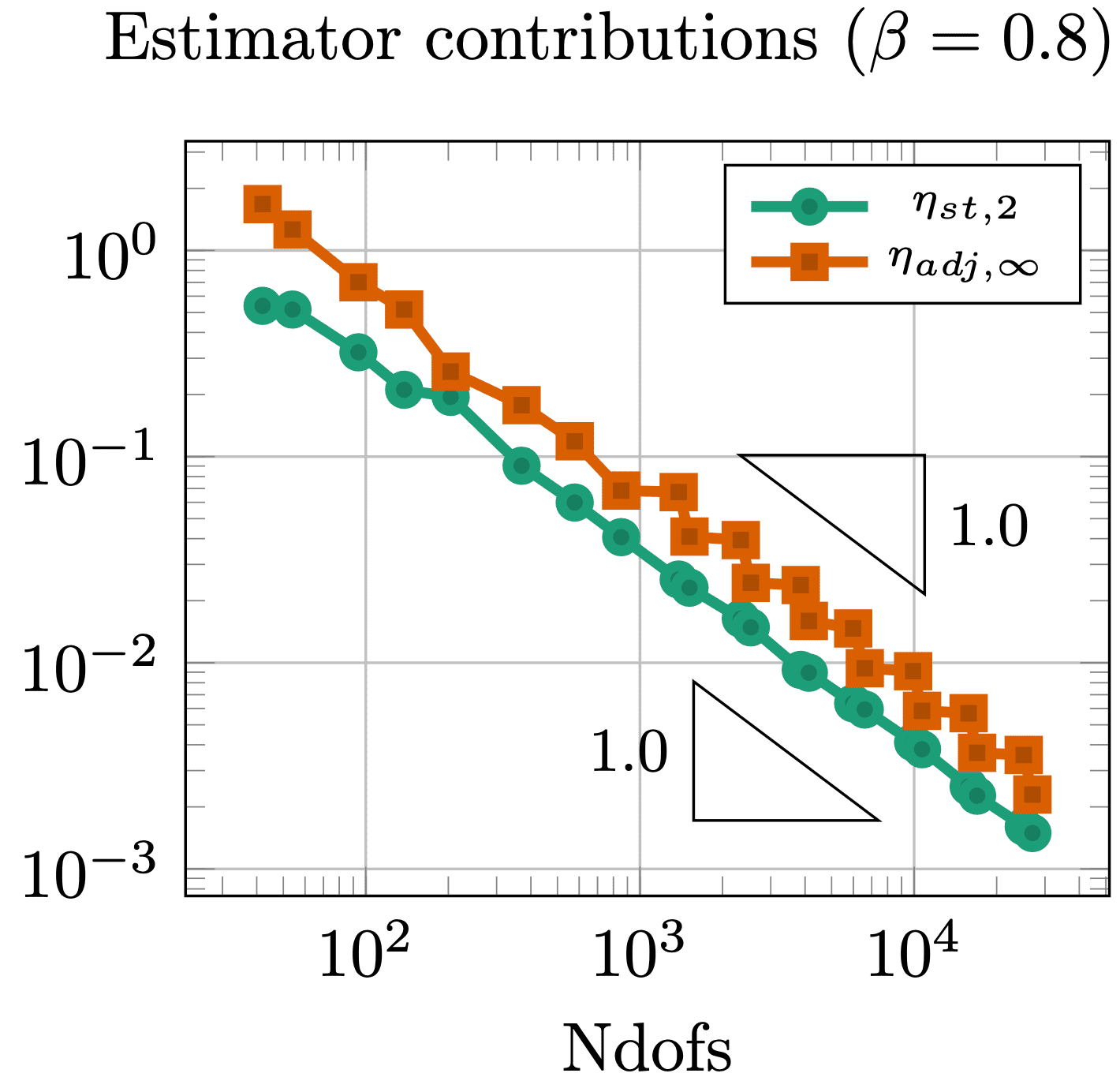}\\
\qquad
{\small{(8.D)}}
\end{minipage}
 \caption{Experimental rates of convergence for individual contributions of the estimator $E$ when $\beta=0.2$ (8.A), $\beta=0.4$ (8.B), $\beta = 0.6$ (8.C), and $\beta = 0.8$ (8.D) for the problem from section \ref{sec:ex_3}.}
\label{fig:ex_3_1}
\end{figure}


\bibliographystyle{siam}
\bibliography{biblio}

\begin{thebibliography}{10}

\bibitem{MR1885308}
{\sc M.~Ainsworth and J.~T. Oden}, {\em A posteriori error estimation in finite
  element analysis}, Pure and Applied Mathematics (New York),
  Wiley-Interscience [John Wiley \& Sons], New York, 2000.

\bibitem{MR3878607}
{\sc A.~Allendes, E.~Ot\'{a}rola, R.~Rankin, and A.~J. Salgado}, {\em An {\it a
  posteriori} error analysis for an optimal control problem with point
  sources}, ESAIM Math. Model. Numer. Anal., 52 (2018), pp.~1617--1650.

\bibitem{MR1780911}
{\sc R.~Becker, H.~Kapp, and R.~Rannacher}, {\em Adaptive finite element
  methods for optimal control of partial differential equations: basic
  concept}, SIAM J. Control Optim., 39 (2000), pp.~113--132.

\bibitem{MR2373954}
{\sc S.~C. Brenner and L.~R. Scott}, {\em The mathematical theory of finite
  element methods}, vol.~15 of Texts in Applied Mathematics, Springer, New
  York, third~ed., 2008.

\bibitem{MR3407259}
{\sc F.~Camacho and A.~Demlow}, {\em {$L_2$} and pointwise a posteriori error
  estimates for {FEM} for elliptic {PDE}s on surfaces}, IMA J. Numer. Anal., 35
  (2015), pp.~1199--1227.

\bibitem{MR2974742}
{\sc E.~Casas}, {\em Second order analysis for bang-bang control problems of
  {PDE}s}, SIAM J. Control Optim., 50 (2012), pp.~2355--2372.

\bibitem{CDJ2023}
{\sc E.~Casas, A.~Dom\'inguez~Corella, and N.~Jork}, {\em New assumptions for
  stability analysis in elliptic optimal control problems}, SIAM J. Control
  Optim., 61 (2023), pp.~1394--1414.

\bibitem{MR4298694}
{\sc E.~Casas and M.~Mateos}, {\em State error estimates for the numerical
  approximation of sparse distributed control problems in the absence of
  {T}ikhonov regularization}, Vietnam J. Math., 49 (2021), pp.~713--738.

\bibitem{MR3982675}
{\sc E.~Casas, M.~Mateos, and A.~R\"osch}, {\em Error estimates for semilinear
  parabolic control problems in the absence of {T}ikhonov term}, SIAM J.
  Control Optim., 57 (2019), pp.~2515--2540.

\bibitem{MR3706910}
{\sc E.~Casas, D.~Wachsmuth, and G.~Wachsmuth}, {\em Sufficient second-order
  conditions for bang-bang control problems}, SIAM J. Control Optim., 55
  (2017), pp.~3066--3090.

\bibitem{MR4525177}
{\sc A.~D. Corella, N.~Jork, and V.~Veliov}, {\em Stability in affine optimal
  control problems constrained by semilinear elliptic partial differential
  equations}, ESAIM Control Optim. Calc. Var., 28 (2022), pp.~Paper No. 79, 30.

\bibitem{MR1740762}
{\sc E.~Dari, R.~G. Dur\'an, and C.~Padra}, {\em Maximum norm error estimators
  for three-dimensional elliptic problems}, SIAM J. Numer. Anal., 37 (2000),
  pp.~683--700.

\bibitem{MR2891922}
{\sc K.~Deckelnick and M.~Hinze}, {\em A note on the approximation of elliptic
  control problems with bang-bang controls}, Comput. Optim. Appl., 51 (2012),
  pp.~931--939.

\bibitem{MR4648515}
{\sc A.~Demlow, S.~Franz, and N.~Kopteva}, {\em Maximum norm {\it a posteriori}
  error estimates for convection-diffusion problems}, IMA J. Numer. Anal., 43
  (2023), pp.~2562--2584.

\bibitem{MR3520007}
{\sc A.~Demlow and N.~Kopteva}, {\em Maximum-norm a posteriori error estimates
  for singularly perturbed elliptic reaction-diffusion problems}, Numer. Math.,
  133 (2016), pp.~707--742.

\bibitem{MR4744369}
{\sc A.~Dom\'inguez~Corella, N.~Jork, v.~S. Ne\v{c}asov\'a, and J.~S.~H.
  Simon}, {\em Stability analysis of the {N}avier-{S}tokes velocity tracking
  problem with bang-bang controls}, J. Optim. Theory Appl., 201 (2024),
  pp.~790--824.

\bibitem{MR2050138}
{\sc A.~Ern and J.-L. Guermond}, {\em Theory and practice of finite elements},
  vol.~159 of Applied Mathematical Sciences, Springer-Verlag, New York, 2004.

\bibitem{MR3621827}
{\sc W.~Gong and N.~Yan}, {\em Adaptive finite element method for elliptic
  optimal control problems: convergence and optimality}, Numer. Math., 135
  (2017), pp.~1121--1170.

\bibitem{MR2434065}
{\sc M.~Hinterm\"{u}ller, R.~H.~W. Hoppe, Y.~Iliash, and M.~Kieweg}, {\em An a
  posteriori error analysis of adaptive finite element methods for distributed
  elliptic control problems with control constraints}, ESAIM Control Optim.
  Calc. Var., 14 (2008), pp.~540--560.

\bibitem{MR2122182}
{\sc M.~Hinze}, {\em A variational discretization concept in control
  constrained optimization: the linear-quadratic case}, Comput. Optim. Appl.,
  30 (2005), pp.~45--61.

\bibitem{MR4791221}
{\sc N.~Jork}, {\em Finite element error analysis of affine optimal control
  problems}, ESAIM Control Optim. Calc. Var., 30 (2024), pp.~Paper No. 60, 25.

\bibitem{MR3212590}
{\sc K.~Kohls, A.~R\"{o}sch, and K.~G. Siebert}, {\em A posteriori error
  analysis of optimal control problems with control constraints}, SIAM J.
  Control Optim., 52 (2014), pp.~1832--1861.

\bibitem{MR1887737}
{\sc W.~Liu and N.~Yan}, {\em A posteriori error estimates for distributed
  convex optimal control problems}, vol.~15, 2001, pp.~285--309 (2002).
\newblock A posteriori error estimation and adaptive computational methods.

\bibitem{MR1270622}
{\sc R.~H. Nochetto}, {\em Pointwise a posteriori error estimates for elliptic
  problems on highly graded meshes}, Math. Comp., 64 (1995), pp.~1--22.

\bibitem{MR3810878}
{\sc F.~P\"orner and D.~Wachsmuth}, {\em Tikhonov regularization of optimal
  control problems governed by semi-linear partial differential equations},
  Math. Control Relat. Fields, 8 (2018), pp.~315--335.

\bibitem{MR2583281}
{\sc F.~Tr\"{o}ltzsch}, {\em Optimal control of partial differential
  equations}, vol.~112 of Graduate Studies in Mathematics, American
  Mathematical Society, Providence, RI, 2010.
\newblock Theory, methods and applications, Translated from the 2005 German
  original by J\"{u}rgen Sprekels.

\bibitem{MR1650051}
{\sc R.~Verf\"{u}rth}, {\em A posteriori error estimators for
  convection-diffusion equations}, Numer. Math., 80 (1998), pp.~641--663.

\bibitem{MR3059294}
{\sc R.~Verf\"{u}rth}, {\em A posteriori error estimation techniques for finite
  element methods}, Numerical Mathematics and Scientific Computation, Oxford
  University Press, Oxford, 2013.

\bibitem{MR3780469}
{\sc N.~von Daniels}, {\em Tikhonov regularization of control-constrained
  optimal control problems}, Comput. Optim. Appl., 70 (2018), pp.~295--320.

\bibitem{MR4076464}
{\sc N.~von Daniels and M.~Hinze}, {\em Variational discretization of a
  control-constrained parabolic bang-bang optimal control problem}, J. Comput.
  Math., 38 (2020), pp.~14--40.

\bibitem{MR3095657}
{\sc D.~Wachsmuth}, {\em Adaptive regularization and discretization of
  bang-bang optimal control problems}, Electron. Trans. Numer. Anal., 40
  (2013), pp.~249--267.

\bibitem{MR3385653}
\leavevmode\vrule height 2pt depth -1.6pt width 23pt, {\em Robust error
  estimates for regularization and discretization of bang-bang control
  problems}, Comput. Optim. Appl., 62 (2015), pp.~271--289.

\end{thebibliography}

\end{document}